\documentclass[11pt,reqno]{amsart}

\usepackage[utf8]{inputenc} 
\usepackage{bbm}

\usepackage{siunitx}
\usepackage{mathrsfs} 

\usepackage{graphicx} 

\usepackage[margin=3cm]{geometry}

\usepackage{verbatim}
\usepackage{amssymb}
\usepackage{amsthm}
\usepackage{amsmath}
\usepackage{amsfonts}
\usepackage{latexsym}
\usepackage{mathtools}
\usepackage{enumitem} 
\usepackage{cite}
\usepackage[english]{babel} 
\usepackage{dsfont}

\usepackage[dvipsnames]{xcolor} 
\usepackage[colorlinks=true, pdfstartview=FitV, linkcolor=blue, citecolor=blue, urlcolor=blue]{hyperref}

\usepackage{orcidlink}

\theoremstyle{plain}
\newtheorem{thm}{Theorem}
\newtheorem{remark}[thm]{Remark}
\newtheorem{prop}[thm]{Proposition}
\newtheorem{cor}[thm]{Corollary}
\newtheorem{lemma}[thm]{Lemma}

\numberwithin{equation}{section}
\numberwithin{thm}{section}

\ExplSyntaxOn

\NewDocumentCommand{\pFq}{O{}mmmmm}
{
	\group_begin:
	\keys_set:nn { hypergeometric } { #1 }
	\hypergeometric_print:nnnnn { #2 } { #3 } { #4 } { #5 } { #6 }
	\group_end:
}
\NewDocumentCommand{\hypergeometricsetup}{m}
{
	\keys_set:nn { hypergeometric } { #1 }
}

\tl_new:N \l_hypergeometric_divider_tl
\tl_new:N \l_hypergeometric_left_tl
\tl_new:N \l_hypergeometric_right_tl

\keys_define:nn { hypergeometric }
{
	symbol .tl_set:N = \l_hypergeometric_symbol_tl,
	symbol .initial:n = F,
	separator .tl_set:N = \l_hypergeometric_separator_tl,
	separator .initial:n = {},
	skip .tl_set:N = \l_hypergeometric_skip_tl,
	skip .initial:n = 8,
	divider .choice:,
	divider/semicolon .code:n = \tl_set:Nn \l_hypergeometric_divider_tl { \;; },
	divider/bar .code:n = \tl_set:Nn \l_hypergeometric_divider_tl { \;\middle|\; },
	divider .initial:n = semicolon,
	fences .choice:,
	fences/brack .code:n = 
	\tl_set:Nn \l_hypergeometric_left_tl {[}
	\tl_set:Nn \l_hypergeometric_right_tl {]},
	fences/parens .code:n = 
	\tl_set:Nn \l_hypergeometric_left_tl {(}
	\tl_set:Nn \l_hypergeometric_right_tl {)},
	fences .initial:n = brack,
}

\cs_new_protected:Nn \hypergeometric_print:nnnnn
{
	{} \sb {#1} \l_hypergeometric_symbol_tl \sb { #2 }
	\left\l_hypergeometric_left_tl
	\genfrac .. 
	{0pt} 
	{} 
	{ \__hypergeometric_process:n { #3 } } 
	{ \__hypergeometric_process:n { #4 } } 
	\l_hypergeometric_divider_tl
	#5
	\right\l_hypergeometric_right_tl
}

\cs_new_protected:Nn \__hypergeometric_process:n
{
	\clist_use:nn { #1 }
	{
		{\l_hypergeometric_separator_tl}
		\mspace { \l_hypergeometric_skip_tl mu }
	}
}

\ExplSyntaxOff

\newcommand{\laplaceSPH}{\Delta_{\mathbb{S}^2}}

\begin{document}


	\title[Nonlinear Stability of 3D Muskat Bubbles]{Nonlinear Stability of 3D Muskat Bubbles}

	\author[F. Gancedo]{Francisco Gancedo}
	\address{(FG) Departamento de An\'alisis Matem\'atico, Universidad de Sevilla, C/Tarfia s/n, Campus Reina Mercedes, 41012, Sevilla, Spain. \href{mailto:fgancedo@us.es}{fgancedo@us.es}}

	\author[E. Garc\'ia-Ju\'arez]{Eduardo Garc\'ia-Ju\'arez}
	\address{(EGJ) Departamento de An\'alisis Matem\'atico, Universidad de Sevilla, C/Tarfia s/n, Campus Reina Mercedes, 41012, Sevilla, Spain. \href{mailto:egarcia12@us.es}{egarcia12@us.es}}
	
	\author[S. Haziot]{Susanna V. Haziot}
	\address{(SVH) Department of Mathematics, Princeton University, Fine Hall, Princeton, NJ 08544, USA. \href{mailto:susanna.haziot@princeton.edu}{susanna.haziot@princeton.edu}}
	
	\author[N. Patel]{Neel Patel}
	\address{(NP) Department of Mathematics and Statistics, University of Maine,  Neville Hall, Orono, Maine, USA  04469. \href{mailto:neel.patel@maine.edu}{neel.patel@maine.edu}}
	
	\author[R. M. Strain]{Robert M. Strain}
	\address{(RMS) Department of Mathematics, University of Pennsylvania, David Rittenhouse Lab., 209 South 33rd St., Philadelphia, PA 19104, USA. \href{mailto:strain@math.upenn.edu}{strain@math.upenn.edu} (\orcidlink{0000-0002-1107-8570} \href{https://orcid.org/0000-0002-1107-8570}{https://orcid.org/0000-0002-1107-8570})}

	
	\begin{abstract}
This paper studies the three-dimensional Muskat problem for two immiscible fluids in a porous medium subject to gravity and surface tension. We consider the setting in which one fluid forms a bounded bubble surrounded by the other. For any value of the physical parameters, we prove the global-in-time nonlinear asymptotic stability of a translating spherical bubble for small initial perturbations in the critical Lipschitz class. This enlarges the class of critical spaces considered in previous two-dimensional bubble results. The proof combines spectral analysis of the linearized operator on $\mathbb{S}^2$ with multilinear singular integral estimates that exploit a crucial null structure in the nonlinear terms.

    \end{abstract}

	\setcounter{tocdepth}{2}
	
	\maketitle
	\tableofcontents

	\section{Introduction}\label{sec:intro}
	
	Fluid flow through a porous medium is a ubiquitous natural occurrence, with examples ranging from water passing through soil to the flow of blood through tissue. Darcy's law, derived experimentally  \cite{Darcy56} and mathematically \cite{Tartar89}, governs the velocity of a porous media flow, given here in three spatial dimensions, $X\in\mathbb{R}^{3}$, with time $t\geq0$,
	\begin{equation*}
	\frac{\mu}{\kappa}u(X,t)  = -\nabla p(X,t) - g \rho(X,t)e_3.
    \end{equation*}
We consider incompressible fluids:
$$\nabla\cdot u(X,t)=0.$$
In Darcy's law, $u$ is the fluid velocity, $\mu$ is the fluid viscosity, $\kappa$ is the permeability of the porous medium, $p$ is pressure, $g$ is the gravitational constant, and $\rho$ is the fluid density. In this paper, we consider constant permeability and gravity.
	
	When two immiscible fluids, such as water and oil, of different but constant  densities meet in a porous medium, we express the density function as
	\begin{equation*}
	\rho(X,t)=\left\{\begin{array}{rl}
	\rho^{\mathrm{in}},& X\in D^{\mathrm{in}}(t),\\
	\rho^{\mathrm{out}},& X\in D^{\mathrm{out}}(t).
	\end{array}\right.
	\end{equation*}
	Here, $D^{\mathrm{in}}(t)$ and $D^{\mathrm{out}}(t)$ represent the two disjoint  connected fluid domains at time $t$, whose common boundary $\Gamma(t)=\partial D^{\mathrm{in}}(t)=\partial D^{\mathrm{out}}(t)$ is the fluid interface, and together they fill the entire space. The constants  $\rho^{\mathrm{in}}$ and $\rho^{\mathrm{out}}$  represent the respective densities. 
    The interface is then transported by the normal velocity of the fluids:
    \begin{equation*}
        u^{\mathrm{in}}(X,t)\cdot \nu(X,t)=u^{\mathrm{out}}(X,t)\cdot \nu(X,t), \quad X(t)\in \Gamma(t),
    \end{equation*}
    where $\nu$ denotes a normal vector to $\Gamma(t)$.
    This regime is called the Muskat problem. The case when only one fluid is present is referred to as the one-phase Muskat problem.
    In this paper, we  consider the two-phase setting with both fluids of equal constant viscosity $\mu>0$, and $\rho^{\mathrm{in}}, \rho^{\mathrm{out}}>0$. The fluid is at rest at infinity, $$u(X,t)\to0 \text{ as } |X|\to\infty.$$ 
	Interest in these problems can be attributed to physical phenomena such as water encroaching on an oil-filled sand \cite{Muskat34}. The main mathematical interest is to study the dynamics of the sharp interface, $\partial D^{\mathrm{in}}(t) = \partial D^{\mathrm{out}}(t)$, between the two fluids. The  dynamics of the interface is driven by the jump in the fluid density due to gravity. The presence of surface tension can also be considered, which causes a pressure jump across the boundary via the Laplace--Young equation:
	\begin{equation*}
	p^{\mathrm{out}}(X,t) - p^{\mathrm{in}}(X,t) = -\sigma K(X), \quad X\in\Gamma(t).
	\end{equation*}
	Here, $\sigma \geq 0$ is twice the surface tension coefficient,  $p^{\mathrm{in}}$ and $p^{\mathrm{out}}$ are the limits of the pressure  from the interior of the two fluid domains, and $K$ is the mean curvature of the interface.

	The focus of this paper is the dynamics of a three-dimensional bubble, meaning the situation where $D^{\mathrm{in}}(t)$ is enclosed inside  $D^{\mathrm{out}}(t)$, and the boundary $\Gamma(t)$ is a closed surface. 
In this situation, the interface necessarily contains regions where the denser fluid lies above the less dense fluid. Thus, unlike an interface separating two stably stratified fluid layers, a bubble in a gravity field cannot satisfy the Rayleigh--Taylor sign condition everywhere.
To obtain well-posed dynamics despite the unstable gravity effects, surface tension is included, i.e., $\sigma>0$.

	\subsection{State of the art}
	
	We first discuss the well-studied infinite interface regime. Without surface tension, local and global well-posedness results in stable configurations have been established first in subcritical spaces and later in critical spaces as well as finite-time blow-up. See \cite{GancedoSEMA,GranerLazar} for two reviews and \cite{AlazardCritical2,GHHaziotGomezSerranoPausader23,DongGancedoNguyen3D,AlazardKoch} for recent developments.

	Surface tension introduces a higher order operator into the equation governing the interface dynamics, which we first discuss in the context of the two-phase Muskat problem. Here, finger-shaped equilibrium solutions exist but are unstable \cite{GuoHallstromSpirn07,EscherMatioc2011}. Local well-posedness is a classical result in this setting \cite{DuchonRobert1984,Chen93,ES97}, later obtained in the full range of subcritical regularity \cite{Nguyen20s,MatiocMatioc23}. The zero surface tension limit was considered \cite{Ambrose14} obtaining convergence to stable capillarity free solution of any subcritical regularity \cite{FlynnNguyen2021}. 
 Recently, Schauder estimates and local well-posedness in a log-critical regularity space were established \cite{ChenHuNguyen2026}. Global well-posedness and instant gain of regularity for small data was shown in \cite{ChenHuNguyenFluids}. On the other hand, solutions are globally well-posed for initial interfaces that are arbitrarily large in the Lipschitz norm but small in certain critical Besov spaces \cite{Lazar2024}. The one-phase regime has a mathematically different structure. There global well-posedness was shown for small data in any subcritical Sobolev space \cite{DongKwon2026}. There also exist self-similar solutions with a corner that instantaneously smoothens out \cite{AgrawalPatel2026}. In a different physical setting, it was shown that surface tension stabilizes the Rayleigh--Taylor instability experienced by a fluid layer on top of a dry region \cite{GancedoGraneroScro2020}. When the fluid makes contact angles with a fixed boundary, global-in-time estimates and decay are shown for solutions close to equilibrium \cite{BocchiCastroGancedo2026}.
	
	The literature on Muskat bubbles under the effect of gravity is more limited and notably is only done thus far in the two-dimensional setting. It was shown in \cite{CP93,YT11,YT12} that smoother small perturbations of a translating circular bubble were globally well-posed. In \cite{GancedoGarciaJuarezPatelStrain23}, circular bubbles that translate vertically with velocity proportional to the density jump were shown to be steady-state solutions. Moreover, medium-sized perturbations of such steady states were shown to be globally well-posed in a class of subcritical spaces. The solutions are instantly analytic and exponentially decay back toward a steady-state solution. Later, \cite{GGPSCriticalBubble} was able to achieve the global result in scaling critical spaces of Wiener type, requiring in particular uniform continuity of the tangent vector.  Without surface tension effects,  mixing solutions starting from Muskat bubble type interfaces were constructed in  \cite{CastroFaracoMengual22}.

	The present work gives the first global nonlinear asymptotic stability result for three-dimensional Muskat bubbles in the presence of gravity.  The three-dimensional setting introduces additional geometric and
analytical difficulties. The interface is a two-dimensional surface,
and its curvature and the singular integral operators governing its
motion must be analyzed on this evolving geometry. 
   We address the global-in-time dynamics and stability of  bubbles that are initially Lipschitz.     
When only surface tension effects are considered, the well-posedness and stability of nearly spherical bubbles was known in more regular spaces (see e.g. \cite{ES97,PrussSimonett2016}). As mentioned above, gravity produces an instability effect and moreover it mixes frequencies even at a linear level. Notably, we do not require smallness in the ratio of gravity versus surface tension effects.

  Related questions concerning closed interfaces in three-dimensional fluids have also been studied in other fluid models, mostly dealing with rigid solutions and subcritical or local-in-time wellposedness. 
  For example, in the Euler equations,  the authors in 
\cite{BaldiJulinLaManna2026} construct uniformly rotating capillary bubbles and \cite{MeyerNiebelSeis2025} shows the existence of steadily translating bubbles.
Recent results on time-dependent interfaces include the collapse of a vacuum bubble surrounded by a viscous incompressible fluid \cite{GigaGu2026}, local well-posedness for a nearly spherical capillary bubble under zero gravity  \cite{Shao2026}, and nonlinear stability under spherically symmetric perturbations in a model of a compressible bubble immersed in an incompressible fluid \cite{ChihWeinstein2023}.
In the mathematically related three-dimensional Peskin problem about immersed elastic membranes, global asymptotic stability at critical regularity has recently been established \cite{GarciaHaziotKuoMoriZhou2026}.

We develop a formulation suited to analysis on $\mathbb{S}^{2}$ and compute the
action of the linearized operator in the spherical harmonic basis, studying the smoothing properties of the associated semigroup. The presence of gravity makes the linear operator non diagonal. To obtain the result for any fixed value of the physical parameters, we cannot treat gravity as a perturbation of the diagonal surface tension term.   
Combining this spectral representation of the linear operator with Littlewood--Paley theory to estimate the nonlinear terms at critical regularity, we show the global-in-time nonlinear asymptotic stability by performing a fixed point in time-weighted spaces that capture both the gain of regularity for positive times and the exponential convergence to a translating sphere. As opposed to the nonlinearity in the Peskin problem \cite{GarciaHaziotKuoMoriZhou2026}, the critical-regularity estimates of the multilinear singular integral operators involved in the equations are not enough to close the fixed point. In the two-dimensional setting, the use of Wiener spaces hided this difficulty, since singular integral operators are well-behaved in those spaces.
Here, a null structure combined with the third-order operator arising from surface tension becomes crucial in the argument (see Lemma \ref{lem:total_output_derivative} and Remark \ref{lem:null}), allowing us to obtain the result in the larger and more natural class of Lipschitz initial interfaces.

	\subsection{Outline of the paper}
	
	In the remainder of this section, we formulate the contour equation for the three-dimensional Muskat bubble and the precise statement of the main result of our paper. In Section \ref{sec:analysisS2}, we develop the relevant theory for analysis on $\mathbb{S}^{2}$ dealing with spherical harmonics, Littlewood--Paley theory, and explicit integral calculations. Section \ref{sec:linearization} contains the derivation of the linear operator of the contour equation, followed by the derivation of the spectral expression of the linear operator  in Section \ref{sec:spectral}. Finally, in Section \ref{sec:semigroup}, the semigroup estimates for the linear operator are proven and, in Section \ref{sec:nonlinear}, the nonlinear remainder terms of the contour equation are appropriately bounded. The proof of the main theorem is then concluded in Section \ref{sec:proofmainthm}.

	\subsection{The contour equation}
	
	Let $S:\mathbb S^2\times[0,\infty)\longrightarrow \mathbb R^3$ describe the interface of the bubble, where $x\in\mathbb S^2$ denotes a point on the unit sphere. The contour equation is written as
	\begin{equation*}
	\partial_t S(x,t)\cdot \nu_S(x,t)
	=
	BR(S,\omega)(x,t)\cdot \nu_S(x,t),
	\qquad x\in\mathbb S^2,
	\end{equation*}
	where $\nu_S$ denotes a unit normal vector to the surface $S(\mathbb S^2,t)$.
	The Birkhoff--Rott integral is
	\begin{equation}\label{def:BR_intrinsic}
	BR(S,\omega)(x,t)
	=
	-\frac{1}{4\pi}\,\operatorname{p.v.}
	\int_{\mathbb S^2}
	\frac{S(x,t)-S(y,t)}
	{|S(x,t)-S(y,t)|^3}\wedge\omega(y,t)
	\,d\sigma(y).
	\end{equation}
	Here \(d\sigma(y)\) is the surface measure on the unit sphere \(\mathbb S^2\) and $\omega$ is the vorticity amplitude given by Darcy's law as
	\begin{equation}\label{vorticity}
	\omega(y,t)=
	D_{\mathbb S^2}S(y,t)
	[
	y\wedge\nabla_{\mathbb S^2}\Omega(y,t)]. 
	\end{equation}
	See \cite{CCG13} for cartesian coordinates. The operator \(\nabla_{\mathbb S^2}\) is the gradient on \(\mathbb S^2\), and \(D_{\mathbb S^2}S(y,t)\) is the tangential differential of the map \(S\). The function $\Omega$ is defined below in \eqref{def:Omega_intrinsic}. Indeed, in spherical coordinates one has
	\[
	\omega\sin(\theta)
	=
	\partial_\theta\Omega\,\partial_\varphi S
	-
	\partial_\varphi\Omega\,\partial_\theta S,
	\]
	so that
	\[
	\omega(\theta,\varphi,t)\sin(\theta)d\theta\,d\varphi
	=
	D_{\mathbb S^2}S(x,t)
	\left[
	x\wedge\nabla_{\mathbb S^2}\Omega(x,t)
	\right]
	\,d\sigma(x).
	\] The function \(\Omega\) is given by
	\begin{equation}\label{def:Omega_intrinsic}
	\Omega(x,t)
	=
	A_\sigma K[S](x,t)
	-
	A_\rho S(x,t)\cdot e_3,
	\end{equation}
	where \(K[S]\) denotes the mean curvature of the surface parametrized by \(S\), 
	$e_3$ is the canonical vector in the vertical direction, and $A_\sigma, A_\rho$ are physical constants measuring the surface tension and gravity effects:
	\begin{equation*}
	A_\sigma=\frac{\kappa \sigma}{\mu},\qquad A_\rho=\frac{g \kappa (\rho^{\text{out}}-\rho^{\text{in}})}{\mu}.
	\end{equation*}
	It will be sometimes convenient to separate each contribution,
	\begin{equation}\label{Omega_split}
	\begin{aligned}
	\Omega(x,t)
	= \Omega_\sigma(x,t) + \Omega_\rho(x,t),\quad  
	\Omega_\sigma =A_\sigma K[S](x,t),\quad
	\Omega_\rho=-
	A_\rho S(x,t)\cdot e_3,
	\end{aligned}
	\end{equation}
	and we denote accordingly 
	\begin{equation*}
	\begin{aligned}
	\omega(x,t)
	&= \omega_\sigma(x,t) + \omega_\rho(x,t),\\
	BR&=BR_\sigma+BR_\rho,\quad BR_\sigma=BR(S,\omega_\sigma),\quad BR_\rho=BR(S,\omega_\rho). 
	\end{aligned}
	\end{equation*}
	In this paper, we consider a translating bubble interface of the form
	\begin{equation}\label{def:interface}
	S(x,t)
	=
	R\bigl(1+3f(x,t)\bigr)^{1/3}x+c(t),
	\qquad x\in\mathbb S^2,
	\end{equation}
	where \(c(t)\in\mathbb R^3\) is given by
	\begin{equation}\label{ct_def}
	\begin{aligned}
	c(t) = \frac23A_\rho t e_3+\frac{3A_\sigma}{4\pi R^2}\int_0^t\int_{\mathbb{S}^2}f(y,s)yd\sigma(y)ds.
	\end{aligned}
	\end{equation}
	The role of the second term in $c(t)$ will become clear when we study the linear operator, see Section \ref{sec:linearization}. 
	Let us denote
	\begin{equation}\label{rhodef}
	r(x,t)=\bigl(1+3f(x,t)\bigr)^{1/3}.    
	\end{equation}
	The reason why we choose this particular expression for the radial deviation is that the volume conservation due to incompressibility translates into preservation of the mean of $f$. This is shown in Subsection \ref{VolumeCentroid} below.
	The exterior normal vector is
	\begin{equation}\label{def:nonunit_normal_f}
	\mathcal N_f(x,t)
	=
	(1+3f(x,t))x-\nabla_{\mathbb S^2}f(x,t).
	\end{equation}
	Using this non-unit normal, the contour equation may be written  as
	\begin{equation*}
	\partial_t S(x,t)\cdot \mathcal N_f(x,t)
	=
	BR(S,\omega)(x,t)\cdot \mathcal N_f(x,t),
	\end{equation*}
	and substituting \eqref{def:interface},
	\begin{equation}\label{contour_eq}
	\partial_t f(x,t)
	=
	\frac{1}{Rr(x,t)}
	\bigl(BR(S,\omega)(x,t)-\dot c(t)\bigr)
	\cdot
	\mathcal N_f(x,t).    
	\end{equation}

	\subsubsection{Volume and centroid}\label{VolumeCentroid}

	The volume enclosed by the interface is
	\begin{equation*}
	\operatorname{Vol}
	=
	\int_{\mathbb S^2}
	\int_0^{Rr(x,t)}
	\eta^2\,d\eta\,d\sigma(x)
	=
	\frac{R^3}{3}
	\int_{\mathbb S^2}
	r(x,t)^3\,d\sigma(x)
	=
	\frac{4\pi}{3}R^3
	+
	R^3\int_{\mathbb S^2}f(x,t)\,d\sigma(x).
	\end{equation*}
	Thus, having as initial volume the value $\frac{4\pi}{3}R^3$, the incompressibility condition implies that
	\begin{equation}\label{zeroVOL_intrinsic}
	\int_{\mathbb S^2}f(x,t)\,d\sigma(x)=0,\quad \forall t\geq0.
	\end{equation}
	The first moment of the enclosed region is
	\begin{equation*}
	\begin{aligned}
	Q_S
	&=
	\int_{\mathbb S^2}
	\int_0^{Rr(x,t)}
	\bigl(\eta x+c(t)\bigr)\eta^2\,d\eta\,d\sigma(x)
	\\
	&=
	\frac{R^4}{4}
	\int_{\mathbb S^2}
	r(x,t)^4 x\,d\sigma(x)
	+
	\frac{R^3}{3}c(t)
	\int_{\mathbb S^2}
	r(x,t)^3\,d\sigma(x).
	\end{aligned}
	\end{equation*}
	If \eqref{zeroVOL_intrinsic} holds, then the centroid is
	\begin{equation*}
	C_S
	=
	\frac{Q_S}{\operatorname{Vol}}
	=
	c(t)
	+
	\frac{3R}{16\pi}
	\int_{\mathbb S^2}
	\bigl(1+3f(x,t)\bigr)^{4/3}x\,d\sigma(x).
	\end{equation*}
	In particular, if \(f\) is constant, the last integral vanishes and $C_S=c(t)$.
	
	\subsubsection{Main Theorem}

    Let $\alpha\in(0,1)$ and $\gamma>0$.
	We consider functions $f:\mathbb{S}^2\times(0,\infty)\to\mathbb{R}$ that are strongly measurable in time with values in $L^2(\mathbb{S}^2)$, and define the norm
	\begin{equation*}
	\begin{aligned}
	\|f\|_{X_\gamma}=\sup_{t>0} e^{\gamma t}\|f(\cdot,t)\|_{W^{1,\infty}}+\sup_{t>0} t^{\frac{2+\alpha}{3}}e^{\gamma t}\|f(\cdot,t)\|_{C^{3,\alpha}}.
	\end{aligned}
	\end{equation*}
    Here, $W^{1,\infty}(\mathbb{S}^2)$ denotes the space of Lipschitz functions on the sphere, with norm 
    	  \begin{equation*}
        \|f\|_{W^{1,\infty}}=\|f\|_{L^\infty}+\|\nabla_{\mathbb{S}^2}f\|_{L^\infty},
    \end{equation*}
and $C^{3,\alpha}(\mathbb{S}^2)$ denotes the usual H\"older space, with the norm defined precisely in  \eqref{Holder_norm_def}.
The associated Banach space is 
	\begin{equation*}
	X_\gamma=\Big\{f:\mathbb{S}^2\times (0,\infty)\to\mathbb{R}:\,\|f\|_{X_\gamma}<\infty,\text{ with } \int_{\mathbb{S}^2}f(x,t)d\sigma(x)=0 \text{ for all }t>0\Big\}.
	\end{equation*}

    \begin{thm}\label{thm:main}
		Let $\alpha\in(0,1)$, $R>0$, $A_\sigma>0$, and $A_\rho\in\mathbb{R}$. Let $f_0\in W^{1,\infty}(\mathbb{S}^2)$ be such that $\int_{\mathbb{S}^2}f_0d\sigma=0$. There exist $\gamma=\gamma\big(A_\sigma/R^3\big)>0$ and $\varepsilon=\varepsilon\big(|A_\rho|R^2/A_\sigma\big)>0$ such that, if $\|f_0\|_{W^{1,\infty}}<\varepsilon$, then there exists a unique, global-in-time  solution $f\in X_\gamma$ to the equation \eqref{contour_eq} with initial datum $f_0$  that satisfies $\|f\|_{X_\gamma}\lesssim \varepsilon$. 
	\end{thm}
	
	\begin{remark}
	    As shown in Section \ref{sec:proofmainthm}, we will construct first a mild solution to \eqref{contour_eq}. Then, standard parabolic bootstrapping shows that the solution is actually classical for $t>0$. 
	\end{remark}	

\begin{remark}
	     We use the standard notion of
         mild solution for parabolic equations: after showing that the linearization around $f=0$ (i.e., the sphere) gives rise to a parabolic semigroup, we will construct the solution as a fixed point in the space $X_\gamma$ of the corresponding Duhamel formulation; see Proposition \ref{prop:fixed_point}. 
    In particular,  $f\in X_\gamma$ implies exponential convergence of $f$ to zero, which in turn implies exponential convergence of the bubble to a translating sphere. In fact,
\begin{equation*}
    S(x,t) - \frac{2}{3}A_\rho te_3\to Rx+\frac{3A_\sigma}{4\pi R^2}\int_0^\infty\int_{\mathbb{S}^2}f(y,\tau)yd\sigma(y)d\tau,
\end{equation*}
exponentially fast in $W^{1,\infty}(\mathbb{S}^2)$.
\end{remark}	
   
\begin{remark}
    For general initial data  in $W^{1,\infty}$,  strong continuity at $t=0$ in this space cannot be expected: parabolic smoothing would require $f_0$ to have a $C^1(\mathbb{S}^2)$ representative. The initial datum is  attained in the sense that $f(t)\to f_0$ in $C^{\beta}$ for any $\beta\in(0,1)$, and $f(t) \overset{*}{\rightharpoonup} f_0$ in $W^{1,\infty}(\mathbb S^2)$ as $t\rightarrow0^+$.
    \end{remark}

	\section{Analysis on $\mathbb{S}^{2}$}\label{sec:analysisS2}
	
	In this section, we provide an overview of the analysis on $\mathbb{S}^{2}$ that is needed for this paper. 
	
	\subsection{Spherical harmonics}

	The spherical harmonics $\{Y_{\ell,m}\}_{\ell\geq0, |m|\leq \ell}$ are the eigenfunctions of the Laplace--Beltrami operator on the sphere $-\laplaceSPH$,
	\begin{equation}\label{LaplaceSPH}
	-\laplaceSPH Y_{\ell,m}(x)
	=
	\ell(\ell+1)
	Y_{\ell,m}(x).
	\end{equation}
	We call the eigenspace associated with eigenvalue $\ell(\ell+1)$ the \textit{degree $\ell$ space} of spherical harmonics.  For  $f:\mathbb{S}^{2} \to \mathbb{C}$ in $L^2(\mathbb{S}^2)$, we have the spherical harmonics expansion
	\begin{equation*}
	f(x) = \sum_{\ell=0}^\infty \sum_{m=-\ell}^{\ell} \hat{f}_{\ell,m} Y_{\ell,m}(x), 
	\end{equation*}
	where the spherical Fourier transform on $\mathbb{S}^{2}$ is defined by 
	\begin{equation*}
	\mathcal{F}_{\mathbb{S}^{2}}(f)({\ell,m}) =   \hat{f}_{\ell,m} = \int_{\mathbb{S}^{2}} f(y) \overline{Y_{\ell, m}}(y) d\sigma(y).
	\end{equation*}
	We use the orthonormal family of spherical harmonics given by
		\begin{equation}\label{Ortho_Sphe}
		Y_{\ell,m}(\theta,\varphi)
		=
		N_{\ell,m}P_{\ell}^{m}(\cos\theta)e^{im\varphi}, \quad N_{\ell,m}
		=
		(-1)^m
		\Big(
		\frac{2\ell+1}{4\pi}
		\frac{(\ell-m)!}{(\ell+m)!}
\Big)^{1/2},
		\end{equation}
		where $P_\ell^m$ are the associated Legendre functions. Parseval's identity then reads as follows
	\begin{equation*}
	\| f \|_{L^2(\mathbb{S}^2)}^2 = 
	\sum_{\ell=0}^\infty \sum_{m=-\ell}^{\ell} |\hat{f}_{\ell,m}|^2.
	\end{equation*}

    For simplicity of notation, throughout the paper we adopt the standard  convention that empty sums and empty products equal zero and one, respectively. Similarly, spherical harmonics and their coefficients are extended by zero outside the admissible range $\ell\geq0$, $|m|\leq \ell$.

	The following two lemmas will be repeatedly used to obtain the linearized equation and the action of the linear operator on spherical harmonics (see Sections \ref{sec:linearization} and \ref{sec:spectral}).

	\begin{lemma}
		Let $\{Y_{\ell,m}\}_{\ell\geq0,|m|\leq \ell}$ be given by \eqref{Ortho_Sphe}. Then
		\begin{equation}\label{x3Y_recurrence}
		x_3Y_{\ell,m}(x)
		=
		b_{\ell,m}^{+}Y_{\ell+1,m}(x)
		+
		b_{\ell,m}^{-}Y_{\ell-1,m}(x),
		\end{equation}
		where
		\begin{equation}\label{a_coefficients_pm}
		b_{\ell,m}^{+}
		=
		\left(\frac{(\ell+1)^2-m^2}{(2\ell+1)(2\ell+3)}\right)^{1/2},
		\qquad
		b_{\ell,m}^{-}
		=
		\left(\frac{\ell^2-m^2}{(2\ell-1)(2\ell+1)}\right)^{1/2}.
		\end{equation}
		Here, we use the convention that $b_{\ell,m}^{-}=0$ whenever the lower mode is not
		admissible.
	\end{lemma}
	
	\begin{proof}
		Write a point of \(\mathbb S^2\) in spherical coordinates as
		$x=(\sin\theta\cos\varphi,\sin\theta\sin\varphi,\cos\theta)$, so that $x_3=\cos\theta$.
		The associated Legendre functions satisfy (see e.g. \cite[Section~12.5]{ArfkenWeber2005}) 
		\begin{equation}\label{associated_legendre_recurrence}
		(2\ell+1)sP_{\ell}^{m}(s)
		=
		(\ell-m+1)P_{\ell+1}^{m}(s)
		+
		(\ell+m)P_{\ell-1}^{m}(s),
		\qquad s\in[-1,1].
		\end{equation}
		Taking \(s=\cos\theta\), multiplying \eqref{associated_legendre_recurrence}
		by \(N_{\ell,m}e^{im\varphi}\), and dividing by \(2\ell+1\), we obtain
		\begin{equation*}
		\begin{aligned}
		\cos\theta\,Y_{\ell,m}
		&=
		\frac{\ell-m+1}{2\ell+1}
		\frac{N_{\ell,m}}{N_{\ell+1,m}}
		Y_{\ell+1,m}
		+
		\frac{\ell+m}{2\ell+1}
		\frac{N_{\ell,m}}{N_{\ell-1,m}}
		Y_{\ell-1,m}.
		\end{aligned}
		\end{equation*}
		We now compute the two normalization ratios. First,
		$$
		\frac{N_{\ell,m}}{N_{\ell+1,m}}
		=
		\left(
		\frac{2\ell+1}{2\ell+3}
		\frac{\ell+m+1}{\ell-m+1}
		\right)^{1/2}.
		$$
		Hence
		$$
		\frac{\ell-m+1}{2\ell+1}
		\frac{N_{\ell,m}}{N_{\ell+1,m}}
		=
		\left(
		\frac{(\ell-m+1)(\ell+m+1)}
		{(2\ell+1)(2\ell+3)}
		\right)^{1/2}
		=
		\left(
		\frac{(\ell+1)^2-m^2}
		{(2\ell+1)(2\ell+3)}
		\right)^{1/2}.
		$$
		Similarly,
		$$
		\frac{N_{\ell,m}}{N_{\ell-1,m}}
		=
		\left(
		\frac{2\ell+1}{2\ell-1}
		\frac{\ell-m}{\ell+m}
		\right)^{1/2},
		$$
		and therefore
		$$
		\frac{\ell+m}{2\ell+1}
		\frac{N_{\ell,m}}{N_{\ell-1,m}}
		=
		\left(
		\frac{(\ell+m)(\ell-m)}
		{(2\ell-1)(2\ell+1)}
		\right)^{1/2}
		=
		\left(
		\frac{\ell^2-m^2}
		{(2\ell-1)(2\ell+1)}
		\right)^{1/2}.
		$$
		This proves \eqref{x3Y_recurrence} for \(m\geq 0\).
		For negative \(m\), we use the standard symmetry
		$$
		Y_{\ell,-m}=(-1)^m\overline{Y_{\ell,m}},
		\qquad m>0.
		$$
		Since $x_3$ is real and the coefficients above depend only on $m^2$,
		taking complex conjugates of the already proved identity gives the same
		formula for $Y_{\ell,-m}$. 
		
	\end{proof}

	We will denote the Newtonian operator by $\mathcal{G}$,
	\begin{equation}\label{newtonian_operator}
	\mathcal G[h](x)
	=
	\frac{1}{4\pi}\int_{\mathbb S^2}
	\frac{h(y)}{|x-y|}
	\,d\sigma(y).
	\end{equation}

	\begin{lemma}\label{lem:newtonian}
		Let $\{Y_{\ell,m}\}$ be an orthonormal basis of spherical harmonics in
		$L^2(\mathbb S^2)$ and $\mathcal{G}$ the Newtonian operator \eqref{newtonian_operator}.
		Then
		\begin{equation}\label{newtonian_spectral}
		\mathcal G[Y_{\ell,m}](x)
		=
		\frac{1}{2\ell+1}Y_{\ell,m}(x).
		\end{equation}
	\end{lemma}
	
	\begin{proof}
		Fix $\ell\geq 0$ and $|m|\leq \ell$, and set $Y=Y_{\ell,m}$. Consider the
		single-layer potential
		\begin{equation*}
		u(X)
		=
		\frac{1}{4\pi}\int_{\mathbb S^2}
		\frac{1}{ |X-y|}
		Y(y)\,d\sigma(y),
		\qquad X\in \mathbb R^3.
		\end{equation*}
		The function $u$ is harmonic in the two
		regions $|X|<1$ and $|X|>1$. Moreover, the single-layer potential is
		continuous across $\mathbb S^2$, and its normal derivative satisfies the jump
		relation
		\[
		\partial_r u^+ - \partial_r u^-
		=
		-Y
		\qquad\text{on } \mathbb S^2,
		\]
		where $+$ denotes the exterior trace, $-$ the interior trace, and $\partial_r$
		is the outward normal derivative relative to the unit ball.
		We now construct explicitly a function with the same 
		properties. Write $X=r\omega$, with $r=|X|$ and $\omega\in \mathbb S^2$, and
		define
		\[
		v(X)
		=
		\begin{cases}
		c\, r^\ell Y(\omega),
		& 0<r<1, \\[4pt]
		c\, r^{-\ell-1}Y(\omega),
		& r>1,
		\end{cases}
		\]
		where $c$ is a constant to be determined. Both $r^\ell Y(\omega)$ and $r^{-\ell-1}Y(\omega)$ are harmonic away from the
		origin. Hence $v$ is harmonic in $|X|<1$ and $|X|>1$. It is also continuous
		across $r=1$, because both traces are equal to $cY$.
		The radial derivatives at $r=1$ are
		$$
		\partial_r v^-
		=
		c\ell Y,
		\qquad
		\partial_r v^+
		=
		-c(\ell+1)Y.
		$$
		Therefore
		$$
		\partial_r v^+ - \partial_r v^-
		=
		-c(2\ell+1)Y.
		$$
		Choosing
		$c=1/(2\ell+1)$, we obtain
		$\partial_r v^+ - \partial_r v^-=-Y$.
		Thus $u$ and $v$ are both continuous across $\mathbb S^2$, harmonic away from
		$\mathbb S^2$, decay at infinity, and have the same jump in normal derivative
		across $\mathbb S^2$. Consequently, $w=u-v$ is harmonic in the sense of
		distributions on all of $\mathbb R^3$. Moreover, $w$ is bounded near the
		origin and tends to zero at infinity. By Liouville's theorem, we have $w\equiv 0$. Hence $u=v$.
		
		Finally, taking the trace on $\mathbb S^2$, we obtain the result.
		
	\end{proof}

	\subsection{Littlewood--Paley on $\mathbb{S}^{2}$}\label{subsec:LP}

	Let $\chi : [0,\infty)\rightarrow [0,1]$ be a smooth, non-increasing compactly supported function such that
	\begin{equation*}
	\begin{aligned}
	\chi(s) &= 0 \quad \text{if} \quad s\notin \left[0,1\right]\\
	\chi(s) &= 1 \quad \text{if} \quad s\in [0,2/3].
	\end{aligned}
	\end{equation*}
	Then, define $\phi_{0}(s) = \chi(s)$ and for $k\geq 1$,
	\begin{equation*}
	\phi_{k}(s) = \chi(2^{-k}s) - \chi(2^{-k+1}s).
	\end{equation*}
	We have
	\begin{equation*}
	\sum_{k=0}^{\infty} \phi_{k}(s) = 1, \quad \forall s\geq 0
	\end{equation*}
	and
	\begin{equation*}
	\text{supp}\phi_{j} \cap \text{supp}\phi_{k} = \emptyset \quad \text{if   } |j-k|\geq 2.
	\end{equation*}
	Define further the low-frequency cutoff $\phi_{\leq k}(s)=\phi_0(2^{-k}s)=\phi_0(s)+\sum_{m=1}^k\phi_{m}(s)$.
	For an integrable function $f:\mathbb{S}^{2} \to \mathbb{C}$, we will then define the Littlewood--Paley operators $\Delta_{j}$ for $j\geq 0$ as 
	\begin{equation}\label{LPD}
	\Delta_{j}f = \sum_{n\geq 0} \phi_j(n)\text{proj}_n f,\quad\mbox{where}\quad
	\text{proj}_n f = \sum_{|m|\leq n} \hat{f}_{n,m}Y_{n,m},
	\end{equation}
	with analogous definition for $\Delta_{\leq j}f$. We define $\Delta_j =0=\Delta_{\leq j}$ for $j<0$.
	These operators can be written as convolution operators on the sphere by using the classical addition formula for the spherical harmonics. We include the proof here for completeness.

	\begin{lemma}[Addition formula for spherical harmonics]\label{lemma:addition_formula}
		Let $\{Y_{\ell,m}\}$ be an orthonormal basis of
		spherical harmonics in $L^2(\mathbb S^2)$. Then, for every
		$\ell\geq 0$ and every $x,y\in \mathbb S^2$,
		\begin{equation}\label{Addition_formula}
		\sum_{|m|\leq \ell}Y_{\ell,m}(x)\overline{Y_{\ell,m}}(y)
		=
		\frac{2\ell+1}{4\pi}P_\ell(x\cdot y),
		\end{equation}
		where $P_\ell$ denotes the Legendre polynomial of order $\ell$, normalized by
		$P_\ell(1)=1$.
	\end{lemma}
	
	\begin{proof}
		Define
		$$
		K_\ell(x,y)=
		\sum_{|m|\leq \ell}Y_{\ell,m}(x)\overline{Y_{\ell,m}}(y).
		$$
		This is the integral kernel of the orthogonal projection onto the eigenspace
		of $-\laplaceSPH$ associated with the eigenvalue $\ell(\ell+1)$. Indeed, if
		$f\in L^2(\mathbb S^2)$, then
		$$
		\operatorname{proj}_\ell f(x)
		=
		\sum_{|m|\leq \ell}\widehat f_{\ell,m}Y_{\ell,m}(x)
		=
		\int_{\mathbb S^2}K_\ell(x,y)f(y)\,d\sigma(y).
		$$
		We first show that $K_\ell(x,y)$ depends only on $x\cdot y$. Let
		$Q\in SO(3)$. Since the Laplace--Beltrami operator commutes with rotations,
		the functions
		\[
		Y_{\ell,m}\circ Q,\qquad |m|\leq \ell,
		\]
		form another orthonormal basis of the same eigenspace. Therefore the
		projection kernel is independent of the chosen orthonormal basis, and hence
		\[
		K_\ell(Qx,Qy)=K_\ell(x,y).
		\]
		Thus, for fixed $y$, the function $x\mapsto K_\ell(x,y)$ is invariant under
		rotations preserving $y$. Consequently there exists a function $F_\ell$ such that
		\[
		K_\ell(x,y)=F_\ell(x\cdot y).
		\]
		Moreover, since $x\mapsto K_\ell(x,y)$ belongs to the degree-$\ell$ eigenspace, the function $t\mapsto F_\ell(t)$ satisfies the Legendre equation
		\begin{equation}\label{eq:Legendre}
		(1-t^2)F_\ell''(t)-2tF_\ell'(t)+\ell(\ell+1)F_\ell(t)=0.
		\end{equation}
		The regular solutions to \eqref{eq:Legendre} on $[-1,1]$ are multiples of $P_\ell$. Hence
		\[
		K_\ell(x,y)=c_\ell P_\ell(x\cdot y)
		\]
		for some constant $c_\ell$.
		It remains to determine $c_\ell$. Taking $x=y$ and integrating over
		$\mathbb S^2$, we obtain
		\[
		\int_{\mathbb S^2} K_\ell(x,x)\,d\sigma(x)
		=
		\sum_{|m|\leq \ell}\int_{\mathbb S^2}|Y_{\ell,m}(x)|^2\,d\sigma(x)
		=
		2\ell+1.
		\]
		On the other hand, since $P_\ell(1)=1$,
		\[
		\int_{\mathbb S^2} K_\ell(x,x)\,d\sigma(x)
		=
		\int_{\mathbb S^2} c_\ell P_\ell(1)\,d\sigma(x)
		=
		4\pi c_\ell.
		\]
		Therefore $c_\ell=(2\ell+1)/(4\pi)$, which proves
		\eqref{Addition_formula}.
	\end{proof}
	
	The projectors \eqref{LPD} can be then written as follows
	\begin{equation*}
	\text{proj}_n f =  \int_{\mathbb{S}^{2}}K_n(x\cdot y)f(y)d\sigma(y),
	\end{equation*}
	\begin{equation}\label{DeltajfIntegral}
	\Delta_j f(x)=\int_{\mathbb{S}^{2}}\sum_{n\geq 0}\phi_j(n)K_n(x\cdot y)f(y)d\sigma(y)=
	\int_{\mathbb{S}^{2}}\tilde{K}_j(x\cdot y)f(y)d\sigma(y),
	\end{equation}
	with
	\begin{equation}\label{KjandtildeKj}
	K_n(x\cdot y)=\sum_{|m|\leq n}Y_{n,m}(x)\overline{Y_{n,m}}(y),\quad \mbox{and}\quad \tilde{K}_j(x\cdot y)=\sum_{n\geq 0}\phi_j(n)K_n(x\cdot y).
	\end{equation}
	We analogously define $\tilde{K}_{\leq j}$ using $\phi_{\leq j}$ in place of $\phi_j$,
    \begin{equation}\label{tildeKj}
        \tilde{K}_{\leq j}(x\cdot y)=\sum_{n\geq 0}\phi_{\leq j}(n)K_n(x\cdot y).
    \end{equation}
	More generally, for $\beta\in\mathbb{R}$, we consider Fourier multipliers $\lambda_n=\lambda_\beta(n)$ satisfying, for some positive constants $c<C$, 
	\begin{equation}\label{multiplier}
	\begin{aligned}
	c (1+n)^\beta \leq \lambda_n\leq C (1+n)^\beta,\qquad |\frac{d^m}{ds^m} \lambda_\beta(s)|\lesssim (1+s)^{\beta-m}.
	\end{aligned}
	\end{equation}
	Then, for $t\geq 0$, $\beta>0$, $r\in\mathbb{R}$, $\lambda_n=\lambda_\beta(n)$ and $\gamma_n=\gamma_r(n)$ satisfying \eqref{multiplier}, we consider the kernels
	\begin{equation}\label{Gj_kernel}
	G_j(x\cdot y, t)=\sum_{n\geq 0}\phi_j(n)\gamma_r(n)e^{-\lambda_n t} K_{n}(x\cdot y)=\frac{1}{4\pi}\sum_{n\geq 0}\phi_j(n)\gamma_r(n)e^{-\lambda_n t}(2n+1)P_n(x\cdot y),
	\end{equation}
    and
    	\begin{equation}\label{Glqj_kernel}
	G_{\leq j}(x\cdot y, t)=\sum_{n\geq 0}\phi_{\leq j}(n)e^{-\lambda_n t} K_{n}(x\cdot y).
	\end{equation}
    \begin{lemma}\label{lem:kernel_bounds}
		Let $\beta>0$, $r\in\mathbb{R}$, and let $\lambda_n=\lambda_\beta(n)$ and $\gamma_n=\gamma_r(n)$ be multipliers satisfying \eqref{multiplier}. 
		Let $\ell$ be a nonnegative integer and $\tilde{K}_j$, $G_j$, $\tilde{K}_{\leq j}$, $G_{\leq j}$ be defined in \eqref{KjandtildeKj}, \eqref{Gj_kernel}, \eqref{tildeKj}, and \eqref{Glqj_kernel}.
		Then, for any $p\in\mathbb{N}$,
		\begin{equation}\label{KernelBound}
		\begin{aligned}
		|\tilde{K}^{(\ell)}_j(x\cdot y)|&\lesssim 2^{(2+2\ell)j}(1+2^j|x-y|)^{-p},
		\end{aligned}
		\end{equation}
    \begin{equation}\label{KernelBound_low}
		\begin{aligned}
		|\tilde{K}^{(\ell)}_{\leq j}(x\cdot y)|&\lesssim 2^{(2+2\ell)j}(1+2^j|x-y|)^{-p}.
	\end{aligned}
		\end{equation}
		Furthermore, for $j\geq 1$, there exists $c>0$ such that, for any $p\in\mathbb{N}$,
		\begin{equation}\label{KernelBound_semigroup}
		\begin{aligned}
|G^{(\ell)}_j(x\cdot y,t)|&\lesssim 2^{rj}2^{(2+2\ell)j}e^{-c2^{\beta j}t}(1+2^j|x-y|)^{-p}.
		\end{aligned}
		\end{equation}
	\end{lemma}
	\begin{proof} The proof of the estimate \eqref{KernelBound} can be found in Theorem 2.6.5 in \cite{DaiXu13}. Here, we show the details for the more general kernel $G_j$, and then we adapt it to show \eqref{KernelBound_low}.
		
		First, we notice that the Legendre polynomials $P_n$ appearing in $G_j$ are a particular case of the Jacobi polynomials, $P_n^{(a,b)}$, when $a=b=0$ (see e.g. \cite{DaiXu13}). We recall the following properties 
		\begin{equation}\label{Jacobi_1}
		\begin{aligned}
		\frac{d^\ell}{dz^\ell}P_k^{(0,0)}(z)&=\frac{(k+\ell)!}{2^\ell k!}P_{k-\ell}^{(\ell,\ell)}(z),
		\end{aligned}
		\end{equation} 
		\begin{equation}\label{Jacobi_2}
		\begin{aligned}
		|P_m^{(a,b)}(\cos(\theta))|&\lesssim m^{-\frac12}(m^{-1}+\theta)^{-a-\frac12}(m^{-1}+\pi-\theta)^{-b-\frac12},
		\end{aligned}
		\end{equation} 
		\begin{equation}\label{Jacobi_3}
		\begin{aligned}    
		P_m^{(a+1,b)}(x)&=\frac{(m+b)!}{(m+a+b+1)!}\sum_{k=0}^m\frac{(2k+a+b+1)(k+a+b)!}{(k+b)!}P_k^{(a,b)}(x),
		\end{aligned}
		\end{equation} 
		whose proof can be found in [B.1.5], [B.1.7], and [B.1.10] of \cite{DaiXu13}, respectively.
		
		Let us denote $\tilde{\phi}_j(k)=\phi_j(k)\gamma_r(k)$.
		We  take $\ell$ derivatives of $G_j(z)$ and use equality \eqref{Jacobi_1} to obtain
		\begin{equation*}
		\begin{aligned}
		G_j^{(\ell)}(z,t)&=\frac{1}{4\pi}\sum_{k\geq 0}\tilde{\phi}_j(k)e^{-\lambda_k t}(2k+1)\frac{d^{\ell}}{dz^\ell}P_k^{(0,0)}(z)\\
		&=\frac{1}{4\pi}\sum_{k\geq 0}\tilde{\phi}_j(k)e^{-\lambda_k t}\frac{(2k+1)(k+\ell)!}{2^\ell k!}P_{k-\ell}^{(\ell,\ell)}(z)\\
		&=\frac{2^{-\ell}}{4\pi}\sum_{k=0}^{2^{j}-\ell}\tilde{\phi}_j(k+\ell)e^{-\lambda_{k+\ell} t}\frac{(2k+2\ell+1)(k+2\ell)!}{(k+\ell)!}P_{k}^{(\ell,\ell)}(z).
		\end{aligned}
		\end{equation*}
		The upper limit in the sum above comes from the fact that $\tilde{\phi}_j(k)=0$ for $k\geq 2^j$.
		Next, let us denote 
		\begin{equation*}
		\begin{aligned}
		f(k)=\tilde{\phi}_j(k+\ell)e^{-\lambda_{k+\ell} t}, \quad g (k)=\frac{(2k+2\ell+1)(k+2\ell)!}{(k+\ell)!}P_{k}^{(\ell,\ell)}(z),
		\end{aligned}
		\end{equation*}
		so that, noticing that $f(k)=0$ for $k\geq 2^j-\ell$, summation by parts gives
		\begin{equation*}
		\begin{aligned}
		2^\ell 4\pi G_j^{(\ell)}(z,t)&=\sum_{m=0}^{2^j-\ell}(f(m)-f(m+1))\sum_{k=0}^m g(k).
		\end{aligned}
		\end{equation*}
		We compute the last sum using equality \eqref{Jacobi_3} to obtain
		\begin{equation*}
		\begin{aligned}
		2^\ell 4\pi G_j^{(\ell)}(z,t)&\!=\!\!\sum_{m=0}^{2^j-\ell}\!\Big(\tilde{\phi}_j(m\!+\!\ell)e^{-\lambda_{m+\ell} t}\!-\!\tilde{\phi}_j(m\!+\!\ell\!+\!1)e^{-\lambda_{m+\ell+1} t}\Big)\frac{(m\!+\!2\ell\!+\!1)!}{(m+\ell)!}P_m^{(\ell+1,\ell)}(z).
		\end{aligned}
		\end{equation*}
		We can now define recursively
		\begin{equation}\label{hj_def}
		\begin{aligned}
		h_{j,0}(m)&=(2m+2\ell+1)\tilde{\phi}_j(m+\ell)e^{-\lambda_{m+\ell} t},\\
		h_{j,p+1}(m)&=\frac{h_{j,p}(m)}{2m+2\ell+p+1}-\frac{h_{j,p}(m+1)}{2m+2\ell+p+3},
		\end{aligned}
		\end{equation}
		to sum by parts $p-1$ times. Therefore, we get
		\begin{equation*}
		\begin{aligned}
		2^\ell 4\pi G_j^{(\ell)}(z,t)&=\sum_{m=0}^{2^j-\ell}h_{j,p}(m)\frac{(m+2\ell+p)!}{(m+\ell)!}P_m^{(\ell+p,\ell)}(z).
		\end{aligned}
		\end{equation*}
		We claim that, for  $j\geq 1$, $p\geq1$, and $n\geq0$, there exists $c>0$ such that 
        \begin{equation}\label{hjp_bound}
		|h_{j,p}^{(n)}(m)|\lesssim 2^{jr} 
        2^{-j(n+2p-1)}e^{-c2^{\beta j}t}.
		\end{equation}
		With this assumption, taking $z=\cos\theta$, we have
		\begin{equation}\label{Gjlcota}
		\begin{aligned}
		|G_j^{(\ell)}(z,t)|&\lesssim 2^{jr}2^{-j(2p-1)}e^{-c2^{\beta j}t}\sum_{m=0}^{2^j-\ell}\frac{(m+2\ell+p)!}{(m+\ell)!}|P_m^{(\ell+p,\ell)}(z)|\\
		&\lesssim2^{jr}2^{-j(2p-1)}e^{-c2^{\beta j}t}\sum_{m=0}^{2^j-\ell}(m+1)^{\ell+p}|P_m^{(\ell+p,\ell)}(z)|\lesssim2^{jr}2^{-j(2p-1)}e^{-c2^{\beta j}t}S(\theta),
		\end{aligned}
		\end{equation}
		where 
		$$
		S(\theta)=\sum_{m=1}^{2^j}(m+1)^{\ell+p}m^{-\frac12}(m^{-1}+\theta)^{-\ell-p-\frac12}(m^{-1}+\pi-\theta)^{-\ell-\frac12}.
		$$
		In the last step we use \eqref{Jacobi_2}, and remark that the inequalities above depend on the fixed numbers $\ell$ and $p$. To estimate $S(\theta)$ we distinguish between $\theta\in[0,\pi/2]$ and $\theta\in [\pi/2,\pi]$. In the first case, we have 
		$$
		S(\theta)\lesssim \sum_{m=1}^{2^j}(m+1)^{\ell+p}m^{-\frac12}(m^{-1}+\theta)^{-\ell-p-\frac12}.
		$$
		Then, we split further as follows
		$$
		S(\theta)\lesssim\sum_{m=1}^{\min\{\theta^{-1},2^j\}}+\sum_{m=\min\{\theta^{-1},2^j\}}^{2^j}=S_1(\theta)+S_2(\theta).
		$$
		For the first sum
		\begin{equation*}
		\begin{aligned}
		S_1(\theta)&\lesssim \sum_{m=1}^{\min\{\theta^{-1},2^j\}}m^{\ell+p}m^{-\frac12}m^{\ell+p+\frac12}\lesssim (\min\{\theta^{-1},2^j\})^{2\ell+2p+1}\\
		&\lesssim 2^{j(2\ell+2p+1)}(1+2^j\theta)^{-(2\ell+2p+1)}.    
		\end{aligned}
		\end{equation*}
		For the second one, we only need to consider the case $2^j\theta\geq1$, so we obtain
		\begin{equation*}
		\begin{aligned}
		S_2(\theta)&\lesssim \sum_{m=\min\{\theta^{-1},2^j\}}^{2^j}\frac{m^{\ell+p-\frac12}}{\theta^{\ell+p+\frac12}}\lesssim \frac{2^{j(\ell+p+\frac12)}}{\theta^{\ell+p+\frac12}}  \lesssim \frac{2^{j(2\ell+2p+1)}}{(1+2^j\theta)^{\ell+p+\frac12}}.
		\end{aligned}
		\end{equation*}
		The case $\theta\in[\pi/2, \pi]$ is done similarly: 
		\begin{equation*}
		S(\theta)\lesssim\sum_{m=1}^{2^j}m^{\ell+p-\frac12}(m^{-1}+\pi-\theta)^{-\ell-\frac12}\lesssim \sum_{m=1}^{2^j}m^{2\ell+p}\lesssim 2^{j(2\ell+p+1)}\lesssim \frac{2^{j(2\ell+2p+1)}}{(1+2^j\theta)^{p}}.
		\end{equation*}
		For all $\theta\in[0,\pi]$ it is possible to get
		$$
		S(\theta)\lesssim \frac{2^{j(2\ell+2p+1)}}{(1+2^j\theta)^{p}}.
		$$
		Plugging the above estimate in \eqref{Gjlcota} provides finally the desired bound \eqref{KernelBound_semigroup}. It remains to prove the claim.  We show the details in the case $p=1$, as $p>1$ follows by induction. From \eqref{hj_def}, we have
		\begin{equation*}
		\begin{aligned}
		h_{j,1}(s)&=\tilde{\phi}_j(s+\ell)e^{-\lambda_{\beta}(s+\ell) t}-\tilde{\phi}_j(s+\ell+1)e^{-\lambda_{\beta}(s+\ell+1) t}\\
		&=-\int_{\ell/2^j}^{(\ell+1)/2^j} \frac{d}{dz}\Big[\phi_1\big(\frac{s}{2^{j-1}}+2z\big)\gamma_r(s+2^jz)\exp{\big(-\lambda_\beta(s+2^jz) t}\big)\Big]dz\\
		&=\int_{\ell/2^j}^{(\ell+1)/2^j} \!\!\Big(\!-\phi_1\big(\frac{s}{2^{j-1}}+2z\big)\gamma_r'(s+2^jz)2^j+\gamma_r(s+2^jz)\phi_1\big(\frac{s}{2^{j-1}}+2z\big)\lambda'_\beta(s+2^jz)2^{j}t\\
		&\hspace{2.5cm}-2\gamma_r(s+2^jz)\phi_1'\big(\frac{s}{2^{j-1}}+2z\big)\Big)\times\exp{\big(-\lambda_\beta(s+2^jz) t}\big)dz.
		\end{aligned}
		\end{equation*}
		Finally, recalling that  $\text{supp }\phi_1\subset[2/3,2]$, it holds that $\|\phi_1^{(m)}\|_{L^\infty[0,z]}\leq z\|\phi_1^{(m+1)}\|_{L^\infty([0,z])}$ for any $z\geq0$ and $m\geq0$. Since $|\lambda'_\beta(s)|\lesssim (1+s)^{\beta-1}$, $|\gamma_r(2^j)|\lesssim 2^{jr}$, $|\gamma'_r(s)|\lesssim (1+s)^{r-1}$ and $s\lesssim 2^j$ due to the support of $\phi_j$, we obtain
		\begin{equation*}
		\begin{aligned}
		|h_{j,1}(s)|&\lesssim 2^{rj}e^{-c2^{j\beta}t}\frac{1}{2^j}\Big(1+\frac{s+\ell+1}{2^j}\Big)\lesssim 2^{rj}2^{-j}e^{-c2^{j\beta}t},
		\end{aligned}
		\end{equation*}
and, analogously, 
		\begin{equation*}
		\begin{aligned}
		|h_{j,1}^{(n)}(s)|&\lesssim 2^{rj}e^{-c2^{j\beta}t}\frac{1}{2^{j(1+n)}}\Big(1+\frac{s+\ell+1}{2^j}\Big)\lesssim 2^{rj}2^{-j(1+n)}e^{-c2^{j\beta}t}.
		\end{aligned}
		\end{equation*}        
        Finally,  the estimate \eqref{KernelBound_low} follows similarly to \eqref{KernelBound_semigroup}. In fact, we only need to adapt  the claim on $h_{j,p}$, which in this case is defined as in \eqref{hj_def} with $\tilde{\phi}_j(m)e^{-\lambda_{m}t}$ replaced by $\phi_{\leq j}(m)$. In particular, $h_{j,1}(s)=\phi_{\leq j}(s+\ell)-\phi_{\leq j}(s+\ell+1)$, and thus the bound on $|h_{j,1}(s)|\lesssim 2^{-j}$ follows as above. Crucially, the induction on $p>1$ follows in the same way thanks to the fact that $\phi_{\leq j}(s)=\chi(2^{-j}s)$ and $\chi$ is identically one near the origin. Hence, for fixed $\ell$ and large $j$, $h_{j,1}(s)$ also vanishes at the origin and is  supported in the region $s\sim 2^j$. 
        
	\end{proof}
	
The surface gradient $\nabla_{\mathbb{S}^2}$ is not a Fourier multiplier and it does not commute with the Littlewood-Paley projectors. Because of this, it will be useful to define  the following set of three vector fields on $\mathbb{S}^2$, $\mathcal{V}=\{V_1,V_2,V_3\}$, 
\begin{equation}\label{Vector-fields}
 V_1=x_2\partial_3-x_3\partial_2,\qquad
 V_2=x_3\partial_1-x_1\partial_3,\qquad
 V_3=x_1\partial_2-x_2\partial_1.
 \end{equation}
These are smooth divergence-free vector fields on $\mathbb S^2$, so that
\begin{equation}\label{eq:L62_rotation_basic}
 \int_{\mathbb S^2}Vf\cdot g\,d\sigma=-\int_{\mathbb S^2}f\cdot Vg\,d\sigma,\quad\mbox{for any }V\in\mathcal V,\quad
 \laplaceSPH=V_1^2+V_2^2+V_3^2.
\end{equation} 
The tangential vector fields $V$ are also called angular derivatives (see e.g. \cite[Section~1.8]{DaiXu13}) since they are the infinitesimal generators of rotations. Due to this fact, they commute with the spectral Littlewood--Paley projections,  $V\Delta_n=\Delta_nV$ for any $n\geq 0$ and $V\in\mathcal V$ (see \cite{DaiXu13} or \cite{GarciaHaziotKuoMoriZhou2026} for details).

	\begin{lemma}
		Let $\Delta_j f$ be given by \eqref{DeltajfIntegral} and $q\in [1,+\infty]$, then    
		\begin{equation}\label{BernsteinBound}
		\|\nabla_{\mathbb{S}^2}\Delta_j f\|_{L^q}\lesssim 2^j\|\Delta_j f\|_{L^q}. 
		\end{equation}
		Moreover, when $j>0$, 
		\begin{equation}\label{BernsteinBound02}
		\|\Delta_j f\|_{L^q}\lesssim 2^{-j}\|\nabla_{\mathbb{S}^2}\Delta_j f\|_{L^q}. 
		\end{equation}
	\end{lemma}
	
	\begin{proof} Differentiating \eqref{DeltajfIntegral} we obtain
		\begin{equation*}
		\begin{aligned}
		|\nabla_{\mathbb{S}^2}\Delta_j f(x)| &=
		\Big|\int_{\mathbb{S}^{2}}\tilde{K}'_j(x\cdot y)(y- (x\cdot y)x)\Delta_j f(y) d\sigma(y)\Big| \\
		&\leq \int_{\mathbb{S}^{2}}|\tilde{K}'_j(x\cdot y)| \sqrt{1-(x\cdot y)^2} |\Delta_j f(y)| d\sigma(y),
		\end{aligned}
		\end{equation*}
		where we abuse notation and still denote $\tilde{K}$ the kernel associated to the enlarged dyadic block $\tilde{\Delta}_j=\sum_{|k-j|\leq 1}\Delta_k$, which satisfies $\Delta_j \tilde{\Delta}_j=\Delta_j$. 
		By Young's inequality for convolutions, where we let $\cos(\theta) = x\cdot y$,
		\begin{equation*}
		\|\nabla_{\mathbb{S}^2}\Delta_j f\|_{L^q_{x}}\lesssim \|\tilde{K}'_j(\cos(\theta)) \sin^2(\theta)\|_{L^1_{\theta}([0,\pi])}\|\Delta_jf\|_{L^q_{x}}.
		\end{equation*}
		From \eqref{KernelBound} with $p=4$,
		\begin{equation*}
		|\tilde{K}'_j(\cos(\theta))| \lesssim 2^{4j}(1+2^j\sin(\theta))^{-4}.
		\end{equation*}
		Therefore, 
		\begin{equation*}
		\|\tilde{K}'_j(\cos(\theta)) \sin^2(\theta)\|_{L^1_{\theta}([0,\pi])} \lesssim 2^{4j} \int_{0}^{\frac{\pi}2} (1+2^j\sin(\theta))^{-4}\sin^2(\theta) \ d\theta,
		\end{equation*}
		and we compute 
		\begin{equation*}
		\int_{0}^{\frac{\pi}2} (1+2^j\sin(\theta))^{-4}\sin^2(\theta) \ d\theta\lesssim \int_{0}^{\frac{\pi}2} (1+c2^j\theta)^{-4}\sin^2(\theta) \ d\theta\lesssim2^{-3j}\int_{0}^{c2^j} \!\!\frac{\eta^2d\eta}{(1+\eta)^4}\lesssim 2^{-3j}.
		\end{equation*}
		Hence, we obtain that $  \|\tilde{K}'_j(\cos(\theta)) \sin^2(\theta)\|_{L^1_{\theta}([0,\pi])} \lesssim 2^{j}$, which gives the first estimate \eqref{BernsteinBound}.
		
		Next, to prove the reverse inequality, we notice that for $j\geq1$ we can write
		\begin{equation*}
		\begin{aligned}
		\Delta_j f(x)&=\Delta_j\tilde{\Delta}_j f(x)=\int_{\mathbb{S}^2}\sum_{|k-j|\leq1}\sum_{n\geq1}\frac{\phi_k(n)}{n(n+1)}K_n(x\cdot y)(-\Delta_{\mathbb{S}^2})\Delta_jf(y)d\sigma(y).
		\end{aligned}
		\end{equation*}
		We define the kernel
		\begin{equation*}
		\begin{aligned}
		H_j(z)=\sum_{|j-k|\leq1}\sum_{n\geq1}\frac{\phi_k(n)}{n(n+1)}K_n(z),
		\end{aligned}
		\end{equation*}
		so that integration by parts gives
		\begin{equation*}
		\begin{aligned}
		\Delta_j f(x)&=\int_{\mathbb{S}^2}\nabla_{\mathbb{S}^2}H_j(x\cdot y)\cdot\nabla_{\mathbb{S}^2}\Delta_jf(y)d\sigma(y).
		\end{aligned}
		\end{equation*}
		We notice that  $H_j$ consists of three terms which are the particular case of $G_j$ in \eqref{Gj_kernel} when $t=0$ and $\gamma_r(n)=\big(n(n+1)\big)^{-1}$. Therefore, 
		the estimates for the kernel $H_j$ follow from Lemma \ref{lem:kernel_bounds}. Taking $p=4$,
		\begin{equation*}
		\begin{aligned}
		|H_j'(x\cdot y)|&\lesssim  2^{2j}(1+2^j|x-y|)^{-4}.
		\end{aligned}
		\end{equation*}
		Therefore, 
		\begin{equation*}
		\begin{aligned}
		\|\Delta_j f\|_{L^q}&\lesssim\|H_j'(\cos\theta)\sin^2\theta\|_{L^1_\theta[0,\pi]}\|\nabla_{\mathbb{S}^2}\Delta_j f\|_{L^q}\lesssim 2^{-j}\|\nabla_{\mathbb{S}^2}\Delta_j f\|_{L^q}.
		\end{aligned}
		\end{equation*}
		
	\end{proof}

\begin{remark}
By the  definition of the tangential vector fields $V\in\mathcal{V}$ in \eqref{Vector-fields}, the previous Bernstein's inequalities also apply: 
\begin{equation}\label{eq:L62_rotation_reverse}
 \begin{aligned}
 \|V\Delta_nf\|_{L^\infty}&\lesssim 2^n\|\Delta_nf\|_{L^\infty},\hspace{2.45cm} n\geq0,\\
  \|\Delta_nF\|_{L^\infty}
 &\lesssim 2^{-n}\sum_{V\in\mathcal V}
 \|\Delta_n(VF)\|_{L^\infty},\qquad n\geq1,
 \end{aligned}
\end{equation} 
where in the second one the fact that the fields $V\in\mathcal{V}$ commute with the projector $\Delta_n$ is used.
\end{remark}

	\begin{cor}\label{cor:multiplier}
		Let $r$ be a nonnegative integer, $p\in[1,+\infty]$, and let $\beta<0$. Let $\lambda_\beta$ be a multiplier satisfying
		\begin{equation*}
		c(1+n)^{\beta}\leq \lambda_\beta(n)\leq C(1+n)^{\beta},\quad |\frac{d^m}{ds^m}\lambda_{\beta}(s)|\lesssim (1+s)^{\beta-m},
		\end{equation*}
		and let $D^\beta$ be the associated Fourier multiplier operator. Then, 
		\begin{equation*}
		\begin{aligned}
		\|\Delta_j D^\beta f\|_{L^p}\lesssim 2^{j\beta}\|\Delta_j f\|_{L^p},
		\end{aligned}
		\end{equation*}
		and, for $r+\beta<0$,
		\begin{equation*}
		\begin{aligned}
		\|D^\beta f\|_{W^{r,p}}\lesssim \|f\|_{L^p}.
		\end{aligned}
		\end{equation*}
		Moreover, 
		\begin{equation*}
		\begin{aligned}
		\|D^{-r} f\|_{W^{r,\infty}}\lesssim \|f\|_{W^{1,\infty}}.
		\end{aligned}
		\end{equation*}
	\end{cor}
	
	\begin{proof}
		The first estimate follows again from Young's inequality and the kernel estimate in Lemma \ref{lem:kernel_bounds}.
		The second estimate follows immediately from the first. Indeed
		\begin{equation*}
		\begin{aligned}
		\|D^\beta f\|_{W^{r,p}}&\lesssim \sum_{j\geq 0}2^{rj}\|D^\beta \Delta_j f\|_{L^p} \lesssim\sum_{j\geq 0}2^{(r+\beta)j}\|\Delta_j f\|_{L^p}\lesssim \|f\|_{L^p}. 
		\end{aligned}
		\end{equation*}
		Similarly, we obtain the last estimate:
		\begin{equation*}
		\begin{aligned}
		\|D^{-r} f\|_{W^{r,\infty}}&\lesssim \sum_{j\geq 0}2^{rj}\|D^{-r} \Delta_j f\|_{L^\infty} \lesssim\sum_{j\geq 0}\|\Delta_j f\|_{L^\infty}\lesssim \|f\|_{W^{1,\infty}}. 
		\end{aligned}
		\end{equation*}
	\end{proof}
	
	Given $m$ a nonnegative integer, $\alpha\in(0,1)$, $s\in\mathbb{R}$, and a function $f:\mathbb{S}^{2} \rightarrow\mathbb{R}$, we define the H\"older norm
	\begin{equation}\label{Holder_norm_def}
	\|f\|_{C^{m,\alpha}} = \sum_{j=0}^m\|\nabla_{\mathbb{S}^2}^j f\|_{L^{\infty}} + \|\nabla_{\mathbb{S}^2}^m f\|_{\dot{C}^\alpha},\quad\mbox{with}\quad \|f\|_{\dot{C}^{\alpha}}= \sup_{x\neq y} \frac{|f(x)-f(y)|}{|x-y|^{\alpha}}
	\end{equation}
	and the Besov norm
	\begin{equation*}
	\|f\|_{B^{s}_{\infty,\infty}} = \sup_{j\geq 0} 2^{js}\|\Delta_{j}f\|_{L^{\infty}}.
	\end{equation*}
    For simplificy of notation, we will denote $C^{0,\alpha}(\mathbb{T}^2)$ by $C^{\alpha}(\mathbb{T}^2)$.

	\begin{prop}\label{Prop-equiv}
		Suppose $0<\alpha<1$ and $m$ is a nonnegative integer. Then, we have that 
		\begin{equation}\label{BesovHolder}
		\|f\|_{C^{m,\alpha}}  \cong\|f\|_{B^{m+\alpha}_{\infty,\infty}}.
		\end{equation}
	\end{prop}
	
	\begin{proof} For the sake of completeness, we provide the proof for the case $m=0$. The cases $m>0$ are left to the interested reader. Using \eqref{DeltajfIntegral} for $j\geq 2$, we can write
		$$
		\Delta_j f(x)=
		\int_{\mathbb{S}^{2}}\tilde{K}_j(x\cdot y)(f(y)-f(x))d\sigma(y),
		$$
		allowing the following decomposition for non-constant functions
		$$
		|\Delta_j f(x)|\|f\|_{\dot{C}^{\alpha}}^{-1}\leq \int_{\mathbb{S}^{2}}|\tilde{K}_j(x\cdot y)||y-x|^{\alpha}d\sigma(y)=\int_{|x-y|<2^{-j}}d\sigma(y)+\int_{|x-y|\geq 2^{-j}}d\sigma(y).
		$$
		Using the bound \eqref{KernelBound} it is possible to estimate as follows
		\begin{equation*}
		\begin{aligned}
		|\Delta_j f(x)|\|f\|_{\dot{C}^{\alpha}}^{-1} &\lesssim 2^{-j\alpha}\int_{|x-y|<2^{-j}}2^{2j}d\sigma(y)+2^{-j\alpha}\int_{|x-y|\geq 2^{-j}}\frac{(2^j|x-y|)^{\alpha}}{(1+2^j|x-y|)^{4}}2^{2j}d\sigma(y)\\
		&\lesssim 2^{-j\alpha}.  
		\end{aligned}
		\end{equation*}
		For $0\leq j<2$, identity \eqref{DeltajfIntegral} and bound \eqref{KernelBound} provide $|\Delta_j f(x)|\lesssim \|f\|_{L^\infty}$. The last two estimates give the desired control. 
		
		Consider next $f$ such that $\|f\|_{B^{\alpha}_{\infty,\infty}}<\infty$. It is direct to obtain
		$$
		\|f\|_{L^\infty}\leq \sum_{j\geq 0}\|\Delta_j f\|_{L^\infty}\leq (1-2^{-\alpha})^{-1} \|f\|_{B^{\alpha}_{\infty,\infty}}.
		$$
		Choosing $J\in\mathbb{Z}$ such that $2^{-J-1}<|x-y|\leq 2^{-J}$, we split as follows
		$$
		f(x)-f(y)=\sum_{j\geq0}^{\infty}(\Delta_jf(x)-\Delta_jf(y))=\sum_{j\leq J}+\sum_{j> J}.
		$$
		Then
		$$
		\big|\sum_{j> J}\big|\leq 2\sum_{j>J}\|\Delta_j f\|_{L^\infty}\leq 2(1-2^{-\alpha})^{-1}2^{-(J+1)\alpha}\|f\|_{B^\alpha_{\infty,\infty}}\lesssim|x-y|^\alpha\|f\|_{B^\alpha_{\infty,\infty}}.
		$$
		The regularity of $\Delta_jf$ provides 
		$$
		\big|\sum_{j\leq J}\big|\lesssim \sum_{j\leq J}\|\nabla_{\mathbb{S}^2}\Delta_j f\|_{L^\infty}|x-y|\lesssim \sum_{j\leq J}2^j\|\Delta_j f\|_{L^\infty}|x-y|  \lesssim \sum_{j\leq J}2^{j(1-\alpha)}\|f\|_{B^\alpha_{\infty,\infty}}|x-y|,
		$$
		using Bernstein's inequality \eqref{BernsteinBound}. Hence
		$$
		\big|\sum_{j\leq J}\big|\lesssim  \frac{2^{(J+1)(1-\alpha)}-1}{2^{1-\alpha}-1}\|f\|_{B^\alpha_{\infty,\infty}}2^{-J(1-\alpha)}|x-y|^{\alpha} \lesssim|x-y|^\alpha\|f\|_{B^\alpha_{\infty,\infty}} ,
		$$
		which completes the proof in the case $m=0$.	
	\end{proof}

	Finally we compute some integrals on the sphere used throughout the paper.
    
	\begin{lemma}
		Let \(x\in \mathbb S^2\). Then
		\begin{equation}\label{eq:sphere-tensor}
		\begin{aligned}
		\frac1{4\pi}\int_{\mathbb S^2}
		\frac{(x-y)\otimes(x-y)}{|x-y|^3}\,d\sigma(y)
		=\frac{1}{3}\operatorname{Id}_{\mathbb R^3},
		\end{aligned}
		\end{equation}
		\begin{equation}\label{eq:sphere-vector}
		\begin{aligned}
		\frac1{4\pi}\operatorname{p.v.}\int_{\mathbb S^2}
		\frac{x-y}{|x-y|^3}\,d\sigma(y)
		=\frac12 x. 
		\end{aligned}
		\end{equation}
		Furthermore, for every \(a\in\mathbb R^3\),
		\begin{equation}
		\frac1{4\pi}\operatorname{p.v.}\int_{\mathbb S^2}
		\frac{x-y}{|x-y|^3}\wedge (y\wedge a)\,d\sigma(y)
		=
		\frac12(a\cdot x)x+\frac{1}{6}a.
		\label{eq:sphere-cross-identity}
		\end{equation}
	\end{lemma}
	
	\begin{proof}
		Denote the tensor integral in \eqref{eq:sphere-tensor} by $M(x)$ and let \(Q\in SO(3)\).
		Since \(Q\mathbb S^2=\mathbb S^2\), \(d\sigma(Qy)=d\sigma(y)\), and
		$
		|Qx-Qy|=|x-y|,
		$
		taking $y=Qz$, we find that $M(Qx)=QM(x)Q^*$, with $Q^*$ the transpose of $Q$. Then for any \(x\in\mathbb S^2\) we can choose
		\(Q\in SO(3)\) such that \(Qe_3=x\). Therefore, it suffices to compute $M(e_3)$.
		We parametrize
		\[
		y=(\sin\theta\cos\varphi,\sin\theta\sin\varphi,\cos\theta),
		\qquad
		d\sigma(y)=\sin\theta\,d\theta\,d\varphi.
		\]
		The off-diagonal terms vanish by
		symmetry in \(\varphi\). For the first diagonal component it is possible to find
		\[
		\frac1{4\pi}\int_{\mathbb S^2}
		\frac{(e_3-y)_1^2}{|e_3-y|^3}\,d\sigma(y)
		=\frac1{4\pi}\int_0^{2\pi}\cos^2\varphi\,d\varphi
		\int_0^\pi \cos^3\frac{\theta}{2}\,d\theta=\frac{1}{3},
		\]
		by writing $\sin\theta=2\sin(\theta/2)\cos(\theta/2)$ and $1-\cos\theta=2\sin^2(\theta/2)$.
		The same computation gives the second diagonal term.
		For the third diagonal component, in a similar manner we obtain
		\[
		\frac1{4\pi}\int_{\mathbb S^2}
		\frac{(e_3-y)_3^2}{|e_3-y|^3}\,d\sigma(y)
		=
		\frac1{4\pi}\int_0^{2\pi}\int_0^\pi
		\sin^2\frac{\theta}{2}\cos\frac{\theta}{2}
		\,d\theta\,d\varphi=\frac13.
		\]
		We now prove the principal value identity. Define
		\[
		V(x)=
		\frac1{4\pi}\operatorname{p.v.}\int_{\mathbb S^2}
		\frac{x-y}{|x-y|^3}\,d\sigma(y),\quad\mbox{and}\quad
		U(x)=\frac1{4\pi}\int_{\mathbb S^2}\frac{1}{|x-y|}\,d\sigma(y).
		\]
		By \eqref{newtonian_spectral}, $U(x)=1$ for every
		\(x\in\mathbb S^2\). Hence
		$
		\nabla_{\mathbb S^2}U(x)=0.
		$
		On the other hand,
		\[
		\nabla_x\frac{1}{|x-y|}
		=
		-\frac{x-y}{|x-y|^3}.
		\]
		Projecting onto the tangent plane, we get
		\[
		\nabla_{\mathbb S^2,x}\frac{1}{|x-y|}
		=
		-(\mathrm{Id}-x\otimes x)\frac{x-y}{|x-y|^3}.
		\]
		Therefore, in the principal value sense,
		\[
		0
		=
		\nabla_{\mathbb S^2}U(x)
		=
		-(\mathrm{Id}-x\otimes x)
		\operatorname{p.v.}\frac1{4\pi}\int_{\mathbb S^2}
		\frac{x-y}{|x-y|^3}\,d\sigma(y).
		\]
		Equivalently,
		$
		(\mathrm{Id}-x\otimes x)V(x)=0.
		$ Thus \(V(x)\) has no tangential component, and consequently,
		$
		V(x)=\lambda(x)x.
		$ It remains to compute the normal component. Taking the scalar product
		with \(x\), and using the identity
		$
		x\cdot(x-y)=|x-y|^2/2,
		$
		we find
		\[
		x\cdot V(x)
		=
		\frac1{4\pi}\int_{\mathbb S^2}
		\frac{x\cdot(x-y)}{|x-y|^3}\,d\sigma(y)
		=
		\frac1{8\pi}\int_{\mathbb S^2}\frac{1}{|x-y|}\,d\sigma(y)
		=
		\frac12.
		\]
		This gives
		$
		\lambda(x)=1/2
		$
		and hence
		\eqref{eq:sphere-vector} is proved. The differentiation above may be
		justified by introducing a symmetric truncation around \(y=x\), performing
		the integration by parts/differentiation on the truncated integral, and
		then passing to the principal value limit.
		It remains to prove \eqref{eq:sphere-cross-identity}. We rewrite
		\[
		\frac{x-y}{|x-y|^3}\wedge(y\wedge a)
		=
		\frac{x-y}{|x-y|^3}\wedge(x\wedge a)
		-
		\frac{x-y}{|x-y|^3}\wedge\bigl((x-y)\wedge a\bigr)=f_1-f_2.
		\] Next, we use the triple product identity 
		\begin{equation}\label{triple_product}
		u\wedge(v\wedge a)=v(u\cdot a)-a(u\cdot v)
		\end{equation}
		to obtain
		\[
		f_1
		=
		x\frac{(x-y)\cdot a}{|x-y|^3}
		-
		a\frac{(x-y)\cdot x}{|x-y|^3},\quad\mbox{and}
		\quad
		f_2=
		(x-y)\frac{(x-y)\cdot a}{|x-y|^3}
		-
		a\frac{|x-y|^2}{|x-y|^3}.
		\]
		The above two identities for $f_1$ and $f_2$ together with equality $(x-y)\cdot x
		=|x-y|^2/2$ yield
		\begin{equation*}
		\frac{x-y}{|x-y|^3}\wedge(y\wedge a)
		=x\frac{(x-y)\cdot a}{|x-y|^3}
		-
		\frac{(x-y)\otimes(x-y)}{|x-y|^3}a
		+
		\frac{a}{2|x-y|}.
		\end{equation*}
		Integrating over \(\mathbb S^2\), we get
		the last identity by using \eqref{eq:sphere-vector}, \eqref{eq:sphere-tensor} and
		\eqref{newtonian_spectral}.
		
	\end{proof}

	\section{Linearization}\label{sec:linearization}

	We linearize the contour equation \eqref{contour_eq} around the translating sphere
	\begin{equation*}
	S_0(x,t)=Rx+c_0(t), 
	\end{equation*}
	where we decompose $c(t)$ in \eqref{ct_def} as follows
	\begin{equation}\label{c0cL_def}
	\begin{aligned}
	c(t)=c_0(t)+c_L(t),\,\,\mbox{with}\,\, c_0(t)=\frac23 A_\rho t e_3,\,\,\mbox{and}\,\,  c_L(t)= \frac{3A_\sigma}{4\pi R^2}\int_0^t\int_{\mathbb{S}^2}f(y,s)yd\sigma(y)ds.
	\end{aligned}
	\end{equation}
	We shall use repeatedly the elementary identities
	\begin{equation}\label{aux_e3}
	\nabla_{\mathbb S^2}x_3=e_3-x_3x,
	\qquad
	\Delta_{\mathbb S^2}x_3=-2x_3.    
	\end{equation}
	For small \(f\), Taylor expansion gives
	\[
	r=1+f+\mathcal O(f^2).
	\]
	Thus
	\begin{equation*}
	S(x,t)
	=S_0(x,t)+S_L(x,t)+\mathcal O(f^2)    
	\end{equation*}
	with
	\[
	S_0(x,t)=Rx+c_0(t),
	\qquad
	S_L(x,t)=Rf(x,t)x+c_L(t).
	\]
   For simplicity of notation, we suppress the time dependence in the rest of this section and, whenever no confusion can arise, also the spatial argument.
   Since the Birkhoff--Rott kernel only depends on differences of the form
	\(S(x)-S(y)\), the translational part cancels in the kernel:
	\[
	S(x)-S(y)
	=
	R(x-y)+R\bigl(f(x)x-f(y)y\bigr)+\mathcal O(f^2).
	\]
	We expand
	\[
	BR(S,\omega)=BR_0+BR_L+\mathcal O(f^2).
	\]
	Therefore
	\[
	BR(S,\omega)-\dot c(t)
	=
	BR_0-\dot c_0
	+
	BR_L-\dot c_L
	+
	\mathcal O(f^2).
	\]
	Noticing that the expression of the non-unit normal \eqref{def:nonunit_normal_f} is linear in $f$, we can multiply by the expansion of \(\mathcal N_f/r\) to obtain the linear terms 
	\begin{equation*}
	\partial_t f
	=\frac1R\bigl(BR_0-\dot c_0\bigr)\cdot x+
	\frac1R
	\left[
	\bigl(BR_L-\dot c_L\bigr)\cdot x
	+
	\bigl(BR_0-\dot c_0\bigr)
	\cdot
	\left(2fx-\nabla_{\mathbb S^2}f\right)
	\right]
	+
	\mathcal O(f^2).
	\end{equation*}
	We note that the zeroth-order term vanishes by the choice of the translating vertical velocity $\dot{c}_0$ in \eqref{c0cL_def}. In fact, we will see that $BR_0-\dot{c}_0$ is tangential to the sphere (see \eqref{c0_cancellation}).
	Using the notation introduced in \eqref{Omega_split},  we define the linear operator
	\begin{equation}\label{linear_operator}
	\mathcal{L}[f](x)=\mathcal{L}_\sigma[f](x)+\mathcal{L}_\rho[f](x),
	\end{equation}
	where
	\begin{equation*}
	\begin{aligned}
	\mathcal{L}_\sigma[f](x)&=\frac1R
	\bigl(BR_{\sigma,L}-\dot c_L\bigr)\cdot x,
	\\
	\mathcal{L}_\rho[f](x)&=\frac1R
	\left[
	BR_{\rho,L}\cdot x
	+
	\bigl(BR_0-\dot c_0\bigr)
	\cdot
	\left(2fx-\nabla_{\mathbb S^2}f\right)
	\right].
	\end{aligned}
	\end{equation*}

	We now compute the linear terms appearing in this formula.
	
	\subsubsection*{Linearization of the potential}
	
	The curvature of the perturbed surface satisfies
	\[
	K[S]=\frac1R+K_L+\mathcal O(f^2),
	\]
	where $1/R$ is the curvature of the sphere of radius \(R\).
	Computing the first variation of the curvature under the radial perturbation \(S_L=Rf\,x\) (see e.g. \cite[Section~2.2]{PrussSimonett2016}), one obtains 
	\[
	K_L
	=
	-\frac{1}{2R}
	\left(
	\Delta_{\mathbb S^2}f+2f
	\right).
	\]
	Therefore the scalar potential \(\Omega\) expands as
	\[
	\Omega=\Omega_0+\Omega_L+\mathcal O(f^2),
	\]
	with
	\[
	\Omega_0(x)
	=
	\frac{A_\sigma}{R}
	-
	A_\rho R x_3-A_\rho c_0(t)\cdot e_3,
	\] and
	\begin{equation}\label{OmegaL}
	\begin{aligned}
	\Omega_L(x)
	&=\Omega_{\sigma,L}(x)+\Omega_{\rho,L}(x)\\
	&=
	-\frac{A_\sigma}{2R}
	\left(
	\Delta_{\mathbb S^2}f(x)+2f(x)
	\right)
	-
	A_\rho R x_3 f(x)-A_\rho c_L(t)\cdot e_3.
	\end{aligned}
	\end{equation}
	
	\subsubsection*{Linearization of the  vorticity}
	We expand  the  vorticity density
	$\omega=\omega_0+\omega_L+\mathcal O(f^2)$. Using \eqref{vorticity}
	we have
	\begin{equation*}
	\omega_0(x)=D_{\mathbb S^2}S_0(x)\left[x\wedge\nabla_{\mathbb{S}^2}\Omega_0(x)\right].
	\end{equation*}
	Since
	\begin{equation}\label{DsS_0}
	D_{\mathbb S^2}S_0=R\,\mathrm{Id}_{T_x\mathbb S^2},   \quad \nabla_{\mathbb S^2}\Omega_0
	=
	-A_\rho R\nabla_{\mathbb S^2}x_3,
	\end{equation}
	we use \eqref{aux_e3} and
	$x\wedge\nabla_{\mathbb S^2}x_3=x\wedge e_3$, 
	to find
	\begin{equation}\label{omega_0_final}
	\omega_0(x)
	=
	-A_\rho R^2 x\wedge e_3.
	\end{equation}
	The linear part is
	\begin{equation*}
	\begin{aligned}
	\omega_L(x)&=\omega_{\sigma,L}(x)+\omega_{\rho,L}(x),\\
	\omega_{\sigma,L}(x)&=D_{\mathbb S^2}S_0(x)
	\left[
	x\wedge\nabla_{\mathbb S^2}\Omega_{\sigma,L}(x)
	\right],\\
	\omega_{\rho,L}(x)&=
	D_{\mathbb S^2}S_0(x)
	\left[
	x\wedge\nabla_{\mathbb S^2}\Omega_{\rho,L}(x)
	\right]+D_{\mathbb S^2}S_L(x)
	\left[
	x\wedge\nabla_{\mathbb S^2}\Omega_0(x)
	\right].
	\end{aligned}
	\end{equation*}
	First, the surface tension term gives
	\begin{equation}\label{omega_sigma_L}
	\omega_{\sigma, L}=R\,x\wedge\nabla_{\mathbb S^2}
	\left[
	-\frac{A_\sigma}{2R}
	\left(
	\Delta_{\mathbb S^2}f+2f
	\right)
	\right]
	=
	-\frac{A_\sigma}{2}
	x\wedge\nabla_{\mathbb S^2}
	\left(
	\Delta_{\mathbb S^2}f+2f
	\right).
	\end{equation}
	The first part of the density-jump term gives
	\begin{equation*}
	R\,x\wedge\nabla_{\mathbb S^2}
	\left(
	-A_\rho R x_3f
	\right)
	=
	-A_\rho R^2
	x\wedge\nabla_{\mathbb S^2}(x_3f).
	\end{equation*}
	Expanding the derivative, we get
	\[
	x\wedge\nabla_{\mathbb S^2}(x_3f)
	=
	x_3\,x\wedge\nabla_{\mathbb S^2}f
	+
	f\,x\wedge e_3.
	\]
	Next we compute the contribution coming from \(D_{\mathbb S^2}S_L\).
	For a tangent vector \(V\in T_x\mathbb S^2\),
	\[
	D_{\mathbb S^2}S_L[V]
	=
	R\left[
	fV+(V\cdot\nabla_{\mathbb S^2}f)x
	\right].
	\]
	Using as before that 
	$
	x\wedge\nabla_{\mathbb S^2}\Omega_0
	=
	-A_\rho R\,x\wedge e_3,
	$
	we obtain
	\[
	D_{\mathbb S^2}S_L
	\left[
	x\wedge\nabla_{\mathbb S^2}\Omega_0
	\right]
	=
	-A_\rho R^2
	\left[
	f\,x\wedge e_3
	+
	\bigl((x\wedge e_3)\cdot\nabla_{\mathbb S^2}f\bigr)x
	\right].
	\]
	Recalling \eqref{DsS_0} and \eqref{OmegaL}  and combining the two contributions, we arrive at
	\begin{equation}\label{omega_rho_L}
	\omega_{\rho,L}
	=-
	A_\rho R^2
	\left[
	x_3\,x\wedge\nabla_{\mathbb S^2}f
	+
	2f\,x\wedge e_3
	+
	\bigl((x\wedge e_3)\cdot\nabla_{\mathbb S^2}f\bigr)x
	\right].
	\end{equation}

	\subsubsection*{Zeroth-order Birkhoff--Rott velocity}
	
	The zeroth-order velocity is
	\[
	BR_0(x)
	=
	-\frac{1}{4\pi R^2}\operatorname{p.v.}
	\int_{\mathbb S^2}
	\frac{x-y}{|x-y|^3}
	\wedge
	\omega_0(y)
	\,d\sigma(y).
	\]
	Substituting the expression \eqref{omega_0_final} for $\omega_0$, the
	factors of \(R\) cancel. Using \eqref{eq:sphere-cross-identity}, one obtains
	\begin{equation*}
	BR_0(x)
	=A_\rho\big(\frac12x_3x+\frac16e_3\big)=
	A_\rho
	\big(
	\frac23 x_3 x
	+
	\frac16\nabla_{\mathbb S^2}x_3
	\big).
	\end{equation*}
	By our choice of $c_0$ in \eqref{c0cL_def}, this term removes the constant vertical motion contribution from $BR_0$. In fact,
	using \eqref{aux_e3},
	we have
	\[
	\dot c_0
	=
	\frac23A_\rho x_3x
	+
	\frac23A_\rho\nabla_{\mathbb S^2}x_3.
	\]
	Therefore
	\begin{equation}\label{c0_cancellation}
	BR_0-\dot c_0
	=
	-\frac12 A_\rho\nabla_{\mathbb S^2}x_3.    
	\end{equation}
	We can now compute the second term of the density-jump part of the linear operator, \eqref{linear_operator}. We have
	\[
	\bigl(BR_0-\dot c_0\bigr)
	\cdot
	\left(
	2fx-\nabla_{\mathbb S^2}f
	\right)
	=
	-\frac12A_\rho\nabla_{\mathbb S^2}x_3
	\cdot
	\left(
	2fx-\nabla_{\mathbb S^2}f
	\right).
	\]
	Since \(\nabla_{\mathbb S^2}x_3\) is tangent to \(\mathbb S^2\),
	we obtain finally
	\begin{equation}\label{2ndInLinear_equation}
	\bigl(BR_0-\dot c_0\bigr)
	\cdot
	\left(
	2fx-\nabla_{\mathbb S^2}f
	\right)
	=
	\frac12A_\rho
	\nabla_{\mathbb S^2}x_3
	\cdot
	\nabla_{\mathbb S^2}f.  
	\end{equation}

	\subsubsection*{Linearization of the Birkhoff--Rott integral}
	
	We now compute the linear part of the Birkhoff--Rott velocity \eqref{def:BR_intrinsic}. The surface kernel expansion is given by 
	\begin{equation*}
	\frac{S(x)\!-\!S(y)}{|S(x)\!-\!S(y)|^3}
	=
	\frac1{R^2}
	\left[
	\frac{x\!-\!y}{|x\!-\!y|^3}
	+
	\frac{f(x)x\!-\!f(y)y}{|x\!-\!y|^3}
	-
	3
	\frac{(xf(x)\!-\!yf(y))\cdot (x\!-\!y)}{|x\!-\!y|^5}
	(x\!-\!y)
	\right]
	+
	\mathcal O(f^2).
	\end{equation*}
	Introducing the expansion of the vorticity, 
	we obtain
	\begin{equation}\label{BRL_linearization_expanded}
	\begin{aligned}
	BR_L(x)
	={}&
	-\frac{1}{4\pi R^2}
	\operatorname{p.v.}
	\int_{\mathbb S^2}
	\frac{x-y}{|x-y|^3}
	\wedge
	\omega_L(y)
	\,d\sigma(y)
	\\
	&-
	\frac{1}{4\pi R^2}
	\operatorname{p.v.}
	\int_{\mathbb S^2}
	\frac{f(x)x-f(y)y}{|x-y|^3}
	\wedge
	\omega_0(y)
	\,d\sigma(y)
	\\
	&+
	\frac{3}{4\pi R^2}
	\operatorname{p.v.}
	\int_{\mathbb S^2}
	\frac{
		\bigl(f(x)x-f(y)y\bigr)\cdot(x-y)
	}
	{|x-y|^5}
	(x-y)\wedge\omega_0(y)
	\,d\sigma(y).
	\end{aligned}
	\end{equation}
	It remains to project this expression onto \(x\). We therefore compute
	\[
	BR_L(x)\cdot x
	=
	BR_{\sigma,L}(x)\cdot x
	+
	BR_{\rho,L}(x)\cdot x,
	\]
	separating the surface-tension and density-jump contributions.
	
	\subsubsection*{Surface-tension contribution}
	
	The surface-tension contribution appears only through the linear part
	of $\omega_{\sigma, L}$ \eqref{omega_sigma_L},
	\begin{equation*}
	BR_{\sigma,L}(x)
	=
	-\frac{1}{4\pi R^2}
	\operatorname{p.v.}
	\int_{\mathbb S^2}
	\frac{x-y}{|x-y|^3}
	\wedge
	\omega_{\sigma,L}(y)
	\,d\sigma(y).
	\end{equation*}
	Substituting the expression for
	\(\omega_{\sigma,L}\), we obtain
	\begin{equation*}
	BR_{\sigma,L}(x)
	=
	\frac{A_\sigma}{8\pi R^2}
	\operatorname{p.v.}
	\int_{\mathbb S^2}
	\frac{x-y}{|x-y|^3}
	\wedge
	\left(
	y\wedge\nabla_{\mathbb S^2}(\Delta_{\mathbb S^2}f(y)+2f(y))
	\right)
	\,d\sigma(y).
	\end{equation*}
	We now project onto the radial direction \(x\). Thus
	\begin{equation*}
	BR_{\sigma,L}(x)\cdot x
	=
	\frac{A_\sigma}{8\pi R^2}
	\operatorname{p.v.}
	\int_{\mathbb S^2}
	x\cdot
	\Big[
	\frac{x-y}{|x-y|^3}
	\wedge
	(
	y\wedge\nabla_{\mathbb S^2}(\Delta_{\mathbb S^2}f(y)+2f(y))
	)
	\Big]
	\,d\sigma(y).
	\end{equation*}
	Using the triple-product identity \eqref{triple_product}
	with
	$u=x-y$, $v=y$, $a=\nabla_{\mathbb S^2}g(y)$, where $g$ is a general function on the sphere, we get
	\[
	(x-y)\wedge
	\left(
	y\wedge\nabla_{\mathbb S^2}g(y)
	\right)
	=
	y\bigl((x-y)\cdot\nabla_{\mathbb S^2}g(y)\bigr)
	-
	\nabla_{\mathbb S^2}g(y)\bigl((x-y)\cdot y\bigr).
	\]
	Taking the scalar product with \(x\), this becomes
	\begin{equation}\label{triple_g_op}
	\begin{aligned}
	x\cdot
	\left[
	(x-y)\wedge
	\left(
	y\wedge\nabla_{\mathbb S^2}g(y)
	\right)
	\right]
	&=
	(x\cdot y)\bigl((x-y)\cdot\nabla_{\mathbb S^2}g(y)\bigr)
	-
	\bigl(x\cdot\nabla_{\mathbb S^2}g(y)\bigr)
	\bigl((x-y)\cdot y\bigr)\\
	&=x\cdot\nabla_{\mathbb S^2}g(y).        
	\end{aligned}
	\end{equation}
	Consequently,
	\begin{equation*}
	BR_{\sigma,L}(x)\cdot x
	=
	\frac{A_\sigma}{8\pi R^2}
	\operatorname{p.v.}
	\int_{\mathbb S^2}
	\frac{x\cdot\nabla_{\mathbb S^2}(\Delta_{\mathbb S^2}f(y)+2f(y))}
	{|x-y|^3}
	\,d\sigma(y).
	\end{equation*}
	We now rewrite the integrand as a spherical gradient. Since
	\(\nabla_{\mathbb S^2}g(y)\) is tangent at \(y\), we have the identity
	\[
	x\cdot\nabla_{\mathbb S^2}g(y)
	=
	\bigl(x-(x\cdot y)y\bigr)
	\cdot\nabla_{\mathbb S^2}g(y).
	\]
	On the other hand,
	\begin{equation*}
	\nabla_{\mathbb S^2,y}\frac1{|x-y|}
	=
	\frac{x-(x\cdot y)y}{|x-y|^3}.
	\end{equation*}
	Therefore
	\begin{equation}\label{nabla_newt_g}
	\frac{x\cdot\nabla_{\mathbb S^2}g(y)}{|x-y|^3}
	=
	\nabla_{\mathbb S^2,y}\frac1{|x-y|}
	\cdot
	\nabla_{\mathbb S^2}g(y).    
	\end{equation}
	The preceding identity allows integration by parts to obtain
	\[
	\operatorname{p.v.}
	\int_{\mathbb S^2}
	\nabla_{\mathbb S^2,y}\frac1{|x-y|}
	\cdot
	\nabla_{\mathbb S^2}g(y)
	\,d\sigma(y)
	=
	-
	\int_{\mathbb S^2}
	\frac{\Delta_{\mathbb S^2}g(y)}{|x-y|}
	\,d\sigma(y).
	\]
	Therefore we obtain that
	\begin{equation*}
	BR_{\sigma,L}(x)\cdot x
	=
	-\frac{A_\sigma}{8\pi R^2}
	\int_{\mathbb S^2}
	\frac{
		\Delta_{\mathbb S^2}
		\left(
		\Delta_{\mathbb S^2}f(y)+2f(y)
		\right)
	}
	{|x-y|}
	\,d\sigma(y),
	\end{equation*}
	hence we conclude that
	\begin{equation}\label{linear_surface_tension_term}
	\mathcal L_\sigma[f](x)
	=
	-\frac{A_\sigma}{8\pi R^3}
	\int_{\mathbb S^2}
	\frac{
		\Delta_{\mathbb S^2}
		\left(
		\Delta_{\mathbb S^2}f(y)+2f(y)
		\right)
	}
	{|x-y|}
	\,d\sigma(y)-\frac{1}{R}\dot{c}_L\cdot x.
	\end{equation}

	\subsubsection*{Density-jump contribution}
	We split this term 
	\eqref{linear_operator} into two parts,
	\begin{equation}\label{linear_rho_split}
	\mathcal{L}_\rho=\mathcal{L}_{\rho,1}+\mathcal{L}_{\rho,2},
	\end{equation}
	with
	\begin{equation*}
	\begin{aligned}
	\mathcal{L}_{\rho,1}[f](x)&=\frac1R
	BR_{\rho,L}\cdot x,\\
	\mathcal{L}_{\rho,2}[f](x)&=\frac1R\bigl(BR_0-\dot c_0\bigr)
	\cdot
	\left(2fx-\nabla_{\mathbb S^2}f\right).
	\end{aligned}
	\end{equation*}
	First, we already computed the local term $\mathcal{L}_{\rho,2}$ in \eqref{2ndInLinear_equation},
	\begin{equation}\label{Lrho2_fin}
	\begin{aligned}
	\mathcal{L}_{\rho,2}[f](x)&=\frac{A_\rho}{2R}
	\nabla_{\mathbb S^2}x_3
	\cdot
	\nabla_{\mathbb S^2}f.  
	\end{aligned}
	\end{equation}
	Second, the term $BR_{\rho,L}\cdot x$ contains both $\omega_{\rho,L}$ and the variation of the singular kernel
	against $\omega_0$. We introduce the decomposition in \eqref{BRL_linearization_expanded} and split as follows
	\begin{equation}\label{Lrho1_split}
	\begin{aligned}
	\mathcal{L}_{\rho,1}[f](x)&=\mathcal{L}^1_{\rho,1}+\mathcal{L}^2_{\rho,1}+\mathcal{L}^3_{\rho,1},
	\end{aligned}
	\end{equation}
	\begin{equation*}
	\begin{aligned}
	\mathcal{L}^1_{\rho,1}&=-\frac{x}{4\pi R^3}\cdot
	\operatorname{p.v.}
	\int_{\mathbb S^2}
	\frac{x-y}{|x-y|^3}
	\wedge
	\omega_{\rho,L}(y)
	\,d\sigma(y),
	\end{aligned}
	\end{equation*}
	\begin{equation*}
	\begin{aligned}
	\mathcal{L}^2_{\rho,1}&=
	-\frac{x}{4\pi R^3}\cdot
	\operatorname{p.v.}
	\int_{\mathbb S^2}
	\frac{f(x)x-f(y)y}{|x-y|^3}
	\wedge
	\omega_0(y)
	\,d\sigma(y),
	\end{aligned}
	\end{equation*}
	\begin{equation*}
	\begin{aligned}
	\mathcal{L}^3_{\rho,1}&=\frac{3x}{4\pi R^3}\cdot
	\operatorname{p.v.}
	\int_{\mathbb S^2}
	\frac{
		\bigl(f(x)x-f(y)y\bigr)\cdot(x-y)
	}
	{|x-y|^5}
	(x-y)\wedge\omega_0(y)
	\,d\sigma(y).
	\end{aligned}
	\end{equation*}
	We compute $\mathcal{L}^1_{\rho,1}$ first. Substituting $\omega_{\rho,L}$ \eqref{omega_rho_L}, we have
	\begin{equation*}
	\begin{aligned}
	\mathcal{L}^1_{\rho,1}&=\frac{A_\rho x}{4\pi R}\cdot
	\operatorname{p.v.}
	\!\int_{\mathbb S^2}\!
	\frac{x\!-\!y}{|x\!-\!y|^3}
	\wedge
	\Big(y_3\,y\wedge\nabla_{\mathbb S^2}f(y)
	\!+\!
	2f(y)y\wedge e_3
	\!+\!
	\bigl((y\wedge e_3)\cdot\nabla_{\mathbb S^2}f(y)\bigr)y\Big)
	\,d\sigma(y).
	\end{aligned}
	\end{equation*}
	The first term simplifies as in \eqref{triple_g_op}, while for the second we use the triple-product identity \eqref{triple_product} to obtain
	\begin{equation}\label{aux2}
	x\cdot
	\left[
	(x-y)\wedge
	\left(
	y\wedge e_3
	\right)
	\right]
	=
	x_3-(x\cdot y)y_3.
	\end{equation}
	The last term does not contribute because $x\cdot\left[(x-y)\wedge y\right]=0$.
	Therefore, $\mathcal{L}_{\rho,1}^1$ simplifies to
	\begin{equation*}
	\begin{aligned}
	\mathcal{L}_{\rho,1}^1
	&=
	\frac{A_\rho}{4\pi R}
	\operatorname{p.v.}
	\int_{\mathbb S^2}
	\frac{
		y_3\,x\cdot\nabla_{\mathbb S^2}f(y)
		+
		2f(y)\bigl(x_3-(x\cdot y)y_3\bigr)
	}
	{|x-y|^3}
	\,d\sigma(y).
	\end{aligned}
	\end{equation*}
	Moreover, applying \eqref{nabla_newt_g} with $g=f$ in the first term and $g=y_3$ in the second, we find
	\begin{equation*}
	\begin{aligned}
	\mathcal{L}_{\rho,1}^1&=
	\frac{A_\rho}{4\pi R}
	\operatorname{p.v.}
	\int_{\mathbb S^2}
	\left[
	y_3
	\nabla_{\mathbb S^2,y}\frac{1}{|x-y|}\cdot\nabla_{\mathbb S^2}f(y)
	+
	2f(y)
	\nabla_{\mathbb S^2,y}\frac{1}{|x-y|}\cdot\nabla_{\mathbb S^2}y_3
	\right]
	\,d\sigma(y).
	\end{aligned}
	\end{equation*}
	Integrating by parts on \(\mathbb S^2\) in the first term provides
	\[
	\begin{aligned}
	\operatorname{p.v.}
	\int_{\mathbb S^2}
	y_3
	\nabla_{\mathbb S^2,y}\frac{1}{|x-y|}\cdot\nabla_{\mathbb S^2}f(y)
	\,d\sigma(y)
	&=
	-
	\int_{\mathbb S^2}
	\frac{1}{|x-y|}
	\nabla_{\mathbb S^2}\cdot
	\left(
	y_3\nabla_{\mathbb S^2}f(y)
	\right)
	\,d\sigma(y)
	\\
	&=
	-
	\int_{\mathbb S^2}
	\frac{
		y_3\Delta_{\mathbb S^2}f(y)
		+
		\nabla_{\mathbb S^2}y_3\cdot\nabla_{\mathbb S^2}f(y)
	}
	{|x-y|}
	\,d\sigma(y).
	\end{aligned}
	\]
	In the second term we find
	\[
	\begin{aligned}
	\operatorname{p.v.}
	\int_{\mathbb S^2}
	f(y)
	\nabla_{\mathbb S^2,y}\frac{1}{|x-y|}\cdot\nabla_{\mathbb S^2}y_3
	\,d\sigma(y)
	&=
	-
	\int_{\mathbb S^2}
	\frac{1}{|x-y|}
	\nabla_{\mathbb S^2}\cdot
	\left(
	f(y)\nabla_{\mathbb S^2}y_3
	\right)
	\,d\sigma(y)
	\\
	&=
	-
	\int_{\mathbb S^2}
	\frac{
		\nabla_{\mathbb S^2}f(y)\cdot\nabla_{\mathbb S^2}y_3
		+
		f(y)\Delta_{\mathbb S^2}y_3
	}
	{|x-y|}
	\,d\sigma(y)
	\\
	&=
	\int_{\mathbb S^2}
	\frac{
		2y_3f(y)
		-
		\nabla_{\mathbb S^2}y_3\cdot\nabla_{\mathbb S^2}f(y)
	}
	{|x-y|}
	\,d\sigma(y),
	\end{aligned}
	\]
	where we used \(\Delta_{\mathbb S^2}y_3=-2y_3\) as shown in \eqref{aux_e3}.  Hence
	\begin{equation}\label{Lrho11}
	\mathcal{L}_{\rho,1}^1
	=
	\frac{A_\rho}{4\pi R}
	\int_{\mathbb S^2}
	\frac{
		4y_3f(y)
		-
		y_3\Delta_{\mathbb S^2}f(y)
		-
		3\nabla_{\mathbb S^2}y_3\cdot\nabla_{\mathbb S^2}f(y)
	}
	{|x-y|}
	\,d\sigma(y).
	\end{equation}
	Next, we compute $\mathcal{L}_{\rho,1}^2$.
	Substituting $\omega_0$ \eqref{omega_0_final},
	we have
	\begin{equation*}
	\begin{aligned}
	\mathcal{L}_{\rho,1}^2
	&=\frac{A_\rho}{4\pi R}
	\operatorname{p.v.}
	\int_{\mathbb S^2}
	\frac{
		x\cdot
		\left[
		\bigl(f(x)x-f(y)y\bigr)
		\wedge
		\left(y\wedge e_3\right)
		\right]
	}
	{|x-y|^3}
	\,d\sigma(y).
	\end{aligned}
	\end{equation*}
	By the triple-product identity \eqref{triple_product},
	\begin{equation*}
	\begin{aligned}
	x\cdot
	\left[
	\bigl(f(x)x-f(y)y\bigr)
	\wedge
	\left(y\wedge e_3\right)
	\right]
	&=
	(x\cdot y)
	\bigl(f(x)x_3-f(y)y_3\bigr)
	-
	x_3
	\bigl(f(x)(x\cdot y)-f(y)\bigr)
	\\
	&=
	f(y)\bigl(x_3-(x\cdot y)y_3\bigr),
	\end{aligned}
	\end{equation*}
	and using \eqref{nabla_newt_g} again provides
	\begin{equation*}
	\begin{aligned}
	\mathcal{L}_{\rho,1}^2
	&=
	\frac{A_\rho}{4\pi R}
	\operatorname{p.v.}
	\int_{\mathbb S^2}
	f(y)
	\nabla_{\mathbb S^2,y}\frac{1}{|x-y|}\cdot\nabla_{\mathbb S^2}y_3
	\,d\sigma(y).
	\end{aligned}
	\end{equation*}
	Integrating by parts, we obtain
	\begin{equation}\label{Lrho12}
	\mathcal{L}_{\rho,1}^2
	=
	\frac{A_\rho}{4 \pi R}
	\int_{\mathbb S^2}
	\frac{
		2y_3f(y)
		-
		\nabla_{\mathbb S^2}y_3\cdot\nabla_{\mathbb S^2}f(y)
	}
	{|x-y|}
	\,d\sigma(y).
	\end{equation}
	Finally, we compute $\mathcal{L}_{\rho,1}^3$. Substituting $\omega_0$ \eqref{omega_0_final}, we get
	\begin{equation*}
	\begin{aligned}
	\mathcal{L}_{\rho,1}^3
	&=
	-\frac{3A_\rho}{4\pi R}
	\operatorname{p.v.}
	\int_{\mathbb S^2}
	\frac{
		\bigl(f(x)x-f(y)y\bigr)\cdot(x-y)
	}
	{|x-y|^5}
	x\cdot
	\left[
	(x-y)\wedge(y\wedge e_3)
	\right]
	\,d\sigma(y).
	\end{aligned}
	\end{equation*}
	Computing
	\begin{equation*}
	\bigl(f(x)x-f(y)y\bigr)\cdot(x-y)
	=
	(1-x\cdot y)\bigl(f(x)+f(y)\bigr)
	=
	\frac{|x-y|^2}{2}\bigl(f(x)+f(y)\bigr),
	\end{equation*}
	and using the triple product as in \eqref{aux2}, we can write
	\begin{equation*}
	\begin{aligned}
	\mathcal{L}_{\rho,1}^3
	&=
	-\frac{3A_\rho}{8\pi R}
	\operatorname{p.v.}
	\int_{\mathbb S^2}
	\frac{
		\bigl(f(x)+f(y)\bigr)
		\bigl(x_3-(x\cdot y)y_3\bigr)
	}
	{|x-y|^3}
	\,d\sigma(y).
	\end{aligned}
	\end{equation*}
	We use formula \eqref{nabla_newt_g} to write 
	\begin{equation*}
	\begin{aligned}
	\mathcal{L}_{\rho,1}^3
	&=
	-\frac{3A_\rho}{8\pi R}
	\operatorname{p.v.}
	\int_{\mathbb S^2}
	\big(f(x)+f(y)\big)\nabla_{\mathbb S^2,y}\frac{1}{|x-y|}\cdot\nabla_{\mathbb S^2}y_3
	\,d\sigma(y).
	\end{aligned}
	\end{equation*}
	Finally, we integrate by parts to obtain
	\begin{equation}\label{Lrho13}
	\begin{aligned}
	\mathcal{L}_{\rho,1}^3
	&=
	\frac{3A_\rho}{8\pi R}
	f(x)\int_{\mathbb S^2}
	\frac{\Delta_{\mathbb S^2}y_3}{|x-y|}
	\,d\sigma(y)+\frac{3A_\rho}{8\pi R}
	\int_{\mathbb S^2}\frac{\nabla_{\mathbb S^2}\cdot(f(y)\nabla_{\mathbb S^2}y_3)}{|x-y|}
	\,d\sigma(y)\\
	&=-\frac{3A_\rho}{R}f(x)\mathcal{G}[y_3](x)
	+\frac{3A_\rho}{8\pi R}
	\int_{\mathbb S^2}\frac{\nabla_{\mathbb S^2}y_3\cdot\nabla_{\mathbb S^2}f(y)-2y_3f(y)}{|x-y|}
	\,d\sigma(y)\\
	&=-\frac{A_\rho}{R}f(x)x_3
	+\frac{3A_\rho}{8\pi R}
	\int_{\mathbb S^2}\frac{\nabla_{\mathbb S^2}y_3\cdot\nabla_{\mathbb S^2}f(y)-2y_3f(y)}{|x-y|}
	\,d\sigma(y),
	\end{aligned}
	\end{equation}
	where we use that $\mathcal{G}[y_3](x)=\frac13x_3$ by \eqref{newtonian_spectral}.
	
	Combining the expressions \eqref{Lrho11}, \eqref{Lrho12}, and \eqref{Lrho13} back in \eqref{Lrho1_split}, and using the notation \eqref{newtonian_operator}, we have
	\begin{equation*}
	\begin{aligned}
	\mathcal{L}_{\rho,1}[f](x)&= 
	-\frac{A_\rho}{R}x_3f(x)
	+
	\frac{A_\rho}{R}\mathcal G[
	3y_3f
	-
	y_3\Delta_{\mathbb S^2}f
	-
	\frac52
	\nabla_{\mathbb S^2}y_3\cdot\nabla_{\mathbb S^2}f](x).
	\end{aligned}
	\end{equation*}
    Combining this  identity with \eqref{Lrho2_fin} back in \eqref{linear_rho_split} gives 
	\begin{equation}\label{linear_rho}
	\begin{aligned}
	\mathcal{L}_{\rho}[f](x)&= \frac{A_\rho}{R}(-x_3f(x)
	+\frac12\nabla_{\mathbb S^2}x_3
	\cdot
	\nabla_{\mathbb S^2}f(x))\\
	&\quad+
	\frac{A_\rho}{R}\mathcal G
	[
	3y_3f
	-
	y_3\Delta_{\mathbb S^2}f
	-
	\frac52
	\nabla_{\mathbb S^2}y_3\cdot\nabla_{\mathbb S^2}f](x).
	\end{aligned}
	\end{equation}

	\subsubsection*{The linearized equation}
	
	Finally, introducing the above identity  together with \eqref{linear_surface_tension_term} back into \eqref{linear_operator}, we obtain the linear operator
\begin{multline}\label{linearized_final}
        \mathcal{L}[f](x)
	=
	-\frac{3A_\sigma}{4\pi R^3}\int_{\mathbb{S}^2}f(y)\,x\cdot y \,d\sigma(y)
	-
	\frac{A_\sigma}{2R^3}
	\mathcal{G}[\Delta_{\mathbb S^2}
	\left(
	\Delta_{\mathbb S^2}f+2f
	\right)](x)
    \\
+
	\frac{A_\rho}{R}
	(
	-x_3f(x)
	+
	\frac12\nabla_{\mathbb S^2}x_3\cdot\nabla_{\mathbb S^2}f(x)+
	\mathcal{G}[
	3y_3f
	-
	y_3\Delta_{\mathbb S^2}f
	-
	\frac52
	\nabla_{\mathbb S^2}y_3\cdot\nabla_{\mathbb S^2}f
	](x)),
    \end{multline}
		where we have used the expression of $c_L$ in \eqref{c0cL_def}.

	\section{Spectral form of the linearized operator}
	\label{sec:spectral}
	
	In this section we compute the action of the linearized operator
	\eqref{linearized_final} on spherical harmonics.  
	From the product formula for the spherical Laplacian, we have that
	\begin{equation*}
	\Delta_{\mathbb S^2}(x_3Y_{\ell,m})
	=
	x_3\Delta_{\mathbb S^2}Y_{\ell,m}
	+
	Y_{\ell,m}\Delta_{\mathbb S^2}x_3
	+
	2\nabla_{\mathbb S^2}x_3\cdot\nabla_{\mathbb S^2}Y_{\ell,m}.
	\end{equation*}
	Then, using the identities $\Delta_{\mathbb S^2}x_3=-2x_3$ and \eqref{LaplaceSPH}, we have
	\begin{equation*}
	\Delta_{\mathbb S^2}(x_3Y_{\ell,m})
	=
	-(\ell(\ell+1)+2)x_3 Y_{\ell,m}
	+
	2\nabla_{\mathbb S^2}x_3\cdot\nabla_{\mathbb S^2}Y_{\ell,m}.
	\end{equation*}
	On the other hand, by \eqref{x3Y_recurrence} and \eqref{LaplaceSPH},
	\begin{equation*}
	\Delta_{\mathbb S^2}(x_3Y_{\ell,m})
	=\Delta_{\mathbb S^2}( b_{\ell,m}^{+}Y_{\ell+1,m} +  b_{\ell,m}^{-}Y_{\ell-1,m}) =  -(\ell+1)(\ell+2) b_{\ell,m}^{+}Y_{\ell+1,m} - \ell (\ell-1)b_{\ell,m}^{-}Y_{\ell-1,m}.
	\end{equation*}
	Combining the two identities above and using again \eqref{x3Y_recurrence}, we obtain that
	\begin{equation}\label{gradx3gradY_recurrence}
	\nabla_{\mathbb S^2}x_3\cdot\nabla_{\mathbb S^2}Y_{\ell,m}
	=
	-\ell b_{\ell,m}^{+}Y_{\ell+1,m}
	+
	(\ell+1)b_{\ell,m}^{-}Y_{\ell-1,m}.
	\end{equation}
	We first consider the surface-tension part. Since
	$$
	\Delta_{\mathbb S^2}
	\big(\Delta_{\mathbb S^2}Y_{\ell,m}+2Y_{\ell,m}\big)
	=
	\ell(\ell+1)(\ell-1)(\ell+2)Y_{\ell,m},
	$$
	we obtain from \eqref{newtonian_spectral} and \eqref{linearized_final} that
	\begin{equation}\label{Lsigma_spectral_action}
	\mathcal L_{\sigma}[Y_{\ell,m}](x)
	=
	-\frac{A_{\sigma}}{2R^3}
	\frac{\ell(\ell+1)(\ell-1)(\ell+2)}{2\ell+1}
	Y_{\ell,m}-\frac{3A_\sigma}{4\pi R^3}\int_{\mathbb{S}^2}Y_{\ell,m}(y)\,x\cdot y \,d\sigma(y).
	\end{equation}
	We now compute the density-jump part given in \eqref{linear_rho}.  Acting on one mode $Y_{\ell,m}$, the
	local density-jump terms are, by \eqref{x3Y_recurrence} and
	\eqref{gradx3gradY_recurrence},
	\begin{equation}\label{density_local_spectral}
	\begin{aligned}
	-x_3Y_{\ell,m}
	+\frac12\nabla_{\mathbb S^2}x_3\cdot\nabla_{\mathbb S^2}Y_{\ell,m}
	&=-\frac{\ell+2}{2}b_{\ell,m}^{+}Y_{\ell+1,m}
	+\frac{\ell-1}{2}b_{\ell,m}^{-}Y_{\ell-1,m}.
	\end{aligned}
	\end{equation}
	Next, for the nonlocal part, we have, using again \eqref{x3Y_recurrence} and \eqref{gradx3gradY_recurrence},
	\begin{equation}\label{density_nonlocal_input_spectral}
	\begin{aligned}
	3x_3Y_{\ell,m}
	&-x_3\Delta_{\mathbb S^2}Y_{\ell,m}
	-\frac52\nabla_{\mathbb S^2}x_3\cdot\nabla_{\mathbb S^2}Y_{\ell,m}\\
	&=   \frac{(2\ell+3)(\ell+2)}{2}b_{\ell,m}^{+}Y_{\ell+1,m}
	+
	\frac{(2\ell-1)(\ell-1)}{2}b_{\ell,m}^{-}Y_{\ell-1,m}.    
	\end{aligned}
	\end{equation}
	Therefore, applying \eqref{newtonian_spectral} gives
	\begin{equation}\label{density_nonlocal_spectral}
	\begin{aligned}
	\mathcal G\Big(
	3x_3Y_{\ell,m}
	-x_3\Delta_{\mathbb S^2}Y_{\ell,m}
	-\frac52\nabla_{\mathbb S^2}x_3\cdot\nabla_{\mathbb S^2}Y_{\ell,m}
	\Big)
	&=
	\frac{\ell+2}{2}b_{\ell,m}^{+}Y_{\ell+1,m}
	+
	\frac{\ell-1}{2}b_{\ell,m}^{-}Y_{\ell-1,m}.
	\end{aligned}
	\end{equation}
	The $Y_{\ell+1,m}$ contributions in \eqref{density_local_spectral} and
	\eqref{density_nonlocal_spectral} cancel. Therefore, recalling \eqref{linear_rho}, we obtain 
	\begin{equation}\label{Lrho_spectral_action}
	\mathcal L_{\rho}[Y_{\ell,m}](x)
	=
	\frac{A_{\rho}}{R}(\ell-1)b_{\ell,m}^{-}Y_{\ell-1,m}.
	\end{equation}
	Combining \eqref{Lsigma_spectral_action} and \eqref{Lrho_spectral_action} back in \eqref{linearized_final}, we have
	\begin{equation*}
	\begin{aligned}
	\mathcal L[Y_{\ell,m}](x)
	&=-\frac{3A_\sigma}{4\pi R^3}\int_{\mathbb{S}^2}Y_{\ell,m}(y)\,x\cdot y \,d\sigma(y)
	-\frac{A_{\sigma}}{2R^3} \frac{(\ell-1)\ell(\ell+1)(\ell+2)}{2\ell+1} Y_{\ell,m}\\
	&\quad+
	\frac{A_{\rho}}{R}(\ell-1)b_{\ell,m}^{-}Y_{\ell-1,m}
	\end{aligned} 
	\end{equation*}
	Finally, the term $\dot{c}_L$ in \eqref{c0cL_def} was chosen precisely so that
	\begin{equation*}
	\begin{aligned}
	\dot{c}_L(t)\cdot x =\frac{A_\sigma}{R^2}\text{proj}_1f(x,t),
	\end{aligned}
	\end{equation*}
	hence
	\begin{equation*}
	\begin{aligned}
	-\frac{3A_\sigma}{4\pi R^3}\int_{\mathbb{S}^2}Y_{\ell,m}(y)\,x\cdot y \,d\sigma(y) = -\frac{A_\sigma}{R^3}Y_{1,m}(x)\delta_{\ell,1}.
	\end{aligned}
	\end{equation*}
	Therefore, the linear operator in spectral form is given by
	\begin{equation*}
	\begin{aligned}
	\widehat{\mathcal L[f]}_{\ell,m}
	&=-\frac{A_\sigma}{R^3}\widehat{f}_{\ell,m}\delta_{\ell,1}
	-\frac{A_{\sigma}}{2R^3} \frac{(\ell-1)\ell(\ell+1)(\ell+2)}{2\ell+1} \widehat{f}_{\ell,m}+
	\frac{A_{\rho}}{R}\ell b_{\ell+1,m}^{-}\widehat{f}_{\ell+1,m},
	\end{aligned} 
	\end{equation*}
	which combines to
	\begin{equation*}
	\begin{aligned}
	\widehat{\mathcal L[f]}_{\ell,m}
	&=-\frac{A_{\sigma}}{2R^3} \frac{(\ell-1+\delta_{\ell,1})\ell(\ell+1)(\ell+2)}{2\ell+1} \widehat{f}_{\ell,m}+
	\frac{A_{\rho}}{R}\ell b_{\ell+1,m}^{-}\widehat{f}_{\ell+1,m}.
	\end{aligned} 
	\end{equation*}
	It will be convenient to consider the following rescaled version of the linear operator, which has eigenvalues independent of the physical constants:
	\begin{equation}\label{linear_spectral}
	\begin{aligned}
	\widehat{\mathcal L_a[f]}_{\ell,m}
	&=\frac{2R^3}{A_\sigma}\widehat{\mathcal L[f]}_{\ell,m}=-\frac{(\ell-1+\delta_{\ell,1})\ell(\ell+1)(\ell+2)}{2\ell+1} \widehat{f}_{\ell,m}+
	\frac{A_{\rho}R^{2}}{A_{\sigma}}2\ell b_{\ell+1,m}^{-}\widehat{f}_{\ell+1,m}.
	\end{aligned} 
	\end{equation}

	\section{Semigroup estimates}\label{sec:semigroup}
	
	\subsection{Diagonalization of $\mathcal{L}_a$}
	
	In this subsection, we will diagonalize the linear operator $\mathcal{L}_a$, defined in spectral form by \eqref{linear_spectral}.

	If we define
	\begin{equation}\label{def:akbk}
	a_{k} = \frac{(k-1+\delta_{k,1})k(k+1)(k+2)}{2k+1}, \quad b_{k}^{(m)} =   \frac{A_{\rho}R^{2}}{A_{\sigma}}2k b_{k+1,m}^{-},
	\end{equation}
	where $b_{k,m}^-$ is given in \eqref{a_coefficients_pm},
	then the linear operator \eqref{linear_spectral} can be written as 
	\begin{equation}\label{La_def}
	\widehat{\mathcal{L}_a[f]}_{k,m} = \sum_{j \geq 0} L^{(m)}_{k,j}\widehat{f}_{j,m},
	\end{equation}
	where
	\begin{equation}\label{Ldef_kj}
	L^{(m)}_{k,j}=
	\begin{cases}
	-a_{k},     & \text{if }  j=k,\\
	b_{k}^{(m)}, & \text{if } j=k+1,\\
	0, & \text{otherwise}. 
	\end{cases}
	\end{equation}
	We now diagonalize the system given by \eqref{Ldef_kj}. To do so, define 
	\begin{equation*}
	S : \ell^{2}\rightarrow \ell^{2},\quad S^{-1} : \ell^{2}\rightarrow \ell^{2},
	\end{equation*}
	by
	\begin{equation}\label{Skj_def}
	S^{(m)}_{k,k+j}=
	\begin{cases}
	\displaystyle \prod_{\ell=1}^{j} \frac{b^{(m)}_{k-1+\ell}}{a_{k-1+\ell} - a_{k+j}},  & \text{if }  j\geq 0,\\
	0, & \text{otherwise}, 
	\end{cases}
	\end{equation}
	and  $S^{-1}$ by
	\begin{equation}\label{Skjinverse_def}
	(S^{-1})^{(m)}_{k,k+j}=
	\begin{cases} \displaystyle 
	(-1)^{j}\prod_{\ell=1}^{j}   \frac{b^{(m)}_{k-1+\ell}}{a_{k} -  a_{k+\ell}},    & \text{if }  j\geq 0,\\
	0, & \text{otherwise}. 
	\end{cases}
	\end{equation}
    The operators $S$ and $S^{-1}$ define  the operators $\mathcal{S}$ and $\mathcal{S}^{-1}$ spectrally by
    \begin{equation}\label{def:mathcalS}
	\widehat{\mathcal{S}[f]}_{k,m} = \sum_{j\geq 0} S^{(m)}_{k,j}\widehat{f}_{j,m},\quad \widehat{\mathcal{S}^{-1}[f]}_{k,m} = \sum_{j\geq 0} (S^{-1})^{(m)}_{k,j}\widehat{f}_{j,m}.
	\end{equation}
    We note that $b_0^{(m)}=0$, hence $S_{0,N}^{(m)}=(S^{-1})_{0,N}^{(m)}=0$ for $N\geq1$.
    In particular, $\mathcal{S}$ and $\mathcal{S}^{-1}$ preserve the mean-zero condition.

	To control the diagonalizing operator $\mathcal{S}$  we first obtain appropriate bounds on the spectral coefficients $S_{k,j}^{(m)}$.
	
	\begin{lemma}\label{Skj_lem}
		Let $S_{\ell,N}^{(m)}$ be defined in \eqref{Skj_def}. Then, for $N> \ell+1$,
		\begin{equation*}
		\sup_{\ell\sim 2^{j}, N\sim 2^{k}} |S^{(m)}_{\ell, N}| \lesssim C\Big(\frac{|A_\rho|R^2}{A_\sigma}\Big) 2^{-2k}e^{-c2^{k-j}}.
		\end{equation*}
	\end{lemma}
	\begin{proof} 
		We first rewrite $S_{\ell,N}^{m}$ in \eqref{Skj_def},
		\begin{equation*}
		\begin{aligned}
		S_{\ell,N}^{(m)}&=\prod_{j=1}^{N-\ell}\frac{b_{\ell-1+j}^{(m)}}{a_{\ell-1+j}-a_N}=\prod_{j=\ell}^{N-1}\frac{b_{j}^{(m)}}{a_{j}-a_N}.
		\end{aligned}
		\end{equation*}
		Denote $M=N-\ell$ and change the index again by $r=N-j$ to get
		\begin{equation*}
		\begin{aligned}
		S_{\ell,N}^{(m)}&=\prod_{r=1}^{M}\frac{b_{N-r}^{(m)}}{a_{N-r}-a_N}.
		\end{aligned}
		\end{equation*}
		Using the definition of $a_k$ and $b_k$ \eqref{def:akbk}, we have
		\begin{equation*}
		a_N-a_{N-r}\gtrsim rN^2, \qquad |b_{N-r}^{(m)}|\lesssim \frac{|A_\rho|R^2}{A_\sigma}N,
		\end{equation*}
		and we obtain that
		\begin{equation}\label{S_lN_bound}
		\begin{aligned}
		|S_{\ell,N}^{(m)}|&\leq \Big(\frac{|A_\rho|R^2}{A_\sigma}\Big)^{M}\frac{C^{M}}{M!N^M} \lesssim  \frac{\exp{\big(C\frac{|A_\rho|R^2}{A_\sigma}\big)}}{N^{M}}.
		\end{aligned}
		\end{equation}
		Therefore, we have
		\begin{equation*}
		\begin{aligned}
		|S_{\ell,N}^{(m)}|&\lesssim C\Big(\frac{|A_\rho|R^2}{A_\sigma}\Big)N^{-(N-\ell)}= C\Big(\frac{|A_\rho|R^2}{A_\sigma}\Big)N^{-\ell(\frac{N}{\ell}-1-\frac{2}{\ell})}N^{-2}\\
        &= C\Big(\frac{|A_\rho|R^2}{A_\sigma}\Big)e^{-\ell(\frac{N}{\ell}-1-\frac{2}{\ell})\log{N}}N^{-2},
		\end{aligned}
		\end{equation*}
		and for $N\sim 2^k$, $\ell\sim 2^j$, $N>\ell+1$, 
		\begin{equation*}
		\begin{aligned}
		|S_{\ell\sim 2^j,N\sim 2^k}^{(m)}|&\lesssim C\Big(\frac{|A_\rho|R^2}{A_\sigma}\Big)e^{-\frac{N}{\ell}}2^{-2k}\lesssim C\Big(\frac{|A_\rho|R^2}{A_\sigma}\Big) 2^{-2k}e^{-c2^{k-j}}.
		\end{aligned}
		\end{equation*}
	\end{proof}	
	
	Next, we notice that the operator $\mathcal{S}$ \eqref{Skj_def} is upper triangular. Let us write it as follows:
	\begin{equation}\label{S_superdiag}
	\mathcal{S}=I+\mathcal{K}=I+\mathcal{K}_1+\sum_{r\geq2}\mathcal{K}_r=I+\mathcal{K}_1+\mathcal{K}_{\geq2},
	\end{equation}
	where $I$ denotes the identity and $\mathcal{K}_r$ the operator given by the $r$-superdiagonal.
	In the following estimates, the first superdiagonal will need a separate treatment.
	Its explicit expression is
	\begin{equation*}
	\begin{aligned}
	\mathcal{K}_1[f]=\sum_{\ell\geq 2}\sum_{|m|\leq \ell}S^{(m)}_{\ell-1,\ell}\widehat{f}_{\ell,m}Y_{\ell-1,m},
	\end{aligned}
	\end{equation*}
	with 
	\begin{equation*}
	\begin{aligned}
	S^{(m)}_{\ell-1,\ell}=\frac{b_{\ell-1}^{(m)}}{a_{\ell-1}-a_{\ell}}=\frac{A_\rho R^2}{A_\sigma}\frac{2(\ell-1)}{a_{\ell-1}-a_{\ell}}b_{\ell,m}^-,
	\end{aligned}
	\end{equation*}
	and $b_{\ell,m}^-$ defined in \eqref{a_coefficients_pm}.
	Equivalently, we have
	\begin{equation*}
	\begin{aligned}
	\mathcal{K}_1[Y_{\ell,m}]=\frac{A_\rho R^2}{A_\sigma}\frac{2(\ell-1)}{a_{\ell-1}-a_{\ell}}b_{\ell,m}^-Y_{\ell-1,m}.
	\end{aligned}
	\end{equation*}
	We notice that, due to the dependence in $m$, it is not a multiplier and we cannot apply Lemma \ref{lem:kernel_bounds} directly to obtain bounds. However, undoing the steps in Section \ref{sec:spectral}, we can find a more suitable expression. In fact, 
	combining \eqref{density_local_spectral} and \eqref{density_nonlocal_input_spectral}, we get
	\begin{equation*}
	\begin{aligned}
	\Big(\frac{\ell-1}{2}+\frac{(2\ell-1)(\ell-1)}{2(2\ell+3)}\Big)b_{\ell,m}^{-}Y_{\ell-1,m}&=-x_3Y_{\ell,m}
	+\frac12\nabla_{\mathbb S^2}x_3\cdot\nabla_{\mathbb S^2}Y_{\ell,m}\\
	&\quad+\frac{1}{2\ell+3}\Big(3x_3Y_{\ell,m}
	-x_3\Delta_{\mathbb S^2}Y_{\ell,m}
	-\frac52\nabla_{\mathbb S^2}x_3\cdot\nabla_{\mathbb S^2}Y_{\ell,m}\Big),
	\end{aligned}
	\end{equation*}
	therefore, for $\ell\geq2$, 
	\begin{equation*}
	\begin{aligned}
	b_{\ell,m}^{-}Y_{\ell-1,m}&=\frac{2\ell+3}{(\ell-1)(2\ell+1)}\Big(-x_3Y_{\ell,m}
	+\frac12\nabla_{\mathbb S^2}x_3\cdot\nabla_{\mathbb S^2}Y_{\ell,m}\Big)\\
	&\quad+\frac{1}{(\ell-1)(2\ell+1)}\Big(3x_3Y_{\ell,m}
	-x_3\Delta_{\mathbb S^2}Y_{\ell,m}
	-\frac52\nabla_{\mathbb S^2}x_3\cdot\nabla_{\mathbb S^2}Y_{\ell,m}\Big).
	\end{aligned}
	\end{equation*}
	Defining the following multiplier operators, of order $-2$ and $-3$, respectively, by $\mathcal{K}_{1,1}[Y_{\ell,m}]=\mathcal{K}_{1,2}[Y_{\ell,m}]=0$ for $\ell=0,1$, and for $\ell\geq2$,
	\begin{equation*}
	\begin{aligned}
	\mathcal{K}_{1,1}[Y_{\ell,m}]=2\frac{2\ell+3}{(2\ell+1)(a_{\ell-1}-a_{\ell})}Y_{\ell,m},\quad
	\mathcal{K}_{1,2}[Y_{\ell,m}]=\frac{2}{(2\ell+1)(a_{\ell-1}-a_{\ell})}Y_{\ell,m},
	\end{aligned}
	\end{equation*}
	we have
	\begin{equation*}
	\begin{aligned}
	\mathcal{K}_1[f]&=\frac{A_\rho R^2}{A_\sigma}\Big(-x_3\mathcal{K}_{1,1}[f]
	+\frac12\nabla_{\mathbb S^2}x_3\cdot\nabla_{\mathbb S^2}\mathcal{K}_{1,1}[f]\\
	&\hspace{3.5cm}+3x_3\mathcal{K}_{1,2}[f]
	-x_3\Delta_{\mathbb S^2}\mathcal{K}_{1,2}[f]
	-\frac52\nabla_{\mathbb S^2}x_3\cdot\nabla_{\mathbb S^2}\mathcal{K}_{1,2}[f]\Big).
	\end{aligned}
	\end{equation*}
	We thus obtain that $\mathcal{K}_1$ is bounded in $L^\infty$, which will be used in subsequent lemmas:
	\begin{equation}\label{K1_bound}
	\begin{aligned}
	\|\mathcal{K}_1[f]\|_{L^\infty}&\lesssim \frac{|A_\rho| R^2}{A_\sigma}\Big(\|\mathcal{K}_{1,1}[f]\|_{W^{1,\infty}}
	+\|\mathcal{K}_{1,2}[f]\|_{W^{2,\infty}}\Big)\lesssim \frac{|A_\rho| R^2}{A_\sigma}\|f\|_{L^\infty},
	\end{aligned}
	\end{equation}
	where Corollary \ref{cor:multiplier} was used in the last step.

	\begin{prop}\label{prop_S_deltak} 
		The operators  $\mathcal{K}_1$ and $\mathcal{K}_{\geq2}$ defined in \eqref{S_superdiag} satisfy that 
		\begin{equation*}
		\begin{aligned}
		\|\Delta_{j}\mathcal{K}_1\Delta_{k}\|_{L^{\infty}\rightarrow L^{\infty}} &\lesssim \frac{|A_\rho|R^2}{A_\sigma},\\
		\|\Delta_{j}\mathcal{K}_{\geq2}\Delta_{k}\|_{L^{\infty}\rightarrow L^{\infty}} &\lesssim C\Big(\frac{|A_\rho|R^2}{A_\sigma}\Big) e^{-c2^{k-j}}.
		\end{aligned}
		\end{equation*}
	\end{prop}
	\begin{proof}
		The first estimate follows directly from \eqref{K1_bound}.
		Next, using \eqref{LPD}, we have that
		\begin{equation*}
		\begin{aligned}
		\Delta_j \mathcal{K}_{\geq2}[\Delta_k f] (x)&=\Delta_j \mathcal{K}_{\geq2} \Big[\sum_{n\geq 0}\phi_k(n)\text{proj}_n f(x)\Big]=\Delta_j \mathcal{K}_{\geq2} \Big[\sum_{n\geq 0}\phi_k(n)\sum_{|m|\leq n} \widehat{f}_{n,m}Y_{n,m}(x)\Big]\\
		&=\Delta_j \Big(\sum_{n\geq 0}\phi_k(n)\sum_{|m|\leq n} \widehat{f}_{n,m}\mathcal{K}_{\geq2}[Y_{n,m}](x)\Big).
		\end{aligned}
		\end{equation*}
		Using the definition of $\mathcal{S}$ in \eqref{def:mathcalS},
		\begin{equation*}
		\widehat{\mathcal{S}[Y_{n,m}]}_{p,r} = \sum_{j\geq 0} S_{p,j}^{(r)}\delta_{j,n}\delta_{r,m} = S_{p,n}^{(m)}\delta_{r,m},
		\end{equation*}
		and hence
		\begin{equation*}
		\mathcal{S}[Y_{n,m}] = \sum_{p \geq 0}\sum_{|r|\leq p}S_{p,n}^{(r)}\delta_{r,m} Y_{p,r} = \sum_{p \leq n}S_{p,n}^{(m)}Y_{p,m}.
		\end{equation*}
		Therefore, we obtain that
		\begin{equation*}
		\begin{aligned}
		\Delta_j \mathcal{K}_{\geq2}[\Delta_k f](x)&=\Delta_j \Big(\sum_{n\geq 0}\phi_k(n)\sum_{|m|\leq n} \widehat{f}_{n,m}\sum_{p\leq n-2}S^{(m)}_{p,n}Y_{p,m}(x)\Big)\\
		&=\sum_{q\geq0}\phi_j(q)\sum_{n\geq 0}\phi_k(n)\sum_{|m|\leq n} \widehat{f}_{n,m}\sum_{p\leq n-2}S^{(m)}_{p,n}\text{proj}_q(Y_{p,m})(x).
		\end{aligned}
		\end{equation*}
		By definition, $\text{proj}_q(Y_{p,m})=\delta_{p,q}Y_{p,m}$, therefore
		\begin{equation*}
		\begin{aligned}
		\Delta_j \mathcal{K}_{\geq2}[\Delta_k f](x)&=\sum_{q\geq0}\phi_j(q)\sum_{n\geq q+2}\phi_k(n)\sum_{|m|\leq n} \widehat{f}_{n,m}S^{(m)}_{q,n}Y_{q,m}(x).
		\end{aligned}
		\end{equation*}
		Next, by Cauchy--Schwarz and Lemma \ref{Skj_lem},
		\begin{equation*}
		\begin{aligned}
		|\Delta_j &\mathcal{K}_{\geq2}[\Delta_k f]|\\
		&\lesssim C\Big(\frac{|A_\rho|R^2}{A_\sigma}\Big)\sum_{q\geq0}\phi_j(q)\sum_{n\geq q+2}\phi_k(n) 2^{-2k}e^{-2c2^{k-j}}\Big(\sum_{|m|\leq n}|\widehat{f}_{n,m}|^2\Big)^{\frac12}\Big(\sum_{|m|\leq q}|Y_{q,m}(x)|^2\Big)^{\frac12},
		\end{aligned}
		\end{equation*}
		and using the addition formula for the spherical harmonics \eqref{Addition_formula}
		\begin{equation*}
		\begin{aligned}
		|\Delta_j \mathcal{K}_{\geq2}[\Delta_k f]|&\lesssim  C\Big(\frac{|A_\rho|R^2}{A_\sigma}\Big)\sum_{q\geq0}\phi_j(q)\sum_{n\geq q+2}\phi_k(n) 2^{-2k}e^{-2c2^{k-j}}2^{\frac{j}{2}}\|\text{proj}_n f\|_{L^2}.
		\end{aligned}
		\end{equation*}
		We conclude that
		\begin{equation*}
		\begin{aligned}
		|\Delta_j \mathcal{K}_{\geq2}[\Delta_k f]|&\lesssim C\Big(\frac{|A_\rho|R^2}{A_\sigma}\Big)2^{-2k}2^{\frac{j}{2}}e^{-2c2^{k-j}}\sum_{n\geq0}\phi_k(n)\|\text{proj}_n f\|_{L^2}\sum_{0\leq q\leq n-2}\phi_j(q) \\
		&\lesssim C\Big(\frac{|A_\rho|R^2}{A_\sigma}\Big)2^{-2k}2^{\frac{3}{2}j}e^{-2c2^{k-j}}\sum_{n\geq0}\phi_k(n)\|\text{proj}_n f\|_{L^2},
        	\end{aligned}
		\end{equation*}
        and using Cauchy--Schwarz inequality,
        \begin{equation*}
		\begin{aligned}
        |\Delta_j \mathcal{K}_{\geq2}[\Delta_k f]|&\lesssim C\Big(\frac{|A_\rho|R^2}{A_\sigma}\Big)2^{-2k}2^{\frac{3}{2}j}e^{-2c2^{k-j}}2^{\frac{k}{2}}\Big(\sum_{n\geq 0}\phi_k(n)^2\|\text{proj}_n f\|_{L^2}^2\Big)^{\frac12}\\
        &\lesssim C\Big(\frac{|A_\rho|R^2}{A_\sigma}\Big)2^{-\frac32(k-j)}e^{-2c2^{k-j}}\|\Delta_k f\|_{L^2}\\
        &\lesssim C\Big(\frac{|A_\rho|R^2}{A_\sigma}\Big)e^{-c2^{k-j}}\|\Delta_k f\|_{L^\infty}.
		\end{aligned}
		\end{equation*}

	\end{proof}
	
	\begin{prop}\label{prop_S_holder}

		Let $m\in\mathbb{Z}$ be a nonnegative integer, $\alpha\in(0,1)$,  $f\in C^{m,\alpha}(\mathbb{S}^2)$, and $\mathcal{S}$ the operator defined in \eqref{def:mathcalS}.
		Then, 
		\begin{equation*}
		\begin{aligned}
		\|\mathcal{S}[f]\|_{L^\infty} &\lesssim C\Big(\frac{|A_\rho|R^2}{A_\sigma}\Big) \|f\|_{L^\infty},\quad \|\mathcal{S}[f]\|_{W^{1,\infty}} \lesssim C\Big(\frac{|A_\rho|R^2}{A_\sigma}\Big) \|f\|_{W^{1,\infty}},\\
		\|\mathcal{S}[f]\|_{C^{m,\alpha}} &\lesssim C\Big(\frac{|A_\rho|R^2}{A_\sigma}\Big) \|f\|_{C^{m,\alpha}}.  
		\end{aligned}
		\end{equation*}
	\end{prop}
	
	\begin{proof}
		Let $f=\sum_{k\geq0} \Delta_k f$. Then,
		\begin{equation*}
		\begin{aligned}
		\Delta_j(\mathcal{S}[f])=\sum_{k\geq 0} \Delta_j(\mathcal{S}\Delta_k f)=\sum_{k\geq 0} (\Delta_j\mathcal{S}\Delta_k) f.
		\end{aligned}
		\end{equation*}
		The definition of $\mathcal{S}$ in \eqref{Skj_def} shows that it is a triangular operator, hence we have
		\begin{equation*}
		\begin{aligned}
		\Delta_j(\mathcal{S}[f])=\sum_{k\geq j-1} (\Delta_j\mathcal{S}\Delta_k) f=\sum_{k\geq j-1} (\Delta_j\mathcal{S}\Delta_k) \tilde{\Delta}_k f,
		\end{aligned}
		\end{equation*}
		where the projector $\tilde{\Delta}_k$ has a slightly bigger support. Then, 
		\begin{equation*}
		\begin{aligned}
		2^{j(m+\alpha)}\|\Delta_j \mathcal{S}[f]\|_{L^\infty}
		\le
		2^{j(m+\alpha)}\sum_{k\geq j-1}\|\Delta_j \mathcal{S}\Delta_k\|_{L^\infty\to L^\infty}\|\tilde{\Delta}_k f\|_{L^\infty},    
		\end{aligned}
		\end{equation*}
		so  
		\begin{equation*}
		\begin{aligned}
		2^{j(m+\alpha)}\|\Delta_j \mathcal{S}[f]\|_{L^\infty}
		\lesssim 
		\sum_{k\geq j-1} 2^{(j-k)(m+\alpha)}\,
		\|\Delta_j \mathcal{S}\Delta_k\|_{L^\infty\to L^\infty}\,
		\|f\|_{B_{\infty,\infty}^{(m+\alpha)}}.   
		\end{aligned}
		\end{equation*}
		Hence, by Proposition 
		\ref{prop_S_deltak} and \eqref{S_superdiag}, we obtain
		\begin{equation*}
		\begin{aligned}
		\|\mathcal{S}[f]\|_{C^{m,\alpha}}&\lesssim \sup_{j\geq0}2^{j(m+\alpha)}\|\Delta_j \mathcal{S}[f]\|_{L^\infty}\lesssim \sup_{j\geq0}2^{j(m+\alpha)}\sum_{k\geq0}\|\Delta_j \mathcal{S}\Delta_kf\|_{L^\infty}\\
		&\lesssim  C\Big(\frac{|A_\rho|R^2}{A_\sigma}\Big)\|f\|_{B^{m+\alpha}_{\infty,\infty}}
		\sup_{j\geq0}\sum_{k\geq j-1} 2^{(j-k)(m+\alpha)}
		\lesssim  C\Big(\frac{|A_\rho|R^2}{A_\sigma}\Big)\|f\|_{C^{m,\alpha}}. 
		\end{aligned}
		\end{equation*}
		Next, we prove the $L^\infty$ estimate. We use the expression of $\mathcal{S}$ as the sum of the superdiagonals \eqref{S_superdiag}.
		It can be seen that $\mathcal{K}_r$ is an operator of order $-r$, and thus $\mathcal{K}_{\geq2}$ is an operator of order $-2$. Indeed, let us first show that for $r\geq2$, $\mathcal{K}_r:L^2(\mathbb{S}^2)\to H^2(\mathbb{S}^2)$, which implies that $\mathcal{K}_r:L^\infty(\mathbb{S}^2)\to L^\infty(\mathbb{S}^2)$. 
		From \eqref{S_lN_bound}, we have
		\begin{equation*}
		\begin{aligned}
		|S_{\ell,\ell+r}^{(m)}|&\leq \Big(\frac{|A_\rho|R^2}{A_\sigma}\Big)^{r}\frac{C^{r}}{r!(\ell+r)^r},
		\end{aligned}
		\end{equation*}
		hence
		\begin{equation*}
\begin{aligned}\|\mathcal{K}_r[f]\|_{H^2}^2&\lesssim\sum_{\ell\geq1}\sum_{|m|\leq \ell} (1+\ell)^4|S_{\ell,\ell+r}^{(m)}|^2|\widehat{f}_{\ell+r,m}|^2\lesssim \frac{C^{2r}}{r!^2}\Big(\frac{|A_\rho|R^2}{A_\sigma}\Big)^{2r}\sum_{\ell\geq1}\sum_{|m|\leq \ell}\frac{(1+\ell)^4}{(\ell+r)^{2r}}|\widehat{f}_{\ell+r,m}|^2\\
		&\lesssim\frac{C^{2r}}{r!^2}\Big(\frac{|A_\rho|R^2}{A_\sigma}\Big)^{2r}\sum_{\ell\geq1}\sum_{|m|\leq \ell}|\widehat{f}_{\ell+r,m}|^2\lesssim \frac{C^{2r}}{r!^2}\Big(\frac{|A_\rho|R^2}{A_\sigma}\Big)^{2r}\|f\|_{L^\infty}^2,
		\end{aligned}
		\end{equation*}
		where we use Parseval's identity and $\|f\|_{L^2}\lesssim\|f\|_{L^\infty}$ in the last inequality.
		Combined with the estimate for $\mathcal{K}_1$ \eqref{K1_bound}, we conclude that
		\begin{equation*}
		\begin{aligned}
		\|\mathcal{S}[f]\|_{L^\infty}&\lesssim
		\|f\|_{L^\infty}\sum_{r\geq0}
		\frac{C^r}{r!}\Big(\frac{|A_\rho|R^2}{A_\sigma}\Big)^{r}\lesssim C\Big(\frac{|A_\rho|R^2}{A_\sigma}\Big)\|f\|_{L^\infty}.
		\end{aligned}
		\end{equation*}
		The $W^{1,\infty}$ estimate follows in the same way. Indeed, it holds that $\mathcal{K}_r:H^1(\mathbb{S}^2)\to H^3(\mathbb{S}^2)$ for $r\geq2$, and Corollary \ref{cor:multiplier} implies that $\|\mathcal{K}_{1,1}f\|_{W^{2,\infty}}+\|\mathcal{K}_{1,2}f\|_{W^{3,\infty}}\lesssim \|f\|_{W^{1,\infty}}$, hence $\|\mathcal{K}_{1}f\|_{W^{1,\infty}}\lesssim \frac{|A_\rho| R^2}{A_\sigma}\|f\|_{W^{1,\infty}}$.
		
	\end{proof}

For each fixed $m$, $L^{(m)}$ \eqref{Ldef_kj} is upper triangular and shares the structure of the linear operator in 	Proposition 7.1 in 
	\cite{GancedoGarciaJuarezPatelStrain23}. The argument there applies here with coefficients $a_k$ and $b_k^{(m)}$, showing that $\mathcal{S}$ and $\mathcal{S}^{-1}$ are inverses of each other and, moreover, that they  diagonalize the linear operator $\mathcal{L}_a$: 
	\begin{equation}\label{L_a_diag}
	\mathcal{L}_a=\mathcal{S}\mathcal{D}\mathcal{S}^{-1}.
	\end{equation}
	We conclude this subsection by stating the analogous bounds for $\mathcal{S}^{-1}$.

	\begin{cor}\label{cor:Sinv}
		Let $m\in\mathbb{Z}$ be a nonnegative integer, $\alpha\in(0,1)$,  $f\in C^{m,\alpha}(\mathbb{S}^2)$, and $\mathcal{S}^{-1}$ the operator defined in \eqref{def:mathcalS}. Then, 
		\begin{equation*}
		\|\mathcal{S}^{-1}f\|_{W^{1,\infty}} \lesssim C\Big(\frac{|A_\rho|R^2}{A_\sigma}\Big) \|f\|_{W^{1,\infty}},\quad
		\|\mathcal{S}^{-1}f\|_{C^{m,\alpha}} \lesssim C\Big(\frac{|A_\rho|R^2}{A_\sigma}\Big)\|f\|_{C^{m,\alpha}}.  
		\end{equation*}
	\end{cor}
	
	\begin{proof}
		
	We write
	\begin{equation*}
	\mathcal S^{-1}
	=
	I+\widetilde{\mathcal K}_1+\widetilde{\mathcal K}_{\geq2},
	\end{equation*}
	where $\widetilde{\mathcal K}_j$ denotes the operator associated with
	the $j$-th superdiagonal of $S^{-1}$.
	We first observe that the coefficients of $S^{-1}$ satisfy the same
	bounds as those of $S$. Indeed, let $N=\ell+M$, with $M\geq1$.
	By \eqref{Skjinverse_def},
	\begin{equation*}
	\begin{aligned}
	\big|(S^{-1})_{\ell,N}^{(m)}\big|
	&=
	\prod_{r=1}^{M}
	\frac{\big|b_{\ell+r-1}^{(m)}\big|}
	{a_{\ell+r}-a_\ell}.
	\end{aligned}
	\end{equation*}
	As in the proof of Lemma \ref{Skj_lem},
	\begin{equation*}
	a_{\ell+r}-a_\ell\gtrsim r(\ell+r)^2,
	\qquad
	\big|b_{\ell+r-1}^{(m)}\big|
	\lesssim \frac{|A_\rho|R^2}{A_\sigma}(\ell+r).
	\end{equation*}
	Consequently,
    \begin{equation*}
	\begin{aligned}
	\big|(S^{-1})_{\ell,N}^{(m)}\big|
	&\lesssim
	\Big(\frac{|A_\rho|R^2}{A_\sigma}\Big)^M\frac{C^M}{M!}
	\prod_{r=1}^{M}\frac1{\ell+r}
	=
	\Big(\frac{|A_\rho|R^2}{A_\sigma}\Big)^M\frac{C^M}{M!}\frac{\ell!}{N!}
	\lesssim
	\Big(\frac{|A_\rho|R^2}{A_\sigma}\Big)^M\frac{C^M}{M!N^M},
	\end{aligned}
	\end{equation*}
    where we rename $C$ in the last inequality.
	This is the same bound as \eqref{S_lN_bound}. Therefore, for
	$N>\ell+1$, the proof of Lemma \ref{Skj_lem} applies \textit{verbatim} and
	gives
	\begin{equation*}
	\sup_{\ell\sim2^j,\;N\sim2^k}
	\big|(S^{-1})_{\ell,N}^{(m)}\big|
	\lesssim
	C(\frac{|A_\rho|R^2}{A_\sigma})2^{-2k}e^{-c2^{k-j}}.
	\end{equation*}
	Moreover, the first superdiagonal satisfies
	\begin{equation*}
	(S^{-1})_{\ell-1,\ell}^{(m)}
	=
	-S_{\ell-1,\ell}^{(m)},
	\end{equation*}
	and hence
	$\widetilde{\mathcal K}_1=-\mathcal K_1$. Thus the bounds for the
	first superdiagonal follow directly from \eqref{K1_bound}. For the
	higher superdiagonals, the preceding coefficient estimate is identical
	to the one used for $\mathcal K_{\geq2}$. Consequently, the proofs of
	Propositions \ref{prop_S_deltak} and \ref{prop_S_holder} apply 
	with $\mathcal S$ replaced by $\mathcal S^{-1}$. We conclude that
	\begin{equation*}
	\|\mathcal S^{-1}f\|_{W^{1,\infty}}
	\lesssim C(\frac{|A_\rho|R^2}{A_\sigma})\|f\|_{W^{1,\infty}},
	\qquad
	\|\mathcal S^{-1}f\|_{C^{m,\alpha}}
	\lesssim C(\frac{|A_\rho|R^2}{A_\sigma})\|f\|_{C^{m,\alpha}}.
	\end{equation*}

	\end{proof}

	\subsection{Semigroup Estimates}
	
	The linear operator $\mathcal{L}_a$ \eqref{La_def}  gives rise to the semigroup $e^{t\mathcal{L}_a}$ defined as follows
	\begin{equation*}
	\begin{aligned}
	e^{t\mathcal{L}_a}f=\mathcal{S}e^{t\mathcal{D}}\mathcal{S}^{-1}f,    
	\end{aligned}
	\end{equation*}
	where $\mathcal{D}$ is the diagonal operator with multiplier $-a_k$ defined in \eqref{def:akbk}.

	\begin{lemma}\label{semig_diag}
		Let $\alpha\in(0,1)$ and let $\mathcal{D}$ be the diagonal linear operator defined by a Fourier multiplier $-\lambda_n$ satisfying \eqref{multiplier} for $n\geq1$ with $\beta=3$.
		Let $f\in W^{1,\infty}(\mathbb{S}^2)$  and such that $\int_{\mathbb{S}^2}f(x)d\sigma(x)=0$. Then, there exists $\gamma>0$ such that, for $k\geq1$,
		\begin{equation*}
		\|\Delta_k e^{t\mathcal{D}}f\|_{L^\infty}\lesssim e^{-c2^{3k}t}\|\Delta_kf\|_{L^\infty},
		\end{equation*}
		and 
		\begin{equation*}
		\|e^{t\mathcal{D}}f\|_{W^{1,\infty}}\lesssim e^{-\gamma t}\|f\|_{W^{1,\infty}},
		\end{equation*} 
		\begin{equation*}
		t^{\frac{2+\alpha}{3}}\|e^{t\mathcal{D}}f\|_{C^{3,\alpha}}\lesssim e^{-\gamma t}\|f\|_{W^{1,\infty}}.
		\end{equation*} 
	\end{lemma}
	\begin{proof}
		We first prove the frequency-localized estimate.
		Using the Littlewood--Paley decomposition for $f$, we have
		\begin{equation*}
		\begin{aligned}
		\Delta_k e^{t\mathcal{D}}f(x)&=\sum_{j\geq1}e^{t\mathcal{D}}\Delta_k\Delta_j f(x)=\sum_{j\geq1, |j-k|\leq 1}\sum_{n\geq0}\phi_j(n)e^{-\lambda_nt}\text{proj}_n(\Delta_k f)(x)\\
		&=\sum_{j\geq1, |j-k|\leq 1}\int_{\mathbb{S}^2}\sum_{n\geq0}\phi_j(n)e^{-\lambda_nt}K_n(x\cdot y)\Delta_kf(y)d\sigma(y)\\
		&=\sum_{j\geq1, |j-k|\leq 1}\int_{\mathbb{S}^2}G_j(x\cdot y,t)\Delta_kf(y)d\sigma(y),
		\end{aligned}
		\end{equation*}
		where $G_j$ is the kernel defined in \eqref{Gj_kernel} with $\gamma_n\equiv 1$. Young's inequality for convolutions and the kernel bound \eqref{KernelBound_semigroup} then gives 
		\begin{equation*}
		\begin{aligned}
		\|\Delta_k e^{t\mathcal{D}}f(x)\|_{L^\infty}&\lesssim\sum_{j\geq1, |j-k|\leq 1}\int_{\mathbb{S}^2}|G_j(x\cdot y,t)|d\sigma(y)\|\Delta_k f\|_{L^\infty}\lesssim e^{-c2^{3k}t}\|\Delta_k f\|_{L^\infty}.
		\end{aligned}
		\end{equation*}
		Next, we use the Littlewood--Paley characterization of H\"older norms \eqref{BesovHolder} and the previous estimate to prove the $C^{3,\alpha}$ estimate. We have
		\begin{equation*}
		\begin{aligned}
		t^{\frac{2+\alpha}{3}}\|e^{t\mathcal{D}}f\|_{C^{3,\alpha}}\lesssim \sup_{k\geq1}t^{\frac{2+\alpha}{3}}2^{(3+\alpha)k}\|\Delta_k e^{t\mathcal{D}}f\|_{L^\infty}&\lesssim \sup_{k\geq1}\big(t2^{3k}\big)^{\frac{\alpha}{3}} t^{\frac{2}{3}}2^{3k}e^{-c2^{3k}t}\|\Delta_k f\|_{L^\infty},
		\end{aligned}
		\end{equation*}    
		so, using the reverse Bernstein's inequality \eqref{BernsteinBound02}, we have $\|\Delta_k f\|_{L^\infty}\lesssim 2^{-k}\|f\|_{W^{1,\infty}}$ and thus
		\begin{equation*}
		\begin{aligned}
		t^{\frac{2+\alpha}{3}}\|e^{t\mathcal{D}}f\|_{C^{3,\alpha}}&\lesssim \|f\|_{W^{1,\infty}}\sup_{k\geq1} \big(t2^{3k}\big)^{\frac{2+\alpha}{3}}e^{-c2^{3k}t}\lesssim e^{-\gamma t}\|f\|_{W^{1,\infty}}. 
		\end{aligned}
		\end{equation*}
		Finally, we prove the estimate in $L^\infty$. 
		Let $0<t<2^{-3}$ and $M=M(t)$ be the  integer such that $2^{-3}<2^{3M}t\leq 1$.
		We write
		\begin{equation*}
		\begin{aligned}
		e^{\mathcal{D}t}f(x)=\Delta_{\leq M}f(x)+\Delta_{\leq M}(e^{\mathcal{D}t}-I)f(x)+\sum_{k>M}\Delta_ke^{\mathcal{D}t}f(x).
		\end{aligned}
		\end{equation*}    
		   By the low-frequency kernel bound \eqref{KernelBound_low} in Lemma \ref{lem:kernel_bounds},
    \begin{equation*}
        \begin{aligned}
            \|\Delta_{\leq M}f\|_{L^\infty}\lesssim \|f\|_{L^\infty}.
        \end{aligned}
    \end{equation*}
For the second term, Taylor expansion and the fact that $\Delta_0f=0$, due to the mean-zero condition, give
\begin{equation*}
    \begin{aligned}
        \Delta_{\leq M}(e^{t\mathcal{D}}-I)f=\sum_{k=1}^M \sum_{q=1}^\infty \frac{(-t)^q}{q!}\sum_{n\geq0}\phi_k(n)\lambda_n^q \text{proj}_n f.
    \end{aligned}
\end{equation*}
Applying \eqref{KernelBound_semigroup} in Lemma \ref{lem:kernel_bounds} with $t=0$ and $r=3q$, we obtain
 \begin{equation*}
    \begin{aligned}
        \Big\|\sum_{n\geq0}\phi_k(n)\lambda_n^q \text{proj}_n f\Big\|_{L^\infty}\lesssim C^q 2^{3kq}\|f\|_{L^\infty},
    \end{aligned}
\end{equation*}
with $C$ a constant independent of $k, q$. The factor $C^q$ comes from the bounds for $\lambda_n^q$ and its derivatives in the boundedness proof for $h_{j,p}$ \eqref{hjp_bound} in the proof of Lemma \ref{lem:kernel_bounds}. Since $t2^{3k}\leq 1$ for $k\leq M$, it follows that
\begin{equation*}
    \begin{aligned}
        \|\Delta_{\leq M}(e^{t\mathcal{D}}-I)f\|_{L^\infty}\lesssim \|f\|_{L^\infty}\sum_{k=1}^M \sum_{q=1}^\infty \frac{C^q}{q!}(t2^{3k})^q\lesssim  \|f\|_{L^\infty}t\sum_{k=1}^M 2^{3k}\lesssim \|f\|_{L^\infty}.
    \end{aligned}
\end{equation*}
For the high frequencies, the  localized estimate \eqref{KernelBound_semigroup} gives
\begin{equation*}
		\begin{aligned}
		\|\sum_{k>M}\Delta_ke^{\mathcal{D}t}f\|_{L^\infty}&\lesssim \|f\|_{L^\infty}\sum_{k>M}e^{-c2^{3k}t}\lesssim\|f\|_{L^\infty}\sum_{k>M}e^{-c2^{3(k-M)}2^{3M}t}\lesssim\|f\|_{L^\infty},
		\end{aligned}
		\end{equation*}
        where in the last step we used $2^{3M}t>2^{-3}$. 
For $t\geq 2^{-3}$, we immediately get from the mean-zero condition on $f$ that
\begin{equation*}
		\begin{aligned}
		\|e^{\mathcal{D}t}f\|_{L^\infty}\lesssim \|f\|_{L^\infty}\sum_{k\geq1}e^{-c2^{3k}t}\lesssim e^{-\gamma t}\|f\|_{L^\infty},
		\end{aligned}
		\end{equation*}
        for some $\gamma>0$.

		Finally, the estimate in $W^{1,\infty}$ follows by the previous one using the tangential vector fields defined in \eqref{Vector-fields}:
        \begin{equation*}
		\begin{aligned}
		\|\nabla_{\mathbb{S}^2}e^{t\mathcal{D}}f\|_{L^\infty}&\lesssim \sum_{j=1}^3\|V_j e^{t\mathcal{D}}f\|_{L^\infty}\lesssim \sum_{j=1}^3\|e^{t\mathcal{D}}V_j f\|_{L^\infty}\lesssim e^{-\gamma t} \sum_{j=1}^3\|V_jf\|_{L^\infty}\\
		&\lesssim e^{-\gamma t}\|\nabla_{\mathbb{S}^2} f\|_{L^\infty}.
		\end{aligned}
		\end{equation*}
	\end{proof}

	\begin{lemma}\label{lem:semigroup_est}
		Let $\alpha\in(0,1)$ and let $\mathcal{L}_a$ be the linear operator defined in \eqref{La_def}. Let $f\in W^{1,\infty}(\mathbb{S}^2)$ and let $g(\cdot,t)\in C^{\alpha}(\mathbb{S}^2)$ be such that  $\int_{\mathbb{S}^2}f(x)d\sigma(x)=0$ and $\int_{\mathbb{S}^2}g(x,t)d\sigma(x)=0$. Then, there exists $\gamma>0$ such that 
		\begin{equation}\label{semig_1}
		\sup_{t>0}e^{\gamma t}\|e^{t\mathcal{L}_a}f\|_{W^{1,\infty}}\lesssim   C\Big(\frac{|A_\rho|R^2}{A_\sigma}\Big)\|f\|_{W^{1,\infty}},
		\end{equation}  
		\begin{equation}\label{semig_2}    \sup_{t>0}t^{\frac{2+\alpha}{3}}e^{\gamma t}\|e^{t\mathcal{L}_a}f\|_{C^{3,\alpha}}\lesssim   C\Big(\frac{|A_\rho|R^2}{A_\sigma}\Big)\|f\|_{W^{1,\infty}},
		\end{equation}
		\begin{equation}\label{semig_3}
		\sup_{t>0}e^{\gamma t}
		\Big\|
		\int_0^t e^{(t-\tau)\mathcal{L}_a} g(\tau)\, d\tau
		\Big\|_{W^{1,\infty}}
		\lesssim C\Big(\frac{|A_\rho|R^2}{A_\sigma}\Big)
		\sup_{t>0}  t^{\frac{2+\alpha}{3}} e^{\gamma t}\|g(t)\|_{C^{\alpha}}, 
		\end{equation}
		\begin{equation}\label{semig_4}
		\sup_{t>0}
		t^{\frac{2+\alpha}{3}} e^{\gamma t}\Big\|
		\int_0^t e^{(t-\tau)\mathcal{L}_a} g(\tau) \, d\tau
		\Big\|_{C^{3,\alpha}}
		\lesssim C\Big(\frac{|A_\rho|R^2}{A_\sigma}\Big)
		\sup_{t>0} t^{\frac{2+\alpha}{3}}e^{\gamma t} \|g(t)\|_{C^{\alpha}}.
		\end{equation}
		
	\end{lemma}
	\begin{proof}
		We recall that $\mathcal{L}_a=\mathcal{S}\mathcal{D}\mathcal{S}^{-1}$ \eqref{L_a_diag}, where $\mathcal{S}$ is defined in \eqref{Skj_def} and $\mathcal{D}$ is a diagonal operator given by the multiplier $-a_k$ \eqref{def:akbk}, which satisfies  \eqref{multiplier} for $k\geq1$ with $\beta=3$. 
		We will abuse notation and denote $\gamma$ for a different constant on the proof of each estimate. In the end, we will choose it sufficiently small, so that the four estimates hold.
        First, we prove inequalities \eqref{semig_1} and \eqref{semig_2}. By Proposition \ref{prop_S_holder}, we have
		\begin{equation*}
		\begin{aligned}
		\sup_{t>0}e^{\gamma t}\|e^{t\mathcal{L}_a}f\|_{W^{1,\infty}}\lesssim   C\Big(\frac{|A_\rho|R^2}{A_\sigma}\Big)\sup_{t>0}e^{\gamma t}\|e^{t\mathcal{D}}\mathcal{S}^{-1}f\|_{W^{1,\infty}}. 
		\end{aligned}
		\end{equation*}
		Then, we use Lemma \ref{semig_diag} and Corollary \ref{cor:Sinv},
		\begin{equation*}
		\begin{aligned}
		\sup_{t>0}e^{\gamma t}\|e^{t\mathcal{L}_a}f\|_{W^{1,\infty}}\lesssim C\Big(\frac{|A_\rho|R^2}{A_\sigma}\Big)\|\mathcal{S}^{-1} f \|_{W^{1,\infty}}\lesssim C\Big(\frac{|A_\rho|R^2}{A_\sigma}\Big)\|f\|_{W^{1,\infty}}.
		\end{aligned}
		\end{equation*}
		Similarly, 
		\begin{equation*}
		\begin{aligned}
		\sup_{t>0}t^{\frac{2+\alpha}{3}}e^{\gamma t}\|e^{t\mathcal{L}_a}f\|_{C^{3,\alpha}}\lesssim   C\Big(\frac{|A_\rho|R^2}{A_\sigma}\Big)\sup_{t>0}t^{\frac{2+\alpha}{3}}e^{\gamma t}\|e^{t\mathcal{D}}\mathcal{S}^{-1}f\|_{C^{3,\alpha}}, 
		\end{aligned}
		\end{equation*}
		and thus
		\begin{equation*}
		\begin{aligned}
		\sup_{t>0}t^{\frac{2+\alpha}{3}}e^{\gamma t}\|e^{t\mathcal{L}_a}f\|_{C^{3,\alpha}}\lesssim  C\Big(\frac{|A_\rho|R^2}{A_\sigma}\Big)\|\mathcal{S}^{-1} f\|_{W^{1,\infty}}\lesssim C\Big(\frac{|A_\rho|R^2}{A_\sigma}\Big)\|f\|_{W^{1,\infty}}.
		\end{aligned}
		\end{equation*}
		Next, we prove the inequality  \eqref{semig_4}. 	Recalling Proposition \ref{Prop-equiv}, we have
		\begin{equation*}
		\begin{aligned}
		\sup_{t>0} t^{\frac{2+\alpha}{3}}e^{\gamma t}
		&\Big\|
		\int_0^t e^{(t-\tau)\mathcal{L}_a} g(\tau)\,d\tau
		\Big\|_{C^{3,\alpha}}\\
		&\lesssim C\Big(\frac{|A_\rho|R^2}{A_\sigma}\Big)
		\sup_{t>0} t^{\frac{2+\alpha}{3}}e^{\gamma t}\sup_{k\geq1} 2^{(3+\alpha)k}
		\Big\|
		\int_0^t e^{(t-\tau)\mathcal{D}}\Delta_k\mathcal{S}^{-1} g(\tau)\,d\tau
		\Big\|_{L^\infty}.
		\end{aligned}
		\end{equation*}
		By the frequency-localized semigroup estimate in Lemma \ref{semig_diag}, we have
		\begin{equation*}
		\begin{aligned}
		\sup_{t>0} t^{\frac{2+\alpha}{3}}&e^{\gamma t}
		\Big\|
		\int_0^t e^{(t-\tau)\mathcal{L}_a} g(\tau)\,d\tau
		\Big\|_{C^{3,\alpha}}\\
		&\lesssim C\Big(\frac{|A_\rho|R^2}{A_\sigma}\Big)
		\sup_{t>0}e^{\gamma t}\sup_{k\geq1}
		t^{\frac{2+\alpha}{3}} 2^{(3+\alpha)k}
		\!\int_0^t e^{-c\,2^{3k}(t-\tau)}
		\|\Delta_k \mathcal{S}^{-1}g(\tau)\|_{L^\infty}\,d\tau.
		\end{aligned}
		\end{equation*}
		Now writing
		\begin{equation*}
		2^{\alpha k}\|\Delta_k \mathcal{S}^{-1}g(\tau)\|_{L^\infty}
		\lesssim C\Big(\frac{|A_\rho|R^2}{A_\sigma}\Big)
		\tau^{-\frac{2+\alpha}{3}}e^{-\gamma \tau}\sup_{s>0} s^{\frac{2+\alpha}{3}}e^{\gamma s}\|g(s)\|_{C^\alpha},
		\end{equation*}
		we obtain
		\begin{equation*}
		\begin{aligned}
		\sup_{t>0} t^{\frac{2+\alpha}{3}}e^{\gamma t}
		&\Big\|
		\int_0^t e^{(t-\tau)\mathcal{L}_a} g(\tau)\,d\tau
		\Big\|_{C^{3,\alpha}}\\
		&\lesssim C\Big(\frac{|A_\rho|R^2}{A_\sigma}\Big)\Big(\sup_{s>0} s^{\frac{2+\alpha}{3}}e^{\gamma s}\|g(s)\|_{C^\alpha}\Big)
		\sup_{t>0}\sup_{k\geq1}
		t^{\frac{2+\alpha}{3}}
		\int_0^t
		2^{3k}e^{\gamma (t-\tau)}e^{-c\,2^{3k}(t-\tau)}
		\frac{d\tau}{\tau^{\frac{2+\alpha}{3}}}.      
		\end{aligned}
		\end{equation*}
		Choosing $\gamma>0$ sufficiently small, rescaling $\tau\to\tau/t$ and splitting the last integral  near and far from $\tau=0$, we obtain that the last term is uniformly bounded in $k$ and $t$, giving the desired estimate.
		
		Lastly, we prove the estimate \eqref{semig_3}. We have
		\begin{equation*}
		\begin{aligned}
		\sup_{t>0}e^{\gamma t}
		&\Big\|
		\int_0^t e^{(t-\tau)\mathcal{L}_a} g(\tau) d\tau
		\Big\|_{W^{1,\infty}}
		\lesssim C\Big(\frac{|A_\rho|R^2}{A_\sigma}\Big)\sup_{t>0}e^{\gamma t}\sum_{k\geq1}
		\Big\|
		\int_0^t e^{(t-\tau)\mathcal{D}} \Delta_k\big(\mathcal{S}^{-1}g(\tau)\big) d\tau
		\Big\|_{W^{1,\infty}}\\
		&\lesssim C\Big(\frac{|A_\rho|R^2}{A_\sigma}\Big)\sup_{t>0}e^{\gamma t}
		\int_0^t\sum_{k\geq1} 2^ke^{-c2^{3k}(t-\tau)}\|\Delta_k\big(\mathcal{S}^{-1}g(\tau)\big)\|_{L^\infty} d\tau\\
		&\lesssim C\Big(\frac{|A_\rho|R^2}{A_\sigma}\Big)\big(\sup_{s>0}s^{\frac{2+\alpha}{3}}e^{\gamma s}\|g(s)\|_{C^{\alpha}}\big)\sup_{t>0}
		\int_0^t e^{\gamma (t-\tau)}\sum_{k\geq1} 2^{(1-\alpha)k}e^{-c2^{3k}(t-\tau)}\frac{d\tau}{\tau^{\frac{2+\alpha}{3}}}.
		\end{aligned}
		\end{equation*}
		The sum in $k$ is bounded by
		\begin{equation*}
		\begin{aligned}
		\sum_{k\geq1} 2^{(1-\alpha)k}e^{-c2^{3k}(t-\tau)}\lesssim \int_1^\infty\!\!2^{(1-\alpha)x}e^{-c2^{3x}(t-\tau)}dx\lesssim\frac{1}{(t\!-\!\tau)^{\frac{1-\alpha}{3}}}\int_{c(t-\tau)}^\infty \frac{e^{-w}}{w^{\frac{2+\alpha}{3}}}dw\lesssim \frac{e^{-c(t-\tau)}}{(t\!-\!\tau)^{\frac{1-\alpha}{3}}}, 
		\end{aligned}
		\end{equation*}
		where the change of variable $w=c2^{3x}(t-\tau)$ was used and the constant $c>0$ was renamed.
		Hence, renaming $c>0$ again and choosing $\gamma>0$ sufficiently small, 
		\begin{equation*}
		\begin{aligned}
		\sup_{t>0}e^{\gamma t}
		\Big\|
		\int_0^t e^{(t-\tau)\mathcal{L}_a} &g(\tau) d\tau
		\Big\|_{W^{1,\infty}}\\	
        &\lesssim C\Big(\frac{|A_\rho|R^2}{A_\sigma}\Big) \big(\sup_{s>0}s^{\frac{2+\alpha}{3}}e^{\gamma s}\|g(s)\|_{C^{\alpha}}\big)\sup_{t>0}
		\int_0^t \frac{e^{-(c-\gamma)(t-\tau)}}{\tau^{\frac{2+\alpha}{3}}(t-\tau)^{\frac{1-\alpha}{3}}}d\tau\\
		&\lesssim C\Big(\frac{|A_\rho|R^2}{A_\sigma}\Big)\sup_{t>0}t^{\frac{2+\alpha}{3}}e^{\gamma t}\|g(t)\|_{C^{\alpha}}.
		\end{aligned}
		\end{equation*}
		
	\end{proof}

	\section{Nonlinear Estimates}\label{sec:nonlinear}
	
In this section we prove our main estimates on the Muskat nonlinearity.	
	
	\subsection{Multilinear estimates} The Muskat equation contains a complex nonlinearity, which is defined non-locally in terms of nonlinear Birkhoff--Rott singular integrals (see Section \ref{sec:linearization}). To estimate it, we prove first several multilinear Calder\'{o}n-type estimates. 
	
	In the following lemmas, for any $k\in\mathbb{Z}$ we let $\Phi_{\leq k}:\mathbb{R}^3\to[0,1]$ be defined by $\Phi_{\leq k}(x)=\chi(2^{-k}|x|)$ and $\Phi_{>k}(x)=1-\Phi_{\leq k}(x)$, where the cutoff $\chi$ was defined in Subsection \ref{subsec:LP}. 

In this subsection, the notation $A\lesssim_QB$ means that there is a constant $C\geq 1$  such that $A\leq C^QB$.
	
	Our first lemma concerns linear estimates.
	
	\begin{lemma}\label{lem:model_linear}
		Let $m\geq 1$ be odd and $k\in\mathbb{Z}_+$. Let $p_m:\mathbb R^3\to\mathbb R$, $p_m(x)=x_1^{\beta_1}x_2^{\beta_2}x_3^{\beta_3}$, $|\beta|=m$, be a
		homogeneous polynomial of degree $m$ and let $g\in C^\alpha(\mathbb S^2)$, $\alpha\in(0,1)$.
		Define
		\begin{equation}\label{linear_model}
		\mathcal T_{p_m, k}^{(0)}[g](x)=\operatorname{p.v.}
		\int_{\mathbb S^2}\frac{p_m(x-y)}{|x-y|^{m+2}}\Phi_{\leq-k}(x-y)g(y)\,d\sigma(y).
		\end{equation}
		Then there is a constant $C_\alpha\geq 2$ (independent of $m, k, g$) such that 
		\begin{equation}\label{linear_model_estimate}
		\big\|\mathcal T_{p_m, k}^{(0)}[g]\big\|_{C^\alpha}\leq C_\alpha^m\|g\|_{C^\alpha}
		\end{equation}
		and 	
		\begin{equation}\label{linear_model_estimate2}
 \big\|\mathcal T^{(0)}_{p_m,k}[\Delta_{\leq k+10}g]\big\|_{L^\infty} \leq C_\alpha^{m}\|\Delta_{\leq k+10}g\|_{L^\infty}.
\end{equation}	
		\end{lemma}

\begin{proof} The estimate \eqref{linear_model_estimate} follows from standard boundedness of order-zero
pseudodifferential operators on H\"older spaces; see, for example,  
\cite[Corollary~2.1.B]{Taylor91}. To prove \eqref{linear_model_estimate2} we write
\begin{align*}
 \mathcal T^{(0)}_{p_m,k}[\Delta_{\leq k+10}g](x)
 &=\int_{\mathbb S^2}\frac{p_m(x-y)}{|x-y|^{m+2}}
   \Phi_{\leq-k}(x-y)\big(\Delta_{\leq k+10}g(y)-\Delta_{\leq k+10}g(x)\big)\,d\sigma(y)\\
   &\quad+\Delta_{\leq k+10}g(x)\mathcal T^{(0)}_{p_m,k}[1](x),
\end{align*}
 The $L^\infty$ norm of the second term is bounded by $C_\alpha^m\|\Delta_{\leq k+10}g\|_{L^\infty}$,
in view of \eqref{linear_model_estimate}. Since $|p_m(z)|\leq|z|^m$ and
$\|\nabla_{\mathbb S^2}\Delta_{\leq k+10}g\|_{L^\infty}\lesssim2^k\|\Delta_{\leq k+10}g\|_{L^\infty}$,
the first term is bounded by
\[
 \|\nabla_{\mathbb S^2}\Delta_{\leq k+10}g\|_{L^\infty}
 \int_{|x-y|\lesssim2^{-k}}|x-y|^{-1}\,d\sigma(y)
 \lesssim\|\Delta_{\leq k+10}g\|_{L^\infty}.
\]
This completes the proof of \eqref{linear_model_estimate2}.
\end{proof}

We prove now a more general multilinear estimate.

\begin{lemma}[Multilinear singular integral operators]
\label{lem:model_nonlinearity}
 Let $q\geq 0$, $m\geq0$, $p_m(x)=x_1^{\beta_1}x_2^{\beta_2}x_3^{\beta_3}$, $|\beta|=m$, and
 assume that $m+q$ is odd.  Assume that
 \[
 a_1,\ldots,a_q\in C^{1,\alpha}(\mathbb S^2),
 \qquad g\in C^\alpha(\mathbb S^2),\qquad 0<\alpha<1,
 \]
 and define
 \begin{equation}\label{multilinear_model}
 \mathcal T_{p_m}^{(q)}[a_1,\ldots,a_q,g](x)
 =\operatorname{p.v.}\int_{\mathbb S^2}
 \frac{\prod_{j=1}^q(a_j(x)-a_j(y))p_m(x-y)}
 {|x-y|^{m+q+2}}g(y)\,d\sigma(y).
 \end{equation}
 Then there is $C_\alpha\geq 1$ such that
 \begin{equation}\label{multilinear_model_estimate}
 \begin{split}
 \big\|\mathcal T_{p_m}^{(q)}[a_1,\ldots,a_q,g]\big\|_{C^\alpha}
 \leq C_\alpha^{m+q}\bigg\{\|g\|_{C^\alpha}\prod_{j=1}^q\|a_j\|_{W^{1,\infty}}+\|g\|_{L^\infty}\sum_{j=1}^q\|a_j\|_{C^{1,\alpha}}
 \prod_{r\ne j}\|a_r\|_{W^{1,\infty}}\bigg\}.
 \end{split}
 \end{equation}
\end{lemma}

\begin{proof}

By multilinearity we may assume that
\begin{equation}\label{sec6eq0}
\|a_1\|_{W^{1,\infty}}=\cdots=\|a_q\|_{W^{1,\infty}}
 =\|g\|_{L^\infty}=1.
\end{equation}
We may also assume that $m=0$ and $p_m\equiv 1$, by appending the \(m\) corresponding coordinate functions to the \(a_j\)’s and relabeling \(q+m\) as \(q\) (thus the new \(q\) is odd). We notice that 
\begin{equation}\label{eq:L62_Taylor}
 |a_j(x)-a_j(y)|\lesssim |x-y|,\qquad
 |a_j(x)-a_j(y)-\nabla_{\mathbb S^2}a_j(x)\cdot(x-y)|
 \lesssim \|a_j\|_{C^{1,\alpha}}|x-y|^{1+\alpha},
\end{equation}
for any $j\in\{1,\ldots,q\}$ and $x,y\in\mathbb{S}^2$.

For any $n\geq0$, we decompose the multilinear singular integral operator in two integrals:
\begin{align*}
 I_n^1(x)&=\sum_{k\geq n-4}\operatorname{p.v.}
 \int_{\mathbb S^2}
 \frac{\prod_{j=1}^q(a_j(x)-a_j(y))}{|x-y|^{q+2}}
 \Delta_kg(y)\,d\sigma(y),\\
 I_n^2(x)&=\operatorname{p.v.}\int_{\mathbb S^2}
 \frac{\prod_{j=1}^q(a_j(x)-a_j(y))}{|x-y|^{q+2}}
 \Delta_{\leq n-5}g(y)\,d\sigma(y).
\end{align*}
In view of Proposition \ref{Prop-equiv}, for \eqref{multilinear_model_estimate} it remains to prove that for any $n\in\mathbb{Z}_+$ and $l\in\{1,2\}$
\begin{equation}\label{sec6eq2}
2^{n\alpha}\|\Delta_nI_n^l\|_{L^\infty}\lesssim_q\big(
 \|g\|_{C^\alpha}+\sum_{j=1}^q\|a_j\|_{C^{1,\alpha}}\big).
\end{equation}

\textbf{Step 1.} We prove first the bounds \eqref{sec6eq2} in the case $l=1$. Let
\begin{equation}\label{sec6eq5}
 K_a(x,y)=\frac{\prod_{j=1}^q(a_j(x)-a_j(y))}{|x-y|^{q+2}}.
\end{equation}
For every $k\geq n-4$, we split $I_n^1=I_n^{1,1}+I_n^{1,2}+I_n^{1,3}$ where
\begin{equation*}
\begin{split}
I_n^{1,1}(x)&=\sum_{k\geq n-4}\operatorname{p.v.}\int_{\mathbb S^2}K_a(x,y)
 \Delta_kg(y)\Phi_{\leq -k}(x-y)\,d\sigma(y),\\
 I_n^{1,2}(x)&=\sum_{k\geq n-4}\int_{\mathbb S^2}
K_a(x,y)\Delta_kg(y)\big[\Phi_{\leq -n+10}(x-y)-\Phi_{\leq -k}(x-y)]\,d\sigma(y),\\
 I_n^{1,3}(x)&=\sum_{k\geq n-4}\int_{\mathbb S^2}
K_a(x,y)\Delta_kg(y)\Phi_{>-n+10}(x-y)\,d\sigma(y).
\end{split}
\end{equation*}
The component $I_n^{1,2}$ is estimated directly, using \eqref{eq:L62_Taylor} and  $\|\Delta_kg\|_{L^\infty}\lesssim \|g\|_{C^\alpha}2^{-k\alpha}$,
\begin{align}
 2^{n\alpha}\|\Delta_nI_n^{1,2}\|_{L^\infty}
 &\lesssim_q\sum_{k\geq n-4} 2^{n\alpha}\|\Delta_kg\|_{L^\infty}(1+|k-n|)\lesssim_q \|g\|_{C^\alpha}.
 \label{eq:L62_high_middle}
\end{align}
Moreover, we notice that $I_n^{1,3}=0$ if $n=0$ and, for any $V\in\mathcal{V}$,
\begin{equation}\label{eq:L62_kernel_crude}
 |K_a(x,y)|\lesssim_q|x-y|^{-2},\qquad
 |V_x[K_a(x,y)\Phi_{>-n+10}(x-y)]|\lesssim_q|x-y|^{-3}\mathbf{1}_{|x-y|\gtrsim 2^{-n}}.
\end{equation}
For $n\geq1$, we use the bounds \eqref{eq:L62_rotation_reverse} and \eqref{eq:L62_kernel_crude} to estimate
\begin{equation} \label{eq:L62_high_far}
\begin{split}
 2^{n\alpha}\|\Delta_nI_n^{1,3}\|_{L^\infty}&\lesssim2^{-n+n\alpha}
 \sum_{V\in\mathcal V}\|\Delta_n(VI_n^{1,3})\|_{L^\infty}\\
 &\lesssim_q 2^{-n+n\alpha}\sum_{k\geq n-4}\int_{|x-y|\gtrsim 2^{-n}}|x-y|^{-3}\|\Delta_kg\|_{L^\infty}\,d\sigma(y)\\
 &\lesssim_q\|g\|_{C^\alpha}.
\end{split}
\end{equation}
We estimate now $I_n^{1,1}$. We use the formula
\begin{equation}\label{eq:L62_telescoping}
\begin{split}
 &\prod_{j=1}^q(a_j(x)-a_j(y))=\prod_{j=1}^q[\nabla_{\mathbb S^2}a_j(x)\cdot(x-y)]+R_q(x,y),\\
 &R_q(x,y)=\sum_{j=1}^q[a_j(x)-a_j(y)-\nabla_{\mathbb S^2}a_j(x)\cdot(x-y)]\prod_{r<j}[\nabla_{\mathbb S^2}a_r(x)\cdot(x-y)]\prod_{l>j}[a_l(x)-a_l(y)]
   \end{split}
\end{equation} 
to decompose $I_n^{1,1}=F_n^{1,1}+R_n^{1,1}$, where
\begin{equation}\label{eq:L62_high_frozen}
\begin{split}
 F_n^{1,1}(x)&=\sum_{k\geq n-4}
 \operatorname{p.v.}\int_{\mathbb S^2}\frac{\prod_{j=1}^q\nabla_{\mathbb{S}^2}\big[a_j(x)\cdot(x-y)\big]}{|x-y|^{q+2}}\Delta_k g(y)\Phi_{\leq -k}(x-y)d\sigma(y),\\
 R_n^{1,1}(x)&=
 \sum_{k\geq n-4}\int_{\mathbb S^2}
 \frac{R_q(x,y)}{|x-y|^{q+2}}
 \Delta_kg(y)\Phi_{\leq-k}(x-y)\,d\sigma(y).
\end{split}
\end{equation}
To estimate $F_n^{1,1}$, we first introduce the following notation
\begin{equation}\label{eq:L62_defi}
c_I(x)=\prod_{j=1}^q(\nabla_{\mathbb{S}^2}a_j(x))_{i_j},
 \qquad
 P_I(z)=z_{i_1}\cdots z_{i_q},
\end{equation}
with the index set $I=(i_1,\ldots,i_q)\in\{1,2,3\}^q$. Then, we can write
\begin{equation*}
\begin{split}
 F_n^{1,1}(x)&=\sum_{k\geq n-4}\sum_{I\in\{1,2,3\}^q}c_I(x)\mathcal{T}^{(0)}_{P_I,k}[\Delta_k g](x).
\end{split}
\end{equation*}
The normalization \eqref{sec6eq0} and the product rule imply
\begin{equation}\label{eq:L62_coefficient_bounds}
 \|c_I\|_{L^\infty}\leq 1,
 \qquad
 \|c_I\|_{C^\alpha}\lesssim_q \sum_{j=1}^q \|a_j\|_{C^{1,\alpha}}.
\end{equation}
Then, we apply Lemma \ref{lem:model_linear} to estimate $F_n^{1,1}$. Since \(\Delta_kg=\Delta_{\leq k+10}\Delta_kg\), we have 
\begin{equation}\label{sec6eq10}
\begin{split}
2^{n\alpha}\|\Delta_nF_n^{1,1}\|_{L^\infty}&\lesssim_q
2^{n\alpha}\sum_{k\geq n-4}
\|\Delta_{\leq k+10}\Delta_k g\|_{L^\infty}\lesssim_q2^{n\alpha}\sum_{k\geq n-4}2^{-k\alpha}\|g\|_{C^\alpha}\lesssim_q
\|g\|_{C^\alpha},
\end{split}
\end{equation}
where we have used the normalization \eqref{sec6eq0}.
By \eqref{eq:L62_Taylor} and \eqref{eq:L62_telescoping}, we obtain
\begin{equation}\label{eq:L62_remainder_pointwise}
 \frac{|R_q(x,y)|}{|x-y|^{q+2}}
 \lesssim_q |x-y|^{-2+\alpha}\sum_{j=1}^q \|a_j\|_{C^{1,\alpha}},
\end{equation}
and since
\begin{equation*}
 \int_{\mathbb S^2}|x-y|^{-2+\alpha}
 \Phi_{\leq-k}(x-y)\,d\sigma(y)
 \lesssim 2^{-k\alpha},
\end{equation*}
we estimate the remainder $R_n^{1,1}$ by
\begin{align}
 2^{n\alpha}\|\Delta_nR_n^{1,1}\|_{L^\infty}
 &\lesssim_q 2^{n\alpha}\|g\|_{L^\infty} \sum_{j=1}^q \|a_j\|_{C^{1,\alpha}}
 \sum_{k\geq n-4}2^{-k\alpha}\lesssim_q \sum_{j=1}^q \|a_j\|_{C^{1,\alpha}}.
 \label{eq:L62_high_remainder}
\end{align}
The desired bounds \eqref{sec6eq2} follow from \eqref{eq:L62_high_middle}, \eqref{eq:L62_high_far}, \eqref{sec6eq10}, and  \eqref{eq:L62_high_remainder}. 
\medskip

\textbf{Step 2.} To prove \eqref{sec6eq2} for $I_n^2$ we may assume $n\geq 5$ and decompose $I_n^2=I_n^{2,1}+I_n^{2,2}$ with
\begin{equation*}
\begin{split}
I_n^{2,1}(x)&=\operatorname{p.v.}\int_{\mathbb S^2}K_a(x,y)
 \Delta_{\leq n-5}g(y)\Phi_{\leq -n}(x-y)\,d\sigma(y),\\
 I_n^{2,2}(x)&=\int_{\mathbb S^2}K_a(x,y)\Delta_{\leq n-5}g(y)\Phi_{> -n}(x-y)\,d\sigma(y).
 \end{split}
\end{equation*}
To bound $I_n^{2,1}$ we use the identity \eqref{eq:L62_telescoping} and decompose $I_n^{2,1}=F_n^{2,1}+R_n^{2,1}$ as in \eqref{eq:L62_high_frozen}. The same estimates as before, using \eqref{eq:L62_remainder_pointwise} and \eqref{eq:L62_coefficient_bounds}, show that
\begin{equation}\label{eq:L62_low_near_remainder}
 2^{n\alpha}\|\Delta_nR_n^{2,1}\|_{L^\infty}\lesssim_q \sum_{j=1}^q \|a_j\|_{C^{1,\alpha}}.
\end{equation}
On the other hand, the $C^\alpha$ product rule, and Lemma \ref{lem:model_linear} yield
\begin{equation}\label{eq:L62_low_near_frozen}
 2^{n\alpha}\|\Delta_nF_n^{2,1}\|_{L^\infty}
 \lesssim_q\|g\|_{C^\alpha}+\sum_{j=1}^q \|a_j\|_{C^{1,\alpha}}.
\end{equation}
To bound $I_n^{2,2}$, for $V\in\mathcal V$, write $V_x$ and $V_y$ for the same rotation
field acting in the $x$ and $y$ variables.  These fields generate
simultaneous rotations of the pair $(x,y)$.  Consequently,
\begin{equation*}
 (V_x+V_y)\big(a_j(x)-a_j(y)\big)=(Va_j)(x)-(Va_j)(y),\quad\mbox{and}\quad (V_x+V_y)[\psi(x-y)]=0,
\end{equation*}
for any radially symmetric function $\psi$. It follows that
\begin{equation}\label{eq:L62_rotation_transfer}
 (V_x+V_y)\big(K_a(x,y)\Phi_{>-n}(x-y)\big)
 =E(x,y)\Phi_{>-n}(x-y),
\end{equation}
where
\begin{equation*}
 E(x,y)=\sum_{j=1}^q \frac{\big((Va_j)(x)-(Va_j)(y)\big)
       \prod_{r\ne j}(a_r(x)-a_r(y))}
 {|x-y|^{q+2}}.
\end{equation*}
Using the identities
\eqref{eq:L62_rotation_basic} and \eqref{eq:L62_rotation_transfer}, we
obtain the exact formula, for any $V\in\mathcal{V}$,
\begin{equation}\label{eq:L62_rotation_far_identity}
\begin{split}
 V_xI_n^{2,2}&=M_{V,n}+R_{V,n},\\
 M_{V,n}(x)
 &=\int_{\mathbb S^2}K_a(x,y)(V\Delta_{\leq n-5}g)(y)\Phi_{>-n}(x-y)\,d\sigma(y),\\
 R_{V,n}(x)&=\int_{\mathbb S^2}E(x,y)\Delta_{\leq n-5}g(y)\Phi_{>-n}(x-y)\,d\sigma(y),
     \end{split}
\end{equation}
Since  $\|Va_j\|_{C^\alpha}\lesssim \|a_j\|_{C^{1,\alpha}}$, we have
\begin{equation*}
 |E(x,y)|\lesssim_q |x-y|^{-3+\alpha}\sum_{j=1}^q \|a_j\|_{C^{1,\alpha}}.
\end{equation*}
Therefore
\begin{align}
 2^{-n+n\alpha}
 \sum_{V\in\mathcal V}\|\Delta_nR_{V,n}\|_{L^\infty}
 &\lesssim_q \sum_{j=1}^q \|a_j\|_{C^{1,\alpha}}\|g\|_{L^\infty}
 2^{-n+n\alpha}2^{n(1-\alpha)}\lesssim_q \sum_{j=1}^q \|a_j\|_{C^{1,\alpha}}.
 \label{eq:L62_rotation_remainder_bound}
\end{align}
It remains to estimate $M_{V,n}$.  We expand
$V\Delta_{\leq n-5}g=\sum_{0\leq k\leq n-5}V\Delta_kg$ and write $M_{V,n}=O_{1,V}+O_{2,V}$, where
\begin{equation*}
\begin{split}
 O_{1,V}(x)&=\sum_{0\leq k\leq n-5}\int_{\mathbb S^2}
 K_a(x,y)(V\Delta_kg)(y)
 (\Phi_{\leq-k}-\Phi_{\leq-n})(x-y)\,d\sigma(y),\\
 O_{2,V}(x)&=\sum_{0\leq k\leq n-5}\int_{\mathbb S^2}
 K_a(x,y)(V\Delta_kg)(y) \Phi_{>-k}(x-y)\,d\sigma(y).
\end{split}
\end{equation*}
Using \eqref{eq:L62_kernel_crude}, we estimate
\begin{align}
 2^{-n+n\alpha}
 \sum_{V\in\mathcal V}\|\Delta_nO_{1,V}\|_{L^\infty}
 &\lesssim_q2^{-n+n\alpha}\sum_{0\leq k\leq n-5}2^k\|\Delta_kg\|_{L^\infty}(1+|n-k|)\lesssim_q\|g\|_{C^\alpha}.
 \label{eq:L62_rotation_O1_bound}
\end{align}
For $O_{2,V}$ we apply \eqref{eq:L62_rotation_reverse}, and then the kernel bounds \eqref{eq:L62_kernel_crude}.  Therefore
\begin{equation}\label{eq:L62_rotation_O2_bound}
\begin{split}
 2^{-n+n\alpha}\sum_{V\in\mathcal V}\|\Delta_nO_{2,V}\|_{L^\infty}
 &\lesssim 2^{-2n+n\alpha}
 \sum_{V,W\in\mathcal V}\|\Delta_n(WO_{2,V})\|_{L^\infty}\\
 &\lesssim_q2^{-2n+n\alpha}\sum_{0\leq k\leq n-5}2^{2k}\|\Delta_kg\|_{L^\infty}\\
 &\lesssim_q\|g\|_{C^\alpha}.
 \end{split}
 \end{equation}
Combining \eqref{eq:L62_rotation_reverse},
\eqref{eq:L62_rotation_far_identity},
\eqref{eq:L62_rotation_remainder_bound},
\eqref{eq:L62_rotation_O1_bound}, and
\eqref{eq:L62_rotation_O2_bound}, we obtain
 $2^{n\alpha}\|\Delta_nI_n^{2,2}\|_{L^\infty}
 \lesssim_q\big(\|g\|_{C^\alpha}+\sum_{j=1}^q \|a_j\|_{C^{1,\alpha}}\big)$. 
Together with \eqref{eq:L62_low_near_remainder}-\eqref{eq:L62_low_near_frozen}, this completes the proof of the main bounds \eqref{sec6eq2} for $l=2$, and therefore the proof of the lemma.

\end{proof}

Next, we prove  several interpolation inequalities.
For $s\geq0$,  we define the Besov norm
\begin{equation*}
 \|f\|_{B_{\infty,1}^{s}}
 =\sum_{k\geq0}2^{sk}\|\Delta_k f\|_{L^\infty},
\end{equation*}
and we will use the notation $W^{0,\infty}(\mathbb{S}^2)=L^\infty(\mathbb{S}^2)$.

\begin{lemma}[Critical interpolation]
\label{lem:interpolation} Assume $0<\alpha<1$, $m\in\{0,1\}$, and $f\in C^{2+m,\alpha}(\mathbb S^2)$. Then for any $\sigma\in (0,2+\alpha)$,
\begin{equation}\label{eq:critical_interpolation_scale}
 \|f\|_{B_{\infty,1}^{m+\sigma}}
 \lesssim
 \|f\|_{W^{m,\infty}}^{1-\frac{\sigma}{2+\alpha}}\|f\|_{C^{2+m,\alpha}}^{\frac{\sigma}{2+\alpha}}.
 \end{equation}
Consequently, if $m\in\{0,1\}$ and $f\in C^{m+2,\alpha}(\mathbb{S}^2)$ then
\begin{equation}\label{eq:critical_interpolation_basic}
\begin{split}
 &\|f\|_{C^{m,\alpha}}
 \lesssim \|f\|_{W^{m,\infty}}^{\frac2{2+\alpha}}\|f\|_{C^{2+m,\alpha}}^{\frac\alpha{2+\alpha}},\qquad\qquad\,\,\|f\|_{C^{m+1,\alpha}}
 \lesssim \|f\|_{W^{m,\infty}}^{\frac1{2+\alpha}}\|f\|_{C^{2+m,\alpha}}^{\frac{1+\alpha}{2+\alpha}},\\
 &\|f\|_{W^{m+1,\infty}}\lesssim \|f\|_{W^{m,\infty}}^{\frac{1+\alpha}{2+\alpha}}\|f\|_{C^{2+m,\alpha}}^{\frac1{2+\alpha}},\qquad\quad
\|f\|_{W^{m+2,\infty}}
 \lesssim \|f\|_{W^{m,\infty}}^{\frac\alpha{2+\alpha}}\|f\|_{C^{2+m,\alpha}}^{\frac2{2+\alpha}}.
\end{split}
\end{equation}
\end{lemma}

\begin{proof} The Littlewood--Paley characterization of $C^{2+m,\alpha}(\mathbb{S}^2)$ \eqref{BesovHolder} and the $L^\infty$ boundedness of $\Delta_k$ show that, for any $k\geq1$,
\[
2^{mk} \|\Delta_k f\|_{L^\infty}
 \lesssim \min\{\|f\|_{W^{m,\infty}},2^{-(2+\alpha)k}\|f\|_{C^{2+m,\alpha}}\}.
\]
The block $\Delta_0 f$ is bounded directly by $\|f\|_{W^{m,\infty}}$.  We choose $M$ so
that 
\begin{equation*}
    \begin{aligned}
    2^{(2+\alpha)M}\approx \|f\|_{C^{2+m,\alpha}}\|f\|_{W^{m,\infty}}^{-1},
    \end{aligned}
\end{equation*}
and estimate
\begin{equation*}
    \begin{aligned}
    \|f\|_{B_{\infty,1}^{m+\sigma}}&=\sum_{k\geq0}2^{(m+\sigma) k}\|\Delta_k f\|_{L^\infty}\lesssim \sum_{k\leq M}\|f\|_{W^{m,\infty}} 2^{\sigma k}+\sum_{k\geq M}\|f\|_{C^{2+m,\alpha}}2^{-(2+\alpha-\sigma)k}\\
    &\lesssim
 \|f\|_{W^{m,\infty}}^{1-\frac{\sigma}{2+\alpha}}
 \|f\|_{C^{2+m,\alpha}}^{\frac{\sigma}{2+\alpha}},
    \end{aligned}
\end{equation*}
which proves \eqref{eq:critical_interpolation_scale}. The embeddings
$B_{\infty,1}^{j+\beta}(\mathbb{S}^2)\hookrightarrow C^{j,\beta}(\mathbb{S}^2)$ and
$B_{\infty,1}^{j}(\mathbb{S}^2)\hookrightarrow W^{j,\infty}(\mathbb{S}^2)$ give
\eqref{eq:critical_interpolation_basic}.

\end{proof}

Our next lemma provides the main estimate we use to control the Muskat nonlinearity.

\begin{lemma}[Full multilinear estimate]
\label{lem:total_output_derivative}
Assume $0<\alpha<1$.  Let $m,q$ be nonnegative integers and let $p_m(z)=z_1^{\beta_1}z_2^{\beta_2}z_3^{\beta_3}$, $|\beta|=m$, and assume that $m+q$ is odd.  Assume that $a_1,\ldots,a_q, b_1\in C^{3,\alpha}(\mathbb S^2)$, $B, b_2,b_3\in C^{2,\alpha}(\mathbb S^2)$ and $V,W\in\mathcal{V}\cup\{\mathrm{Id}\}$. Recall the operators $\mathcal{T}_{p_m}^{(q)}$ defined in \eqref{multilinear_model} and define
\begin{equation}\label{sec6eq23}
G=b_1\cdot V\big(b_2\cdot W(b_3)\big),\qquad\mathcal{N}(x)=B(x)\cdot \mathcal T_{p_m}^{(q)}\big[
 a_1,\ldots,a_q, G\big](x).
\end{equation}
Then,  
\begin{equation}\label{eq:L64_estimate}
\begin{aligned}
\big\|\mathcal{N}\big\|_{C^\alpha}&\lesssim_{m+q}\|B\|_{L^\infty}\|b_1\|_{W^{1,\infty}}\|b_2\|_{L^\infty}\|b_3\|_{L^\infty}\prod_{j=1}^q \|a_j\|_{W^{1,\infty}}\\
&\quad\times\Big(
\frac{\|B\|_{C^{2,\alpha}}}{\|B\|_{L^\infty}}+\frac{\|b_1\|_{C^{3,\alpha}}}{\|b_1\|_{W^{1,\infty}}}+\frac{\|b_2\|_{C^{2,\alpha}}}{\|b_2\|_{L^\infty}}+\frac{\|b_3\|_{C^{2,\alpha}}}{\|b_3\|_{L^\infty}}+\sum_{j=1}^q\frac{\|a_j\|_{C^{3,\alpha}}}{\|a_j\|_{W^{1,\infty}}}\Big).
\end{aligned}
\end{equation}
\end{lemma}

\begin{remark}\label{lem:null} The estimate \eqref{eq:L64_estimate} is critical when $V,W\in\mathcal{V}$, and relies on an important null structure: the vector field $V$ appears in front of the main argument $b_2\cdot W(b_3)$ of the operator $\mathcal{T}_{p_m}^{(q)}$. It is not hard to see that the main bound \eqref{eq:L64_estimate} would fail (logarithmically) without this null structure (even if $b_1=1$ and $q+m=1$) if $G=V(b_2)W(b_3)$ or $G=b_2VW(b_3)$. 
\end{remark}

\begin{proof} As in the proof of Lemma \ref{lem:model_nonlinearity}, we may assume that $m=0$, $p_m=1$, so $q$ is odd. We may also assume that $b_1=1$: if $V=\mathrm{Id}$ then we replace $b_2$ with $b_1b_2$, while if $V\in\mathcal{V}$ we rewrite $G=V(b_1b_2\cdot W(b_3))-V(b_1)b_2\cdot W(b_3)$, and incorporate the factor containing $b_1$ with $b_2$. By multilinearity we may also assume that
\begin{equation}\label{normo2}
\|B\|_{L^\infty}=\|b_2\|_{L^\infty}=\|b_3\|_{L^\infty}=\|a_1\|_{W^{1,\infty}}=\ldots=\|a_q\|_{W^{1,\infty}}=1.
\end{equation}
For \eqref{eq:L64_estimate} it suffices to prove that for any $n\geq 0$
\begin{equation}\label{sec6eq24}
2^{\alpha n}\|\Delta_n\mathcal{N}\|_{L^\infty}\lesssim_{q}
\|B\|_{C^{2,\alpha}}+\|b_2\|_{C^{2,\alpha}}+\|b_3\|_{C^{2,\alpha}}+\sum_{j=1}^q \|a_j\|_{C^{3,\alpha}}.
\end{equation}

\textbf{Step 1.} For simplicity of notation, let \begin{equation*}
H_a=\sum_{j=1}^q \|a_j\|_{C^{3,\alpha}}, \qquad H_b=\|b_2\|_{C^{2,\alpha}}+\|b_3\|_{C^{2,\alpha}}.    
\end{equation*}
The bounds \eqref{eq:critical_interpolation_basic} and the product rule show that
\begin{equation}\label{sec6eq26}
\|a_j\|_{C^{1,\alpha}}\lesssim (H_a)^{\frac{\alpha}{2+\alpha}},\qquad\|G\|_{L^\infty}\lesssim (H_b)^{\frac{2}{2+\alpha}},\qquad \|G\|_{C^\alpha}\lesssim H_b,
\end{equation}
for $j\in\{1,\ldots,q\}$. Using Lemma \ref{lem:model_nonlinearity} it follows that, for any $k\in \mathbb{Z}_+$,
\begin{equation}\label{sec6eq27}
2^{\alpha k}\big\|\Delta_k\mathcal T_{p_m}^{(q)}\big[a_1,\ldots,a_q, G\big]\|_{L^\infty}\lesssim_{q} H_a+H_b.
\end{equation}
We decompose $\mathcal{N}=\mathcal{N}_{n}^1+\mathcal{N}_n^2$, where
\begin{equation*}
\begin{split}
&\mathcal{N}_n^1(x)=B(x)\cdot \Delta_{\leq n-10}\mathcal T_{p_m}^{(q)}\big[a_1,\ldots,a_q, G\big](x),\\
&\mathcal{N}_n^2(x)=B(x)\cdot \Delta_{> n-10}\mathcal T_{p_m}^{(q)}\big[a_1,\ldots,a_q, G\big](x).
\end{split}
\end{equation*}
It follows from \eqref{sec6eq27} and \eqref{normo2} that
\begin{equation*}
2^{\alpha n}\|\mathcal{N}_n^2\|_{L^\infty}\lesssim_{q} H_a+H_b.
\end{equation*}
Since
\begin{equation*}
\Delta_n\big[\Delta_{\leq n-5}B\cdot \Delta_{\leq n-10}\mathcal T_{p_m}^{(q)}\big[a_1,\ldots,a_q, G\big]\big]=0,
\end{equation*}
for \eqref{sec6eq24} it suffices to prove that for any $n\geq 10$
\begin{equation*}
2^{\alpha n}\big\|\Delta_{>n-5}B\cdot \Delta_{\leq n-10}\mathcal T_{p_m}^{(q)}\big[a_1,\ldots,a_q, G\big]\big\|_{L^\infty}\lesssim_{q}\|B\|_{C^{2,\alpha}}+H_a+H_b.
\end{equation*}
Since $2^{\alpha n}\|\Delta_{>n-5}B\|_{L^\infty}\lesssim\min(2^{\alpha n}, 2^{-2n}\|B\|_{C^{2,\alpha}})$, for \eqref{sec6eq24} it suffices to prove that 
\begin{equation}\label{sec6eq30}
\big\|\Delta_{\leq n-10}\mathcal T_{p_m}^{(q)}\big[a_1,\ldots,a_q, G\big]\big\|_{L^\infty}\lesssim_{q}2^{2n}+2^{-\alpha n}\big(H_a+H_b\big). 
\end{equation}

\textbf{Step 2.} Let $H_{ab}=H_a+H_b$. We decompose $\mathcal T_{p_m}^{(q)}\big[a_1,\ldots,a_q, G\big]=J_n^1+J_n^2$, where
\begin{equation*}
\begin{split}
J_n^1(x)&=\operatorname{p.v.}\int_{\mathbb S^2}K_a(x,y)G(y)\Phi_{\leq -L}(x-y)\,d\sigma(y),\\
J_n^{2}(x)&=\int_{\mathbb S^2}K_a(x,y)G(y)\Phi_{>-L}(x-y)\,d\sigma(y),
\end{split}
\end{equation*}
with $K_a$ defined as in \eqref{sec6eq5} and $L=L(n,H_{ab})$ will be chosen in \eqref{sec6eq32} below. Recall that $G=V(G')$, $G'=b_2W(b_3)$, so $\|G'\|_{L^\infty}\lesssim (H_b)^{\frac{1}{2+\alpha}}$. If $V\in\mathcal{V}$ we integrate by parts in $y$ and use the bounds $|V_yK_a(x,y)|\lesssim_q|x-y|^{-3}$ for any $V\in\mathcal{V}$ (compare with \eqref{eq:L62_kernel_crude}). Therefore, 
\begin{equation*}
\|J_n^2\|_{L^\infty}\lesssim_q 2^L\|G'\|_{L^\infty}\lesssim_{q} 2^L (H_b)^{\frac{1}{2+\alpha}}.
\end{equation*}
In particular $\|J_n^2\|_{L^\infty}\lesssim_{q}2^{2n}+2^{-\alpha n}H_{ab}$ provided that we choose $L\in\mathbb{Z}_+$ such that
\begin{equation}\label{sec6eq32}
2^L\approx 2^{2n}(H_{ab})^{-\frac{1}{2+\alpha}}+2^{-\alpha n}(H_{ab})^{\frac{1+\alpha}{2+\alpha}}.
\end{equation}
With this choice of $L$, it remains to prove a similar bound for the function $J_n^1$. For this we use the decomposition \eqref{eq:L62_telescoping} and the error bounds \eqref{eq:L62_remainder_pointwise}. The contribution of the error term can be bounded easily,
\begin{equation*}
\Big|\int_{\mathbb S^2}\frac{R_q(x,y)}{|x-y|^{q+2}}G(y)\Phi_{\leq -L}(x-y)\,d\sigma(y)\Big|\lesssim_q \|G\|_{L^\infty}2^{-L\alpha}\Big(\sum_{j=1}^q\|a_j\|_{C^{1,\alpha}}\Big)\lesssim_{q}2^{-\alpha n}H_{ab},
\end{equation*}
using \eqref{sec6eq26} and the observation that $2^L\gtrsim 2^n$. This is compatible with the desired bounds \eqref{sec6eq30}. Therefore it remains to control the contribution in $J_n^1$ of the main term in the decomposition \eqref{eq:L62_telescoping}. To estimate this term, we recall the notation \eqref{eq:L62_defi}.
Then, it suffices to prove that
\begin{equation}\label{sec6eq33}
\big\|\Delta_{\leq n-10}\big[c_I\cdot\mathcal{T}_{P_I,L}^{(0)}[G]\big]\big\|_{L^\infty}\lesssim_{q}2^{2n}+2^{-\alpha n}H_{ab},
\end{equation}
for any $n\geq 10$ and $I\in\{1,2,3\}^q$, where the operators $\mathcal{T}_{P_I,L}^{(0)}$ are defined in \eqref{linear_model}.
\medskip

{\bf{Step 3.}} To prove \eqref{sec6eq33} we write
\begin{equation*}
\mathcal{T}_{P_I,L}^{(0)}[G](x)=\operatorname{p.v.}
		\int_{\mathbb S^2}\frac{P_I(x-y)}{|x-y|^{q+2}}\Phi_{\leq -L}(x-y)[G(y)-G(x)]\,d\sigma(y)+G(x)\cdot \mathcal{T}_{P_I,L}^{(0)}[1](x).
\end{equation*}
Using Lemma \ref{lem:model_linear} we estimate, for any $x\in\mathbb{S}^2$,
\begin{equation*}
\begin{split}
\big|\mathcal{T}_{P_I,L}^{(0)}[G](x)\big|&\lesssim_q\int_{\mathbb S^2}\frac{|\Phi_{\leq -L}(x-y)|}{|x-y|^{2-\alpha}}\|G\|_{C^\alpha}\,d\sigma(y)+\|G\|_{L^\infty}\\
&\lesssim_q 2^{-\alpha L}\|G\|_{C^\alpha}+\|G\|_{L^\infty}\lesssim_{q}2^{-\alpha n}H_{ab}+(H_{ab})^{\frac{2}{2+\alpha}},
\end{split}
\end{equation*}
where we used $2^L\gtrsim 2^n$ and \eqref{sec6eq26} in the last bound. The desired bounds \eqref{sec6eq33} follow since $\|c_I\|_{L^\infty}\lesssim_q 1$. This completes the proof of the lemma.

\end{proof}

\subsection{Decomposition of the Muskat nonlinearity} We organize all the geometric quantities into their zeroth-order, linear, and higher-order parts. In the rest of this section we assume that $f\in C^{3,\alpha}(\mathbb{S}^2)$ and $\|f\|_{W^{1,\infty}}\ll 1$. Recall the definition \eqref{rhodef} and set
\begin{equation}\label{eq:sec62_u_decomposition}
 \begin{aligned}
     r(x)&=1+r_L(x)+r_H(x), \qquad
     r_L(x)=f(x),\qquad r_H(x)=r(x)-r_L(x)-1.
 \end{aligned}
\end{equation}
Let
\begin{equation*}
 N(f)(x)=\frac{\mathcal N_f(x)}{r(x)}=r(x)^2x-r(x)^{-1}\nabla_{\mathbb{S}^2} f(x)=r(x)^2x-r(x)\nabla_{\mathbb{S}^2}r(x),
\end{equation*}
where \(\mathcal N_f\) was defined in \eqref{def:nonunit_normal_f}. The contour equation
	\eqref{contour_eq} becomes
	\begin{equation}\label{nonlinear_contour}
	\partial_t f=\frac1R\bigl(BR(S,\omega)-\dot c\bigr)\cdot N(f).
	\end{equation}
We decompose
	\begin{equation}\label{AH_decomposition}
	\begin{split}
	&N(f)=N_0+N_L+N_H,\\
	&N_0=x,\qquad N_L=2fx-\nabla_{\mathbb{S}^2} f,\qquad
	N_H=N(f)-N_0-N_L.
	\end{split}
	\end{equation}
Since our surface is given as a radial graph over the sphere, one can define it as a zero level set of a function and compute its normal in terms of the gradient of such a function. Taking the divergence of the extended unit normal, we obtain the expression of the dimensionless curvature $\mathsf{K}=RK[S]$ (see e.g. \cite[Section~2.2]{PrussSimonett2016}),
\begin{equation*}
\mathsf{K}=\frac{2r^3+3r|\nabla_{\mathbb{S}^2} r|^2-(r^2+|\nabla_{\mathbb{S}^2} r|^2)\Delta_{\mathbb{S}^2} r+\frac12\nabla_{\mathbb{S}^2}r\cdot\nabla_{\mathbb{S}^2}|\nabla_{\mathbb{S}^2}r|^2}{2r(r^2+|\nabla_{\mathbb{S}^2} r|^2)^{3/2}}.
\end{equation*}
We decompose
\begin{equation*}
 \mathsf{K}=1+\mathsf{K}_L+\mathsf{K}_H,
 \qquad
 \mathsf{K}_L=-\frac12(\Delta_{\mathbb{S}^2}+2)f.
\end{equation*}
Using $r=1+f+r_H$, the higher-order part has the exact representation
\begin{equation}\label{eq:sec62_kappa_H_exact}
\begin{split}
 \mathsf{K}_H&=-\frac12(\Delta_{\mathbb{S}^2}+2)r_H+\Big(\frac{2r^2+3|\nabla_{\mathbb{S}^2} r|^2}{2(r^2+|\nabla_{\mathbb{S}^2} r|^2)^{3/2}}+r-2\Big)\\
 &\qquad+\frac{r(r^2+|\nabla_{\mathbb{S}^2} r|^2)^{1/2}-1}{2r(r^2+|\nabla_{\mathbb{S}^2} r|^2)^{1/2}}\Delta_{\mathbb{S}^2} r
+\frac{\frac12\nabla_{\mathbb{S}^2}r\cdot\nabla_{\mathbb{S}^2}|\nabla_{\mathbb{S}^2}r|^2}{2r(r^2+|\nabla_{\mathbb{S}^2} r|^2)^{3/2}}.
 \end{split}
\end{equation}
Notice that every term in $\mathsf{K}_H$ is at least quadratic in $f$, and every term is affine in the second-order derivatives of $r$.

We are now ready to analyze the vorticity functions $\omega_\sigma,\omega_\rho$. With $V_1, V_2, V_3$ defined as in \eqref{Vector-fields}, we can easily verify the identities
\begin{equation}\label{eq:sec62_rotation_identities}
\begin{split}
&  (x\wedge\nabla_{\mathbb{S}^2} h)_j=V_jh,
 \qquad
 \sum_{j=1}^3V_j(\nabla_{\mathbb{S}^2} h)_j=0,\\
 &(\nabla_{\mathbb{S}^2}(rx_i))_j=(\nabla_{\mathbb{S}^2}x_i)_j+(r-1)(\nabla_{\mathbb{S}^2}x_i)_j+x_i(\nabla_{\mathbb{S}^2}r)_j.
 \end{split}
\end{equation}
The surface-tension vorticity has the formula 
\begin{equation*}
\begin{split}
&\omega_\sigma=A_\sigma D_{\mathbb{S}^2}(rx)[x\wedge\nabla_{\mathbb{S}^2}\mathsf{K}],\qquad(\omega_\sigma)_i
 =A_\sigma\sum_{j=1}^3V_j\big[(\mathsf{K}-1)(\nabla_{\mathbb{S}^2}(rx_i))_j\big],
 \end{split}
\end{equation*}
where we used the identities in \eqref{eq:sec62_rotation_identities}. Using also the last identity, we have our main decomposition
\begin{equation}\label{eq:sec62_omega_sigma_decomposition}
\begin{split}
&\omega_\sigma=\omega_{\sigma,L}+\omega_{\sigma,H},\\
&(\omega_{\sigma,L})_i=A_\sigma\sum_{j=1}^3V_j(\mathsf{K}_L(\nabla_{\mathbb{S}^2}x_i)_j)=A_\sigma V_i\mathsf{K}_L=-\frac{A_\sigma}{2}V_i(\Delta_{\mathbb{S}^2}+2)f,\\
&(\omega_{\sigma,H})_i=A_\sigma\sum_{j=1}^3V_j\big[\mathsf{K}_H(\nabla_{\mathbb{S}^2}x_i)_j
 +(\mathsf{K}_L+\mathsf{K}_H)\bigl((r-1)(\nabla_{\mathbb{S}^2}x_i)_j+x_i (\nabla_{\mathbb{S}^2}r)_j\bigr)\big].
 \end{split}
\end{equation}
In particular, the vector field $V_j$ remains in front of the entire higher-order expression.

For the density contribution, the definition of the vorticity gives the
exact identity
\begin{equation*}
 \omega_\rho=A_\rho R^2e_3\wedge N(f).
\end{equation*}
Consequently, using \eqref{AH_decomposition},
\begin{equation}\label{eq:sec62_omega_rho_decomposition}
 \begin{split}
 \omega_\rho&=\omega_{\rho,0}+\omega_{\rho, L}+\omega_{\rho, H},\\
 \omega_{\rho,0}&=A_\rho R^2e_3\wedge x,\\
 \omega_{\rho,L}&=A_\rho R^2\big[2fe_3\wedge x-e_3\wedge\nabla_{\mathbb{S}^2} f\big],\\
 \omega_{\rho,H}&=A_\rho R^2\big[-(r-1)^2\big(1+2(r-1)/3\big)e_3\wedge x+(r-1)re_3\wedge\nabla_{\mathbb{S}^2} r\big].
 \end{split}
\end{equation}
Finally, we need to decompose the Birkhoff--Rott operator $BR(S,\cdot)$ (compare with the main evolution equation \eqref{nonlinear_contour}). Recall the general formula
\begin{equation}\label{sec6eq41}
\begin{split}
BR(S, h)(x)&=-\frac{1}{4\pi}\operatorname{p.v.}\int_{\mathbb{S}^2}\frac{S(x)-S(y)}{|S(x)-S(y)|^3}\wedge h(y)\,d\sigma(y)\\
&=-\frac{1}{4\pi R^2}\operatorname{p.v.}\int_{\mathbb{S}^2}\frac{r(x)x-r(y)y}{|r(x)x-r(y)y|^3}\wedge h(y)\,d\sigma(y).
\end{split}
\end{equation}
Notice that, for $x,y\in\mathbb{S}^2$,
\begin{equation*}
\begin{split}
&r(x)x-r(y)y=(x-y)+[f(x)x-f(y)y]+[r_H(x)x-r_H(y)y],\\
&|r(x)x-r(y)y|^2=r(x)r(y)|x-y|^2+(r(x)-r(y))^2.
\end{split}
\end{equation*}
Therefore we can decompose
\begin{equation}\label{sec6eq43}
\begin{split}
\frac{r(x)x-r(y)y}{|r(x)x-r(y)y|^3}&=\mathcal{K}_0(x,y)+\mathcal{K}_L(x,y)+\mathcal{K}_H(x,y),\\
\mathcal{K}_0(x,y)&=\frac{x-y}{|x-y|^3},\\
 \mathcal{K}_L(x,y)&=-\frac{3(f(x)+f(y))(x-y)}{2|x-y|^3}+\frac{f(x)x-f(y)y}{|x-y|^3},\\
  \mathcal K_H(x,y)&=\frac{r(x)x-r(y)y}{|r(x)x-r(y)y|^3}-\mathcal K_0(x,y)-\mathcal K_L(x,y).
  \end{split}
\end{equation}
Letting 
\begin{equation*}
v(x,y)=\Big[r(x)r(y)+\frac{[r(x)-r(y)]^2}{|x-y|^2}\Big]^{-3/2},
\end{equation*}
the higher-order kernel can be written as 
\begin{equation*}
 \begin{split}
 \mathcal K_H(x,y)&=\frac{(x-y)}{|x-y|^3}\Big[v(x,y)-1+\frac{3(f(x)+f(y))}{2}\Big]+(v(x,y)-1)\frac{f(x)x-f(y)y}{|x-y|^3}\\
 &+v(x,y)\frac{r_H(x)x-r_H(y)y}{|x-y|^3}.
 \end{split}
\end{equation*}

We can summarize the preceding decompositions as follows.

\begin{prop}\label{prop:nonlinear_terms}
The source in \eqref{nonlinear_contour} can be written as
\begin{equation*}
 \frac1R\bigl(BR(S,\omega)-\dot c\bigr)\cdot N(f)
 =\mathcal L[f]+\mathcal Q(f),
\end{equation*}
where \(\mathcal L\) is the linear operator defined in
\eqref{linear_operator}. More precisely,
\begin{equation}\label{eq:sec62_BR_linear_pieces}
\begin{split}
 BR_0(x)&=-\frac1{4\pi R^2}\operatorname{p.v.}\!\int_{\mathbb S^2}
 \mathcal K_0(x,y)\wedge\omega_{\rho,0}(y)\,d\sigma(y),\\
 BR_{\sigma,L}(x)&=-\frac1{4\pi R^2}\operatorname{p.v.}\!\int_{\mathbb S^2}
 \mathcal K_0(x,y)\wedge\omega_{\sigma,L}(y)\,d\sigma(y),\\
 BR_{\rho,L}(x)&=-\frac1{4\pi R^2}\operatorname{p.v.}\!\int_{\mathbb S^2}
 \big[\mathcal K_0(x,y)\wedge\omega_{\rho,L}(y)
 +\mathcal K_L(x,y)\wedge\omega_{\rho,0}(y)\big]\,d\sigma(y),
\end{split}
\end{equation}
and
\begin{equation}\label{nonlinear_remainder}
\begin{split}
 R\mathcal Q(f)=&
 \big(BR_\sigma-BR_{\sigma,L}\big)\cdot N(f)
 +\big(BR_{\sigma,L}-\dot c_L\big)\cdot(N_L+N_H)\\
 &+\big(BR_\rho-BR_0-BR_{\rho,L}\big)\cdot N(f)
 +BR_{\rho,L}\cdot(N_L+N_H)
 +\big(BR_0-\dot c_0\big)\cdot N_H.
\end{split}
\end{equation}
The two nonlinear Birkhoff--Rott remainders in this formula are
\begin{equation}\label{eq:sec62_BR_remainders}
\begin{aligned}
 BR_\sigma-BR_{\sigma,L}
 =-\frac1{4\pi R^2}\operatorname{p.v.}\int_{\mathbb S^2}
 \Big(&\mathcal K_0\wedge\omega_{\sigma,H}+\mathcal K_L\wedge(\omega_{\sigma,L}+\omega_{\sigma,H})
 \\&+\mathcal K_H\wedge(\omega_{\sigma,L}+\omega_{\sigma,H})
 \Big)(x,y)\,d\sigma(y),\\
 BR_\rho-BR_0-BR_{\rho,L}
 =-\frac1{4\pi R^2}\operatorname{p.v.}\!\int_{\mathbb S^2}
 \Big(&\mathcal K_0\wedge\omega_{\rho,H}+\mathcal K_L\wedge(\omega_{\rho,L}+\omega_{\rho,H})\\
 &+\mathcal K_H\wedge(\omega_{\rho,0}+\omega_{\rho,L}+\omega_{\rho,H})
 \Big)(x,y)\,d\sigma(y),
\end{aligned}
\end{equation}
where every vorticity factor in \eqref{eq:sec62_BR_remainders} is
evaluated at \(y\), while every kernel is evaluated at \((x,y)\).
\end{prop}

\begin{proof}
Insert \eqref{AH_decomposition},
\eqref{eq:sec62_omega_sigma_decomposition},
\eqref{eq:sec62_omega_rho_decomposition}, and \eqref{sec6eq43} into
\eqref{sec6eq41}. The zeroth-order and linear contributions are
precisely the three terms in \eqref{eq:sec62_BR_linear_pieces}; all
remaining kernel--vorticity pairings give
\eqref{eq:sec62_BR_remainders}. Hence
\begin{align*}
 (BR(S,\omega)-\dot c)\cdot N(f)
 ={}&(BR_0-\dot c_0)\cdot N_0\\
 &+(BR_{\sigma,L}-\dot c_L)\cdot N_0
 +BR_{\rho,L}\cdot N_0
 +(BR_0-\dot c_0)\cdot N_L\\
 &+(BR_\sigma-BR_{\sigma,L})\cdot N(f)
 +(BR_{\sigma,L}-\dot c_L)\cdot(N_L+N_H)\\
 &+(BR_\rho-BR_0-BR_{\rho,L})\cdot N(f)
 +BR_{\rho,L}\cdot(N_L+N_H)\\
 &+(BR_0-\dot c_0)\cdot N_H.
\end{align*}
By \eqref{c0_cancellation},
\[
 (BR_0-\dot c_0)\cdot N_0
 =(BR_0-\dot c_0)\cdot x=0.
\]
After division by \(R\), the second line is exactly
\(\mathcal L[f]\), by \eqref{linear_operator}, and the remaining
lines give \eqref{nonlinear_remainder}.

\end{proof}

\subsection{Nonlinear estimates}

\begin{prop}[Nonlinear estimate]\label{lem:Nonlinear_bounds}
Let \(f\in C^{3,\alpha}(\mathbb S^2)\), and let \(\mathcal Q(f)\) be
defined in \eqref{nonlinear_remainder}. There exists
\(\varepsilon>0\) such that, if
\(\|f\|_{W^{1,\infty}}\leq\varepsilon\), then
\begin{equation}\label{eq:nonlinear_bound}
 \|\mathcal Q(f)\|_{C^\alpha}
 \lesssim
 \left(\frac{|A_\rho|}{R}+\frac{A_\sigma}{R^3}\right)
 \|f\|_{W^{1,\infty}}\|f\|_{C^{3,\alpha}}.
\end{equation}
\end{prop}

\begin{proof}
The standard composition and product estimates, together with
\eqref{eq:sec62_u_decomposition} and \eqref{AH_decomposition}, give
\begin{equation}\label{eq:sec62_composition_bounds}
\begin{split}
 &\|r-1\|_{W^{1,\infty}}+\|N(f)-N_0\|_{L^\infty}
 \lesssim\|f\|_{W^{1,\infty}},\\
 &\|r-1\|_{C^{3,\alpha}}+\|N(f)-N_0\|_{C^{2,\alpha}}
 \lesssim\|f\|_{C^{3,\alpha}},\\
 &\|r_H\|_{W^{1,\infty}}+\|N_H\|_{L^\infty}
 \lesssim\|f\|_{W^{1,\infty}}^2,\\
 &\|r_H\|_{C^{3,\alpha}}+\|N_H\|_{C^{2,\alpha}}
 \lesssim\|f\|_{W^{1,\infty}}\|f\|_{C^{3,\alpha}}.
\end{split}
\end{equation}
We use these bounds below without further comment.
We first estimate the terms carrying the factor \(A_\sigma\). From
\eqref{nonlinear_remainder}, their sum before division by \(R\) is
\begin{equation}\label{eq:sec62_sigma_collection}
 BR(S,\omega_\sigma)\cdot N(f)-BR_{\sigma,L}\cdot N_0
 -\dot c_L\cdot\bigl(N(f)-N_0\bigr).
\end{equation}
The exact decompositions in \eqref{AH_decomposition},
\eqref{eq:sec62_omega_sigma_decomposition}, and \eqref{sec6eq43}
give
\begin{equation*}
\begin{split}
BR(S,\omega_\sigma)&\cdot N(f)-BR_{\sigma,L}\cdot N_0=-\frac{1}{4\pi R^2}N(f)(x)\cdot
\operatorname{p.v.}\int_{\mathbb S^2}
\mathcal K_0(x,y)\wedge\omega_{\sigma,H}(y)\,d\sigma(y)\\
&-\frac{1}{4\pi R^2}N(f)(x)\cdot
\operatorname{p.v.}\int_{\mathbb S^2}
\bigl(\mathcal K_L(x,y)+\mathcal K_H(x,y)\bigr)
 \wedge\bigl(\omega_{\sigma,L}(y)+\omega_{\sigma,H}(y)\bigr)
 \,d\sigma(y)\\
&-\frac{1}{4\pi R^2}\bigl(N(f)(x)-N_0(x)\bigr)\cdot
\operatorname{p.v.}\int_{\mathbb S^2}
\mathcal K_0(x,y)\wedge\omega_{\sigma,L}(y)\,d\sigma(y).
\end{split}
\end{equation*}

The structural information needed to apply Lemma
\ref{lem:total_output_derivative} is already contained in
\eqref{eq:sec62_kappa_H_exact} and
\eqref{eq:sec62_omega_sigma_decomposition}. Indeed,
$\Delta_{\mathbb S^2}=V_1^2+V_2^2+V_3^2$, and every component of
$A_\sigma^{-1}\omega_{\sigma,L}$ is a finite sum of expressions of
the form in \eqref{sec6eq23}, with the right-hand side in
\eqref{eq:L64_estimate} bounded by
$C\|f\|_{C^{3,\alpha}}$. Using the  formula
\eqref{eq:sec62_kappa_H_exact}, and absorbing its smooth coefficient
functions into $b_1$ or $b_2$, every component of
$A_\sigma^{-1}\omega_{\sigma,H}$ is a finite sum of terms of the
form in \eqref{sec6eq23}. Lemma \ref{lem:interpolation} and the tame
composition estimates show that the corresponding right-hand side in
\eqref{eq:L64_estimate} is bounded by
\begin{equation*}
	 C\|f\|_{W^{1,\infty}}\|f\|_{C^{3,\alpha}}.
\end{equation*}
Here the term containing $\Delta_{\mathbb S^2}r_H$ is treated by
taking $b_3=V_\ell r_H$ and $W=V_\ell$. The terms containing
$\Delta_{\mathbb S^2}r$ are treated by taking
$b_3=V_\ell r$. The last term of
\eqref{eq:sec62_kappa_H_exact} is written as
a finite sum of combinations $V_iV_j$, $i,j\in\{1,2,3\}$, with smooth
coefficients, and the two first-order factors are absorbed in
$b_2$. The bounds follow from Lemma \ref{lem:interpolation} and
\eqref{eq:sec62_composition_bounds}.

For the nonlinear kernel we use the absolutely convergent identity
\begin{equation}\label{eq:sec62_kernel_binomial_series}
\begin{split}
\frac{r(x)x-r(y)y}{|r(x)x-r(y)y|^3}
={}&\sum_{k\geq0}\binom{-3/2}{k}
 r(x)^{-1/2-k}r(y)^{-3/2-k}
 \frac{(x-y)[r(x)-r(y)]^{2k}}{|x-y|^{2k+3}}\\
&+\sum_{k\geq0}\binom{-3/2}{k}
 r(x)^{-3/2-k}r(y)^{-3/2-k}
 \frac{y[r(x)-r(y)]^{2k+1}}{|x-y|^{2k+3}}.
\end{split}
\end{equation}
The operators in the first sum have $m=1,q=2k$, while those in
the second have $m=0,q=2k+1$. Thus $m+q$ is odd in both cases.
Moreover, for $\lambda\in\{1/2,3/2\}$,
\begin{equation}\label{eq:sec62_power_bounds}
\begin{aligned}
 \|r^{-k-\lambda}\|_{W^{1,\infty}}
 &\lesssim C^{k+1},&
 \|r^{-k-\lambda}\|_{C^{3,\alpha}}
 &\lesssim C^{k+1}\bigl(1+\|f\|_{C^{3,\alpha}}\bigr),\\
 \|N(f)r^{-k-\lambda}\|_{L^\infty}
 &\lesssim C^{k+1},&
 \|N(f)r^{-k-\lambda}\|_{C^{2,\alpha}}
 &\lesssim C^{k+1}\bigl(1+\|f\|_{C^{3,\alpha}}\bigr).
\end{aligned}
\end{equation}
Apply Lemma \ref{lem:total_output_derivative} componentwise to the terms corresponding to the decomposition
\eqref{eq:sec62_kernel_binomial_series}. The terms in the first sum
paired with \(\omega_{\sigma,L}\) are bounded by
\begin{equation*}
	\frac{A_\sigma}{R^2}C^k
	\|f\|_{W^{1,\infty}}^{2k}\|f\|_{C^{3,\alpha}},
\end{equation*}
and those paired with \(\omega_{\sigma,H}\) are bounded by
\begin{equation*}
	 \frac{A_\sigma}{R^2}C^k
	\|f\|_{W^{1,\infty}}^{2k+1}\|f\|_{C^{3,\alpha}}.
\end{equation*}
The corresponding bounds for the two pairings in the second sum  are
\begin{equation*}
	 \frac{A_\sigma}{R^2}C^k
	\|f\|_{W^{1,\infty}}^{2k+1}\|f\|_{C^{3,\alpha}}
	\quad\hbox{and}\quad
	\frac{A_\sigma}{R^2}C^k
	\|f\|_{W^{1,\infty}}^{2k+2}\|f\|_{C^{3,\alpha}}.
\end{equation*}
All these series are summable if
$\|f\|_{W^{1,\infty}}$ is sufficiently small. The estimates are
uniform for the truncated principal-value integrals, so the limit
can be taken after summation.

The only term in these bounds without a small factor is the
$k=0$ term in the first sum paired with
$\omega_{\sigma,L}$. Its difference from its linear part is
\begin{equation}\label{eq:sec62_sigma_exceptional_term}
\begin{split}
&-\frac{1}{4\pi R^2}
 \bigl(N(f)r^{-1/2}-N_0\bigr)\cdot
 \operatorname{p.v.}\int_{\mathbb S^2}
 \frac{x-y}{|x-y|^3}\wedge
 \bigl(r(y)^{-3/2}\omega_{\sigma,L}(y)\bigr)\,d\sigma(y)\\
&-\frac{1}{4\pi R^2}N_0\cdot
 \operatorname{p.v.}\int_{\mathbb S^2}
 \frac{x-y}{|x-y|^3}\wedge
 \bigl((r(y)^{-3/2}-1)\omega_{\sigma,L}(y)\bigr)\,d\sigma(y).
\end{split}
\end{equation}
The composition estimates give
\begin{equation*}
	\begin{split}
	\|N(f)r^{-1/2}-N_0\|_{L^\infty}
	&\lesssim\|f\|_{W^{1,\infty}},\qquad
	\|N(f)r^{-1/2}-N_0\|_{C^{2,\alpha}}
	\lesssim\|f\|_{C^{3,\alpha}},\\
	\|r^{-3/2}-1\|_{W^{1,\infty}}
	&\lesssim\|f\|_{W^{1,\infty}},\qquad
	\|r^{-3/2}-1\|_{C^{3,\alpha}}
	\lesssim\|f\|_{C^{3,\alpha}}.
	\end{split}
\end{equation*}
Lemma \ref{lem:total_output_derivative} applied once more to
\eqref{eq:sec62_sigma_exceptional_term} therefore yields
\begin{equation}\label{eq:sec62_BR_sigma_bound}
 \bigl\|BR(S,\omega_\sigma)\cdot N(f)
       -BR_{\sigma,L}\cdot N_0\bigr\|_{C^\alpha}
 \lesssim
 \frac{A_\sigma}{R^2}
 \|f\|_{W^{1,\infty}}\|f\|_{C^{3,\alpha}}.
\end{equation}
Since
\begin{equation*}
	 |\dot c_L|\lesssim\frac{A_\sigma}{R^2}\|f\|_{L^\infty},
\end{equation*}
we also have
\begin{equation}\label{eq:sec62_sigma_translation_bound}
 \|\dot c_L\cdot(N(f)-N_0)\|_{C^\alpha}
 \lesssim
 \frac{A_\sigma}{R^2}
 \|f\|_{W^{1,\infty}}\|f\|_{C^{3,\alpha}}.
\end{equation}
Combining \eqref{eq:sec62_sigma_collection},
\eqref{eq:sec62_BR_sigma_bound}, and
\eqref{eq:sec62_sigma_translation_bound}, and dividing by $R$,
gives the contribution of the second term on the right-hand side of
\eqref{eq:nonlinear_bound}.

It remains to estimate the terms carrying the factor $A_\rho$.
Every component of $(A_\rho R^2)^{-1}\omega_{\rho,0}$ is a fixed
smooth function. By \eqref{eq:sec62_omega_rho_decomposition}, every
component of $(A_\rho R^2)^{-1}\omega_{\rho,L}$ is a finite sum of
expressions of the form in \eqref{sec6eq23} containing exactly one
factor depending on $f$, whereas every component of
$(A_\rho R^2)^{-1}\omega_{\rho,H}$ is such a sum containing at
least two factors depending on \(f\). Here one takes
$V=\mathrm{Id}$, and one may take either $W=\mathrm{Id}$ or
$W\in\mathcal V$.

There are only three configurations in which the degree count requires an
explicit subtraction. In the first series with $k=0$ and
$\omega_{\rho,0}$, we subtract the frozen term and the first-order
Taylor polynomial of the two powers of $r$; these give $BR_0$ and
the corresponding part of $BR_{\rho,L}$. In the first series with
$k=0$ and $\omega_{\rho,L}$, we subtract the term obtained by
replacing both powers of $r$ by $1$; this is the
$\mathcal K_0\wedge\omega_{\rho,L}$ part of $BR_{\rho,L}$.
Finally, in the second series with $k=0$ and $\omega_{\rho,0}$,
we write $r=1+r_L+r_H$ and subtract the term obtained by replacing
$r(x)-r(y)$ with $f(x)-f(y)$ and all powers of $r$ by $1$;
this is the remaining part of
$\mathcal K_L\wedge\omega_{\rho,0}$. The Taylor remainders in the
powers of \(r\) contain two factors depending on $f$, while $r_H$
satisfies the last two bounds in \eqref{eq:sec62_composition_bounds}.
Thus every term left after these subtractions has at least two
$f$-dependent factors and has the form required in Lemma
\ref{lem:total_output_derivative}.
We apply Lemma \ref{lem:total_output_derivative} to the two series in
\eqref{eq:sec62_kernel_binomial_series}, now using the density
vorticities. After the zeroth-order and linear terms in
\eqref{eq:sec62_BR_linear_pieces} are removed, every remaining term
contains at least two factors depending on $f$. The tame
right-hand side in \eqref{eq:L64_estimate}, together with
\eqref{eq:sec62_composition_bounds} and
\eqref{eq:sec62_power_bounds}, is therefore bounded by
\begin{equation*}
	 |A_\rho|C^k
	\|f\|_{W^{1,\infty}}^{k+1}\|f\|_{C^{3,\alpha}}
\end{equation*}
after relabeling the summation index. Summing the two series gives
\begin{equation}\label{eq:sec62_BR_rho_bound}
\begin{split}
 \big\|&\big(BR_\rho-BR_0-BR_{\rho,L}\big)\cdot N(f)
 +BR_{\rho,L}\cdot(N_L+N_H)\big\|_{C^\alpha}\lesssim
 |A_\rho|\|f\|_{W^{1,\infty}}\|f\|_{C^{3,\alpha}}.
\end{split}
\end{equation}
Finally, \eqref{c0_cancellation} and
\eqref{eq:sec62_composition_bounds} imply
\begin{equation}\label{eq:sec62_rho_translation_bound}
 \|\big(BR_0-\dot c_0\big)\cdot N_H\|_{C^\alpha}
 \lesssim
 |A_\rho|\|f\|_{W^{1,\infty}}\|f\|_{C^{3,\alpha}}.
\end{equation}
The last three terms in \eqref{nonlinear_remainder} are controlled by
\eqref{eq:sec62_BR_rho_bound} and
\eqref{eq:sec62_rho_translation_bound}. Dividing by $R$ and
combining this with the surface-tension estimate proves
\eqref{eq:nonlinear_bound}.
\end{proof}

\begin{lemma}[Difference estimate]\label{lem:nonlinear_difference}
Let $f,g\in C^{3,\alpha}(\mathbb S^2)$, $0<\alpha<1$, and let $\mathcal Q$
be defined in \eqref{nonlinear_remainder}. There exists $\varepsilon>0$ such that, if
$\max\{\|f\|_{W^{1,\infty}},\|g\|_{W^{1,\infty}}\}\leq\varepsilon$, then
\begin{equation}\label{eq:nonlinear_difference}
\begin{aligned}
\|\mathcal Q(f)-\mathcal Q(g)\|_{C^\alpha}
\lesssim \left(\frac{|A_\rho|}{R}+\frac{A_\sigma}{R^3}\right)
\Big[&
  \bigl(\|f\|_{W^{1,\infty}}+\|g\|_{W^{1,\infty}}\bigr)
  \|f-g\|_{C^{3,\alpha}}\\
&+\bigl(\|f\|_{C^{3,\alpha}}+\|g\|_{C^{3,\alpha}}\bigr)
  \|f-g\|_{W^{1,\infty}}\Big].
\end{aligned}
\end{equation}
\end{lemma}

\begin{proof}
Expanding the analytic coefficient functions in the proof of
Proposition \ref{lem:Nonlinear_bounds}, together with \eqref{eq:sec62_kernel_binomial_series}, and subtracting the constant and
linear terms gives
\begin{equation*}
	 \mathcal Q(f)=\sum_{d\geq2}Q_d(f,\ldots,f),
\end{equation*}
where $Q_d$ is $d$-linear. At kernel index $k$, the Taylor coefficients
of $r^{-k-\lambda}$ of degree $j$ are bounded by $C^{k+j}$, while
the kernel contains at least $2k$ increments. Thus the coefficient
growth is at most exponential in the total degree.
Keeping the outer rotations in \eqref{sec6eq23}, Lemma \ref{lem:total_output_derivative}, and the tame product estimates yield
\begin{equation}\label{eq:polarized_nonlinearity}
  \|Q_d(f_1,\ldots,f_d)\|_{C^\alpha}
  \leq\Big(\frac{|A_\rho|}{R}+\frac{A_\sigma}{R^3}\Big) C_\alpha^d
  \sum_{i=1}^d\|f_i\|_{C^{3,\alpha}}
       \prod_{j\ne i}\|f_j\|_{W^{1,\infty}}.
\end{equation}
Here coefficients involving first derivatives are kept in the
$C^{2,\alpha}$ slots of Lemma \ref{lem:total_output_derivative}.

Telescoping each multilinear term and using
\eqref{eq:polarized_nonlinearity}, we obtain
\begin{equation*}
	\begin{aligned}
	\|Q_d(f,\ldots,f)-&Q_d(g,\ldots,g)\|_{C^\alpha}
	\leq \Big(\frac{|A_\rho|}{R}\!+\!\frac{A_\sigma}{R^3}\Big) C_\alpha^d\Big[d\big(\|f\|_{W^{1,\infty}}\!+\!\|g\|_{W^{1,\infty}}\big)^{d-1}\|f-g\|_{C^{3,\alpha}}\\
    &\quad+d(d-1)\big(\|f\|_{W^{1,\infty}}+\|g\|_{W^{1,\infty}}\big)^{d-2}\big(\|f\|_{C^{3,\alpha}}+\|g\|_{C^{3,\alpha}}\big)\|f-g\|_{W^{1,\infty}}\Big].
	\end{aligned}
\end{equation*}
Summing over $d\geq2$, for sufficiently small $\varepsilon$,
proves \eqref{eq:nonlinear_difference}.

\end{proof}

	\section{Proof of the Main Theorem}\label{sec:proofmainthm}

	Setting $h(x,t)=f(x,2R^3t/A_\sigma)$, Sections \ref{sec:linearization} and \ref{sec:nonlinear} give
\begin{equation*}
	 \partial_t h=\mathcal{L}_a [h]+\frac{2R^3}{A_\sigma}\mathcal{Q}(h),
\end{equation*}
where we used the rescaled version of the linear operator $\mathcal{L}$, given in \eqref{linear_spectral}.
We construct the solution in $X_\gamma$ using the map
\begin{equation}\label{Tmap}
  \Xi(h)(t)=e^{t\mathcal{L}_a}f_0
  +\frac{2R^3}{A_\sigma}\int_0^t
      e^{(t-\tau)\mathcal{L}_a}\mathcal{Q}(h)(\tau)\,d\tau.
\end{equation}
Fix $\gamma>0$ as in Lemma \ref{lem:semigroup_est}. Constants below may depend on $\alpha$
and $|A_\rho|R^2/A_\sigma$.

\begin{prop}\label{prop:fixed_point}
There exist $C_*\geq1$ and $\varepsilon>0$ such that, for every mean-zero
$f_0\in W^{1,\infty}(\mathbb{S}^2)$ with
$\|f_0\|_{W^{1,\infty}}\leq\varepsilon$, the map $\Xi$ has a unique
fixed point in $\{h\in X_\gamma:\|h\|_{X_\gamma}\leq2C_*\varepsilon\}$.
\end{prop}

\begin{proof}
Let $\mathcal{B}_\rho=\{h\in X_\gamma:\|h\|_{X_\gamma}\leq\rho\}$.
By Lemma \ref{lem:semigroup_est}, we obtain
\begin{equation*}
  \|e^{t\mathcal{L}_a}f_0\|_{X_\gamma}
  \leq C_*\|f_0\|_{W^{1,\infty}},
\end{equation*}
for a constant $C_*\geq1$. To bound the nonlinear remainder, we use
Proposition \ref{lem:Nonlinear_bounds} to obtain
\begin{equation}\label{Nonlinear_decay_bound}
  \frac{2R^3}{A_\sigma}\|\mathcal{Q}(h)(s)\|_{C^\alpha}
  \lesssim s^{-\frac{2+\alpha}{3}}e^{-\gamma s}\|h\|_{X_\gamma}^2,
\end{equation}
for $s>0$ and $h\in\mathcal{B}_{2C_*\varepsilon}$, provided that
$\varepsilon$ is sufficiently small.
On the surface $\Gamma_h$ associated with a fixed profile $h$, the
oriented area element is $R^2N(h)\,d\sigma=\nu_h\,dS$.
The normal trace of the divergence-free Biot--Savart field agrees with
$BR\cdot\nu_h$, and a constant vector field has zero flux.
Hence the right-hand side of \eqref{nonlinear_contour} has zero integral for every such $h$.
Since $\mathcal{L}_a [h]$ also has zero mean, so does $\mathcal{Q}(h)$.
Lemma \ref{lem:semigroup_est} therefore gives
\begin{equation*}
  \Big\|\frac{2R^3}{A_\sigma}\int_0^t
      e^{(t-\tau)\mathcal{L}_a}\mathcal{Q}(h)(\tau)\,d\tau\Big\|_{X_\gamma}
  \lesssim \|h\|_{X_\gamma}^2,
  \qquad h\in\mathcal{B}_{2C_*\varepsilon}.
\end{equation*}

Similarly, for $h_1,h_2\in\mathcal{B}_{2C_*\varepsilon}$, we have
\begin{equation*}
  \Xi(h_1)(t)-\Xi(h_2)(t)
=\frac{2R^3}{A_\sigma}\int_0^t e^{(t-\tau)\mathcal{L}_a}
\bigl(\mathcal{Q}(h_1)-\mathcal{Q}(h_2)\bigr)(\tau)\,d\tau.
\end{equation*}
By Lemma \ref{lem:nonlinear_difference}, we obtain
\begin{equation*}
  \frac{2R^3}{A_\sigma}
  \|\bigl(\mathcal{Q}(h_1)-\mathcal{Q}(h_2)\bigr)(s)\|_{C^\alpha}
  \lesssim \varepsilon s^{-\frac{2+\alpha}{3}}e^{-\gamma s}
      \|h_1-h_2\|_{X_\gamma},
\end{equation*}
for $s>0$, provided that $\varepsilon$ is sufficiently small.
The same argument as before then shows that
\begin{equation*}
	  \|\Xi(h_1)-\Xi(h_2)\|_{X_\gamma}
	\lesssim \varepsilon\|h_1-h_2\|_{X_\gamma}.
\end{equation*}
Therefore, for sufficiently small $\varepsilon$, the map $\Xi$
maps $\mathcal B_{2C_*\varepsilon}$ into itself and is a contraction
there. It follows that the map $\Xi$ \eqref{Tmap} admits a unique fixed point in
$\mathcal{B}_{2C_*\varepsilon}$ for
$\|f_0\|_{W^{1,\infty}}\leq\varepsilon$.

\end{proof}

\begin{proof}[Proof of Theorem \ref{thm:main}]
Proposition \ref{prop:fixed_point} gives a unique mild solution $h\in X_\gamma$, $\|h\|_{X_\gamma}\lesssim \varepsilon$.
The nonlinear estimate \eqref{Nonlinear_decay_bound} gives
\begin{equation*}
    \begin{aligned}
        \|\mathcal{Q}(h(t))\|_{L^2}\lesssim \|\mathcal{Q}(h(t))\|_{C^{\alpha}}\lesssim   \frac{A_\sigma}{2R^3}t^{-\frac{2+\alpha}{3}}\|h\|_{X_\gamma}^2,
    \end{aligned}
\end{equation*}
and by Duhamel's formulation we have
\begin{equation*}
\begin{aligned}
\|h(t)-e^{t\mathcal{L}_a}f_0\|_{L^2}
 &\leq \frac{2R^3}{A_\sigma}\int_0^t \|e^{(t-\tau)\mathcal{L}_a}\mathcal{Q}(h(\tau))\|_{L^2}d\tau\lesssim  t^{\frac{1-\alpha}{3}}\|h\|_{X_\gamma}^2.    
\end{aligned} 
\end{equation*}
Since the linear semigroup is strongly continuous on $L^2(\mathbb{S}^2)$, we have $$\|h(t)-f_0\|_{L^2}\leq \|h(t)-e^{t\mathcal{L}_a}f_0\|_{L^2}+\|e^{t\mathcal{L}_a}f_0-f_0\|_{L^2}\to 0,\quad\mbox{as}\quad t\to0^+.$$
Together with the uniform $W^{1,\infty}$ bound, this yields
strong convergence in $C^{\beta}$ for every $\beta\in(0,1)$
and weak-* convergence in $W^{1,\infty}$.

On compact positive-time intervals, the mild formulation gives
continuity into $L^2$, while the $X_\gamma$ bound gives boundedness
in $C^{3,\alpha}$. Interpolation therefore gives continuity into
$C^{3,\eta}$ for every $0<\eta<\alpha$. The nonlinear difference
estimate and the linear multiplier bounds then imply that
$\mathcal{L}_a h+(2R^3/A_\sigma)\mathcal{Q}(h)$ is continuous into
$C^\eta$. The evolution equation thus holds classically for $t>0$.

Finally, the actual solution is recovered by rescaling back the time variable,
$f(x,t)=h(x,A_\sigma t/(2R^3))$.\end{proof}

	\section*{Acknowledgements}
	
	\begingroup
\setlength{\emergencystretch}{2em}
FG and RMS gratefully acknowledge the hospitality of the Institute for Advanced Study (IAS) in Princeton, where they were members during the development of this work. EGJ and NP also thank the IAS for its hospitality during visits devoted to this project.

FG and EGJ were partially supported by the AEI through grant PID2022-140494NA-I00, and through grants RED2022-134784-T (Spain) and IMUS–María de Maeztu CEX2024-001517-M, funded by MICIU/AEI/10.13039/501100011033. FG was also partially supported by the AEI grant PID2024-158418NB-I00. EGJ was also partially supported by the RYC2021-032877 research grant. SVH was partially supported by NSF grant DMS-2511086. NP was partially supported by an AMS-Simons PUI grant. RMS was partially supported by NSF grant DMS-2408264.
	\par
    \textbf{AI usage disclosure:} LLMs, in the form of GPT-5.5, GPT-5.6 Sol and GPT-6 Astra, were used in the preparation of this manuscript as tools for writing, editing, and improving exposition, as well as for learning and applying techniques for clearer mathematical and scientific communication. They were not used to develop the core ideas, proof strategy, or principal results of this work, which predate the use of these LLMs in this project. The authors reviewed and take full responsibility for all content, arguments, and conclusions presented in the manuscript.
\endgroup

	\bibliographystyle{amsplain2linkLink}
	\bibliography{references.bib}{}

\providecommand{\bysame}{\leavevmode\hbox to3em{\hrulefill}\thinspace}
\providecommand{\href}[2]{#2}
\begin{thebibliography}{10}
\expandafter\ifx\csname arxiv\endcsname\relax
  \def\arxiv#1{\burlalt{https://arxiv.org/abs/#1}{https://arxiv.org/abs/#1}}\fi
\expandafter\ifx\csname doi\endcsname\relax
  \def\doi#1{\burlalt{https://doi.org/#1}{https://doi.org/#1}}\fi
\expandafter\ifx\csname href\endcsname\relax
  \def\href#1#2{#2}\fi
\expandafter\ifx\csname burlalt\endcsname\relax
  \def\burlalt#1#2{\href{#2}{#1}}\fi

\bibitem{AgrawalPatel2026}
Siddhant Agrawal and Neel Patel, \emph{Self-similar solutions to the {H}ele-{S}haw problem with surface tension}, 2026, \arxiv{2507.19443}.

\bibitem{AlazardKoch}
Thomas Alazard and Herbert Koch, \emph{The {H}ele-{S}haw semi-flow}, Preprint arXiv:2312.13678 (2023).

\bibitem{AlazardCritical2}
Thomas Alazard and Quoc-Hung Nguyen, \emph{On the {C}auchy problem for the {M}uskat equation. {II}: {C}ritical initial data}, Ann. PDE \textbf{7} (2021), no.~1, Paper No. 7, 25, \doi{10.1007/s40818-021-00099-x}.

\bibitem{Ambrose14}
David~M. Ambrose, \emph{The zero surface tension limit of two-dimensional interfacial {D}arcy flow}, J. Math. Fluid Mech. \textbf{16} (2014), no.~1, 105--143, \doi{10.1007/s00021-013-0146-1}.

\bibitem{ArfkenWeber2005}
George~B. Arfken and Hans~J. Weber, \emph{Mathematical methods for physicists}, 6 ed., Elsevier Academic Press, Burlington, Massachusetts, 2005.

\bibitem{BaldiJulinLaManna2026}
Pietro Baldi, Vesa Julin, and Domenico~Angelo La~Manna, \emph{Liquid drop with capillarity and rotating traveling waves}, Arch. Ration. Mech. Anal. \textbf{250} (2026), no.~1, Paper No. 4, 57, \doi{10.1007/s00205-025-02156-2}.

\bibitem{BocchiCastroGancedo2026}
Edoardo Bocchi, \'Angel Castro, and Francisco Gancedo, \emph{Global-in-time estimates for the 2{D} one-phase {M}uskat problem with contact points}, Comm. Math. Phys. \textbf{407} (2026), no.~6, Paper No. 121, 50, \doi{10.1007/s00220-026-05633-1}.

\bibitem{CastroFaracoMengual22}
\'{A}. Castro, D.~Faraco, and F.~Mengual, \emph{Localized mixing zone for {M}uskat bubbles and turned interfaces}, Ann. PDE \textbf{8} (2022), no.~1, Paper No. 7, 50, \doi{10.1007/s40818-022-00121-w}.

\bibitem{ChenHuNguyenFluids}
Ke~Chen, Ruilin Hu, and Quoc-Hung Nguyen, \emph{Well-posedness for local and nonlocal quasilinear evolution equations in fluids and geometry}, 2024, \arxiv{2407.05313}.

\bibitem{ChenHuNguyen2026}
\bysame, \emph{Schauder-type estimates and log-critical well-posedness for the two-phase {M}uskat problem with surface tension}, 2026, \arxiv{2606.12388}.

\bibitem{Chen93}
Xinfu Chen, \emph{The {H}ele-{S}haw problem and area-preserving curve-shortening motions}, Arch. Rational Mech. Anal. \textbf{123} (1993), no.~2, 117--151, \doi{10.1007/BF00695274}.

\bibitem{CP93}
Peter Constantin and Mary Pugh, \emph{Global solutions for small data to the {H}ele-{S}haw problem}, Nonlinearity \textbf{6} (1993), no.~3, 393--415, \doi{10.1088/0951-7715/6/3/004}.

\bibitem{CCG13}
Antonio C\'{o}rdoba, Diego C\'{o}rdoba, and Francisco Gancedo, \emph{Porous media: the {M}uskat problem in three dimensions}, Anal. PDE \textbf{6} (2013), no.~2, 447--497, \doi{10.2140/apde.2013.6.447}.

\bibitem{DaiXu13}
Feng Dai and Yuan Xu, \emph{Approximation theory and harmonic analysis on spheres and balls}, Springer Monographs in Mathematics, Springer, New York, 2013, \doi{10.1007/978-1-4614-6660-4}.

\bibitem{Darcy56}
Henry Darcy, \emph{Les {F}ontaines {P}ubliques de la {V}ille de {D}ijon}, Dalmont, Paris, 1856.

\bibitem{DongGancedoNguyen3D}
Hongjie Dong, Francisco Gancedo, and Huy~Q. Nguyen, \emph{Global well-posedness for the one-phase {M}uskat problem in 3d}, To appear in Arch. Ration. Mech. Anal., preprint arXiv:2308.14230 (2023).

\bibitem{DongKwon2026}
Hongjie Dong and Hyunwoo Kwon, \emph{Global well-posedness of the one-phase {M}uskat problem with surface tension}, 2026, \arxiv{2604.06545}.

\bibitem{DuchonRobert1984}
Jean Duchon and Raoul Robert, \emph{\'{E}volution d'une interface par capillarit\'{e} et diffusion de volume. {I}. {E}xistence locale en temps}, Ann. Inst. H. Poincar\'{e} Anal. Non Lin\'{e}aire \textbf{1} (1984), no.~5, 361--378.

\bibitem{EscherMatioc2011}
Joachim Escher and Bogdan-Vasile Matioc, \emph{On the parabolicity of the {M}uskat problem: well-posedness, fingering, and stability results}, Z. Anal. Anwend. \textbf{30} (2011), no.~2, 193--218, \doi{10.4171/ZAA/1431}.

\bibitem{ES97}
Joachim Escher and Gieri Simonett, \emph{Classical solutions for {H}ele-{S}haw models with surface tension}, Adv. Differential Equations \textbf{2} (1997), no.~4, 619--642.

\bibitem{FlynnNguyen2021}
Patrick~T. Flynn and Huy~Q. Nguyen, \emph{The vanishing surface tension limit of the {M}uskat problem}, Comm. Math. Phys. \textbf{382} (2021), no.~2, 1205--1241, \doi{10.1007/s00220-021-03980-9}.

\bibitem{GancedoSEMA}
Francisco Gancedo, \emph{A survey for the {M}uskat problem and a new estimate}, SeMA J. \textbf{74} (2017), no.~1, 21--35, \doi{10.1007/s40324-016-0078-9}.

\bibitem{GancedoGarciaJuarezPatelStrain23}
Francisco Gancedo, Eduardo Garc\'{\i}a-Ju\'{a}rez, Neel Patel, and Robert Strain, \emph{Global {R}egularity for {G}ravity {U}nstable {M}uskat {B}ubbles}, Mem. Amer. Math. Soc. \textbf{292} (2023), no.~1455, \doi{10.1090/memo/1455}.

\bibitem{GGPSCriticalBubble}
Francisco Gancedo, Eduardo Garc\'ia-Ju\'arez, Neel Patel, and Robert~M. Strain, \emph{On nonlinear stability of {M}uskat bubbles}, J. Math. Pures Appl. (9) \textbf{194} (2025), Paper No. 103664, 31, \doi{10.1016/j.matpur.2025.103664}.

\bibitem{GancedoGraneroScro2020}
Francisco Gancedo, Rafael Granero-Belinch\'on, and Stefano Scrobogna, \emph{Surface tension stabilization of the {R}ayleigh-{T}aylor instability for a fluid layer in a porous medium}, Ann. Inst. H. Poincar\'e{} C Anal. Non Lin\'eaire \textbf{37} (2020), no.~6, 1299--1343, \doi{10.1016/j.anihpc.2020.04.005}.

\bibitem{GHHaziotGomezSerranoPausader23}
Eduardo Garc\'ia-Ju\'arez, Javier G\'omez-Serrano, Susanna~V. Haziot, and Beno\^it Pausader, \emph{Desingularization of {S}mall {M}oving {C}orners for the {M}uskat {E}quation}, Ann. PDE \textbf{10} (2024), no.~2, Paper No. 17, \doi{10.1007/s40818-024-00175-y}.

\bibitem{GarciaHaziotKuoMoriZhou2026}
Eduardo García-Juárez, Susanna~V. Haziot, Po-Chun Kuo, Yoichiro Mori, and Han Zhou, \emph{On the global asymptotic stability for the 3{D} {P}eskin problem at critical regularity}, 2026, \arxiv{2607.11731}.

\bibitem{GigaGu2026}
Yoshikazu Giga and Zhongyang Gu, \emph{On existence of a collapsed bubble with surface tension in viscous incompressible fluid}, 2026, \arxiv{2606.31266}.

\bibitem{GranerLazar}
Rafael Granero-Belinch\'on and Omar Lazar, \emph{Growth in the {M}uskat problem}, Math. Model. Nat. Phenom. \textbf{15} (2020), Paper No. 7, 23, \doi{10.1051/mmnp/2019021}.

\bibitem{GuoHallstromSpirn07}
Yan Guo, Chris Hallstrom, and Daniel Spirn, \emph{Dynamics near unstable, interfacial fluids}, Comm. Math. Phys. \textbf{270} (2007), no.~3, 635--689, \doi{10.1007/s00220-006-0164-4}.

\bibitem{ChihWeinstein2023}
Chen-Chih Lai and Michael~I. Weinstein, \emph{Free boundary problem for a gas bubble in a liquid, and exponential stability of the manifold of spherically symmetric equilibria}, Arch. Ration. Mech. Anal. \textbf{247} (2023), no.~5, Paper No. 100, 87, \doi{10.1007/s00205-023-01927-z}.

\bibitem{Lazar2024}
Omar Lazar, \emph{Global well-posedness of arbitrarily large lipschitz solutions for the {M}uskat problem with surface tension}, Preprint arXiv:2407.09444 (2024).

\bibitem{MatiocMatioc23}
Anca-Voichita Matioc and Bogdan-Vasile Matioc, \emph{A new reformulation of the {M}uskat problem with surface tension}, J. Differential Equations \textbf{350} (2023), 308--335, \doi{10.1016/j.jde.2023.01.003}.

\bibitem{MeyerNiebelSeis2025}
David Meyer, Lukas Niebel, and Christian Seis, \emph{Steady bubbles and drops in inviscid fluids}, Calc. Var. Partial Differential Equations \textbf{64} (2025), no.~9, Paper No. 299, 30, \doi{10.1007/s00526-025-03144-w}.

\bibitem{Muskat34}
Morris Muskat, \emph{Two fluid systems in porous media. the encroachment of water into an oil sand}, J. Appl. Phys. \textbf{5} (1934), no.~9, 250--264, \doi{10.1063/1.1745259}.

\bibitem{Nguyen20s}
Huy~Q. Nguyen, \emph{On well-posedness of the {M}uskat problem with surface tension}, Adv. Math. \textbf{374} (2020), 107344, 35, \doi{10.1016/j.aim.2020.107344}.

\bibitem{PrussSimonett2016}
Jan Pr\"{u}ss and Gieri Simonett, \emph{Moving interfaces and quasilinear parabolic evolution equations}, Monographs in Mathematics, vol. 105, Birkh\"{a}user/Springer, 2016, \doi{10.1007/978-3-319-27698-4}.

\bibitem{Shao2026}
Chengyang Shao, \emph{On the {C}auchy problem of spherical capillary water waves}, Forum Math. \textbf{38} (2026), no.~3, 727--769, \doi{10.1515/forum-2024-0042}.

\bibitem{Tartar89}
Luc Tartar, \emph{Incompressible fluid flow in a porous medium - convergence of the homogenization process}, appendix to the book ``{N}onhomogeneous media and vibration theory'' by {E}nrique {S}{\'{a}}nchez-{P}alencia ed., Lecture Notes in Physics, vol. 127, pp.~368--377, Springer-Verlag, Berlin-New York, 1980, \doi{10.1007/978-1-4612-1920-0}.

\bibitem{Taylor91}
Michael~E. Taylor, \emph{Pseudodifferential operators and nonlinear {PDE}}, Progress in Mathematics, vol. 100, Birkh\"{a}user Boston, Inc., Boston, MA, 1991, \doi{10.1007/978-1-4612-0431-2}.

\bibitem{YT11}
J.~Ye and S.~Tanveer, \emph{Global existence for a translating near-circular {H}ele-{S}haw bubble with surface tension}, SIAM J. Math. Anal. \textbf{43} (2011), no.~1, 457--506, \doi{10.1137/100786332}.

\bibitem{YT12}
\bysame, \emph{Global solutions for a two-phase {H}ele-{S}haw bubble for a near-circular initial shape}, Complex Var. Elliptic Equ. \textbf{57} (2012), no.~1, 23--61, \doi{10.1080/17476933.2010.504835}.

\end{thebibliography}

\end{document}